\documentclass[a4paper,12pt]{amsart}
\usepackage[margin=1.65in,footskip=0.50in]{geometry}
\usepackage{amsfonts}
\usepackage{mathrsfs}
\usepackage{amsmath,amssymb,latexsym,amsfonts,amscd}
\usepackage[usenames,dvipsnames]{color}
\usepackage{multicol}
\usepackage{verbatim}
\usepackage[normalem]{ulem}
\usepackage[all]{xy}
\usepackage{cancel}
\usepackage{xcolor}
\usepackage{bm}

\renewcommand{\mathbf}{\mathbb}

\newtheorem{thm}{Theorem}[section]
\newtheorem{theorem}[thm]{Theorem}
\newtheorem{lemma}[thm]{Lemma}

\newtheorem{prop}[thm]{Proposition}
\newtheorem{definition}[thm]{Definition}

\newtheorem{corollary}[thm]{Corollary}

\theoremstyle{definition}

\newtheorem{setup}[thm]{}

\newtheorem{remark}[thm]{Remark}

\newtheorem{notation}[thm]{Notation}
\newtheorem{convention}[thm]{Convention}
\theoremstyle{remark}

\numberwithin{equation}{thm}

\title[Abelian canonical covers up to degree $8$]{Abelian canonical covers of minimal rational surfaces and $\mathbb F_1$ up to degree $8$}
\author{Alexandre Dorothée}
\address{Departamento de Álgebra, Geometría y Topología, Universidad Complutense de Madrid, Madrid, Spain}
\email{doroalex@ucm.es}
\author{Francisco Javier Gallego}
\address{Departamento de Álgebra, Geometría y Topología and Instituto de Matemática Interdisciplinar, Universidad Complutense de Madrid, Madrid, Spain}
\email{gallego@ucm.es}
\date{}

\begin{document}

\begin{abstract}

We classify abelian canonical covers of degree at most $8$ of smooth surfaces of minimal degree, and of $\mathbb P^2$ and Hirzebruch surfaces not embedded  as surfaces of minimal degree. Such covers are the surface analogue of the canonical map of a hyperelliptic curve and play a role in the geography of surfaces of general type and in the study of the degree of the canonical map. By a result of Xiao, degree $8$ is a natural threshold, since for higher degree the geometric genus and the irregularity are bounded.
\smallskip

Let $\pi: X \longrightarrow W$ be an abelian canonical cover of a smooth surface of minimal degree $W$ with group $G$ and degree $n$.
 Degrees $n \le 4$ were classified earlier; we complete the picture for $5 \le n \le 8$:

\smallskip
\begin{enumerate}
    \item No such covers exist if $n$ is prime and $n > 3$ (in particular, if $n = 5, 7$), nor if $G = \mathbb Z_8$. Hence only $G = \mathbb Z_6$, $\mathbb Z_2 \times  \mathbb Z_2 \times \mathbb Z_2$ and $\mathbb Z_2 \times \mathbb Z_4$ occur.

    \item For each of these groups we give the complete list of covers, specifying its building data and the singularities of $X$. Only $W =    \mathbb P^2$, $\mathbb F_0$ and $\mathbb F_1$ arise. The list contains several families with unbounded $p_g$, most of them new; $X$ is smooth for general covers with $G=\mathbb Z_2 \times \mathbb Z_2 \times \mathbb Z_2$, but singular if $G=\mathbb Z_6$ and, with one exception, if $G=\mathbb Z_2 \times \mathbb Z_4$.
\end{enumerate}

\smallskip

When $W$ is $\mathbb P^2$ or a Hirzebruch surface not embedded as a surface of minimal degree, we show that abelian canonical covers exist only for $n = 2$ or for $n = 4$ with $G = \mathbb Z_2 \times \mathbb Z_2$, and we classify the latter completely: they form infinitely many families, most new, with unbounded $p_g$ and, most of them, with unbounded irregularity.
\smallskip

The slopes $K^2/\chi$ of $X$ in these families accumulate at $6$ (if $n=6$), at 8 (if $n=8$), and, for the quadruple covers above, at $8\alpha/(\alpha +1)$ for every integer $\alpha \ge 2$, as well as at $8$.
\smallskip

These results rest on more general ones, valid for $X$ normal, locally Gorenstein and of arbitrary dimension, and $W$ smooth with $p_g(W) = 0$: we determine which branch divisors of an abelian canonical cover must vanish when $n \le  8$ or $n$ is prime, and, for canonical covers whose $\pi_*\mathcal O_X$ splits as a sum of line bundles, we describe the $\mathcal O_W$-module structure of $\pi_*\mathcal O_X$, computing it completely when $W$ is  $\mathbb P^2$ or a Hirzebruch surface. 
\end{abstract}
\maketitle

\section*{Introduction}

The goal of this article is to give the complete classification of abelian canonical covers, up to degree $8$, of smooth surfaces of minimal degree, and also of $\mathbb P^2$ or Hirzebruch surfaces not embedded as surfaces of minimal degree. Abelian canonical
covers $\pi: X \longrightarrow W$ (i.e., abelian covers $\pi$ which are a factor of a canonical morphism $\varphi: X \longrightarrow \mathbb P^N$) of surfaces $W$ embedded in $\mathbb P^N$ as surfaces of minimal degree $W$ are central, from various points of view,  in the theory of complex projective surfaces. They can be seen as analogues of canonical morphisms of hyperelliptic curves.  They are important in the geography  and in relation to the degree of the canonical map of surfaces of general type, especially as this degree goes up. They can be used to find moduli components with loci analogous to the hyperelliptic locus of the moduli of curves.  They are also exceptional with respect to their canonical ring.

\smallskip

\noindent
{\bf Relevance of abelian canonical covers of surfaces of minimal degree.}

\vskip .05truecm
\noindent Varieties of minimal degree are the generalization of rational normal curves. Therefore, since double covers are necessarily abelian, abelian canonical covers of surfaces of minimal degree can be considered the anologue of canonical morphisms of hyperelliptic curves. However, unlike curves,  if the canonical map $\varphi$ of a smooth surface of general type is generically finite, then the degree $n$ of $\varphi$ can be as high as $36$ (see \cite[Proposition 5.7]{Persson} for the bound; see  
\cite[Theorem 1]{Yeung}, \cite[Theorem 1.4]{LaiYeung}
 and \cite[Theorem 1]{Rito} for examples that attain the bound). 
In this context, abelian canonical covers of surfaces of minimal degree become relevant in
 filling the geography of surfaces of general type with given invariants and possessing a canonical map of a specific degree
 (see e.g. \cite{Horikawa}, \cite{Persson}, \cite{Konno}, \cite{quad2}, \cite{quad1},  \cite{DuGao}
 ).
 The problem of finding out what degrees $n$ in the range $2 \le n \le 36$ actually occur is, at the same time, classical and very active until the present day, as the  contributions of many algebraic geometers 
 show.

\smallskip
 Another reason why canonical covers of surfaces of minimal degree are significant is that, as $n$ goes up, the list of possible images of $\varphi$ is more and more limited. Beauville proved  (see \cite[Theorem 3.1]{Beauville}) proved that if $\varphi$ is generically finite, then the image $W$ of $\varphi$ has $p_g(W)=0$ or $W$ is a surface of general type. Surfaces of minimal degree indeed have $p_g=0$. Moreover, if $n >4$, then $W$ is birationally equivalent to a (geometrically) ruled surface (see \cite[Proposition 4.1]{Beauville}). Since smooth surfaces of minimal degree (see e.g. \cite{EH}) are either $\mathbb P^2$ or rational (geometrically) ruled surfaces (the so-called Hirzebruch surfaces), they are still among the list of plausible candidates for $W$ as $n$ increases.

\smallskip
 Abelian canonical covers of surfaces minimal degree also play an interesting role in finding moduli components of surfaces of general type with ``jumping" loci (analogous to the hyperelliptic locus of the moduli space of curves of genus $g > 2$), where the degree of the canonical map jumps up, like in  \cite{BGMR}. The abelian covers of degree in Subsection~\ref{section.bidouble.rational} and onwards may be  suitable candidates to produce further components possessing such a jumping locus. 

\smallskip

Finally,  canonical covers of surfaces of minimal degree are remarkable from the point of view of its canonical ring. Indeed, if $X$ is regular and if we exclude the case $n >2$ when $W=\mathbb P^2$, then the canonical ring of $X$ is not generated in degree $1$ and $2$ (see \cite[Theorem 2.1]{canonical.ring}). This is truly exceptional, because, otherwise, the canonical ring  of a regular surface of general type with base-point-free canonical divisor is generated in degree $1$ and $2$ (see  \cite{Green} and \cite{Ciliberto}). 

\smallskip

\noindent
{\bf Results on abelian canonical covers of smooth surfaces of minimal degree and of $\mathbb P^2$ and Hirzebruch surfaces.}

\vskip .05truecm
\noindent By all the above reasons, it is worth having  a complete classification of abelian canonical covers of surfaces of minimal degree. In this article we obtain such a classification when the degree $n \le 8$ and $W$ is smooth. 
{The classification, by the first author, when $W$ is singular will appear elsewhere (see \cite{singular})}. 
The bound $n \le 8$ is important, because if $n > 8$, then $p_g(X)$ (and $q(X)$) is bounded, 
(see \cite[Theorem 3]{Xiao}),
so the families of abelian canonical covers with $n > 8$ are bounded. 

\smallskip
Canonical double and triple covers of surfaces of minimal degree were well-known, since the works of Horikawa, Persson, Konno, Mendes Lopes and Pardini (see \cite{Horikawa}, \cite{Horikawa.triple}, \cite{Persson}, \cite{Konno}, \cite{MLP}).  Degree $4$ abelian canonical covers of surfaces of minimal degree were completely classified in \cite{quad2}, \cite{quad1}. Thus, in this paper we go on from $n = 5$ until $n=8$.

\smallskip
The results for abelian canonical covers of degree $5 \le n \le 8$ of smooth surfaces of minimal degree are collected in Section~\ref{section.abelian.rational} (see Theorems~\ref{cor.prime.P^2}, \ref{cor.prime.ruled}, \ref{thm.P2.Z6}, \ref{thm.Hirzebruch.Z6.regular}, \ref{thm.Hirzebruch.Z6.irregular}, \ref{thm.tridouble.P2}, \ref{thm.scroll.regular}, \ref{thm.Hirzebruch.irregular}, \ref{thm.P2.Z2Z4}, \ref{thm.Hirzebruch.Z2Z4.regular}, \ref{thm.Hirzebruch.Z2Z4.irregular}, \ref{thm.noZ8.covers}).
We prove that covers of degrees $n=5, 7$ do not exist (in fact, we prove the non existence of 
covers 
of prime degree $p$ if $p > 3$), nor do  cyclic covers of degree $8$ exist. For the remaining possibilities ((cyclic of) degree $6$, $\mathbb Z_2 \times \mathbb Z_2 \times \mathbb Z_2$ and $\mathbb Z_2 \times \mathbb Z_4$) we give a complete list of the covers $\pi: X \longrightarrow W$ that occur, specifying the building data (i.e., the line bundles of $\pi_*\mathcal O_X$ and the divisors of the branch locus of $\pi$) and the multiplicative structure of $\pi_*\mathcal O_X$ (as the name ``building data" suggests, this allows anyone to build the canonical cover from the bottom $W$ up to $X$). 

\smallskip
Among the covers that appear in our classification, we recover the covers of $\mathbb P^2$ constructed by Du and Gao in \cite[Theorem 1.1]{DuGao} for degrees $n=6, 8$ (unlike there, here we give explicitly the group of the cover and the description of its building data, as explained before, as well as  the singularities of $X$). We also recover the family constructed by Nguyen Bin in \cite[3.1.1]{Bin} (this is one of the four families in our list of covers with group $\mathbb Z_2 \times \mathbb Z_2 \times \mathbb Z_2$, precisely the family of Theorem~\ref{thm.scroll.regular} when $W=\mathbb F_1$, whose invariants are in Row 10 of Table~\ref{table1}). 
Finally we recover as well two constructions by Beauville in the case  in which his covers are abelian (see \cite[Exemples 4.3, 4.4]{Beauville}, that would correspond to
Theorem~\ref{thm.Hirzebruch.Z6.irregular}, Case 1 of Theorem~\ref{thm.Hirzebruch.irregular} and  Case 1 of Theorem~\ref{thm.Hirzebruch.Z2Z4.irregular}; their invariants are in Rows 4, 9, 14 of Table~\ref{table1}). This is an instance in which  the exhaustivity of our list reveals  its importance (and, indeed, the importance of Beauville's constructions), as it allows us to assure that they are, in the abelian cover case, the only irregular canonical covers of degree $6$, and the only covers of degree $8$ with $q(X)=3$, of smooth surfaces of minimal degree.

\smallskip Our classification of canonical covers $\pi: X \longrightarrow W$ of surfaces of minimal degree tells that the only possibilities for $W$ are $\mathbb P^2$, $\mathbf F_0$ and $\mathbb F_1$. In addition, if the group is $\mathbb Z_2\times \mathbb Z_2 \times \mathbb Z_2$, then, for a general cover $\pi$, $X$ is smooth (see Theorems~\ref{thm.tridouble.P2}, \ref{thm.scroll.regular} and \ref{thm.Hirzebruch.irregular}). By contrast, if the group is $\mathbf Z_6$ the surface $X$ is singular (see Theorems~\ref{thm.P2.Z6}, \ref{thm.Hirzebruch.Z6.regular}, \ref{thm.Hirzebruch.Z6.irregular} or the last column of Table~\ref{table1} for details). The same happens if the group $\mathbb Z_2 \times \mathbb Z_4$, except for one of families (see Theorems~\ref{thm.P2.Z2Z4}, \ref{thm.Hirzebruch.Z2Z4.regular}, \ref{thm.Hirzebruch.Z2Z4.irregular} or the last column of Table~\ref{table1} for details). 

\smallskip In Section~\ref{section.abelian.rational} we do not only look at covers of smooth surfaces of minimal degree. As already mentioned, a smooth surface of minimal degree is abstractly isomorphic to either $\mathbb P^2$ or a Hirzebruch surface. Then, it makes sense to also study canonical covers $\pi: X \longrightarrow W$ when $W$ is  $\mathbb P^2$ or a Hirzebruch surface embedded by a complete linear series but not as a surface of minimal degree. 
We prove that, aside from $n=2$,   $W$ not being  of minimal degree only happens 
if $n=4$ and the group of the cover is $\mathbb Z_2 \times \mathbb Z_2$ (see Proposition~\ref{prop.split.minimal.degree} (1), 
Theorem~\ref{thm.Z4.non.scroll.regular} and Subsection~\ref{section.degree.6.rational} onward). 
This a distinguished feature that makes degree $4$ remarkable. Another distinguished feature is that quadruple canonical covers
$\pi: X \longrightarrow W$ are the only canonical covers of surfaces of minimal degree and, indeed, of Hirzebruch surfaces, with unbounded irregularity $q(X)$, as \cite[Theorem 0.1]{quad1} and Theorem~\ref{thm.Hirzebruch.Z2Z2.irregular}, together with \cite[Theorem 5]{Xiao}, show. 
In Theorems~\ref{thm.bidouble.nonscroll.regular}, \ref{thm.Hirzebruch.Z2Z2.irregular} we completely classify these canonical covers $\pi: X \longrightarrow W$ with $W$ a Hirzebruch surface group $\mathbb Z_2 \times \mathbb Z_2$. 
As it turns out, $W$ can only be $\mathbb F_0$ and $X$ is smooth if $\pi$ is general.

\smallskip
\noindent {\bf Invariants and geography.}

\vskip 0.05truecm

\noindent We now pay attention to the invariants of the above surfaces $X$ and to how they are located  in  the geography of surfaces of general type. 
These invariants are collected in Tables~\ref{table1} and \ref{table2}. 
Degree $8$ covers (see Subsections \ref{section.tridouble.rational}, \ref{section.Z2Z4.rational} and Rows 6 to 10 of Table \ref{table1} for details) yield eight families of surfaces $X$ with unbounded geometric genus and unbounded $c_1^2$. For these families the slope $K^2/\chi$  approaches $8$ 
as $p_g$ goes to infinity and $K^2/\chi=8$  for two of these  families. The irregularity $q(X)$ is $0$, $1$ or $3$.
\smallskip

Degree $6$ covers (see Subsection \ref{section.degree.6.rational} and Rows 1 to 5 of Table \ref{table1} for details) yield three  families of surfaces $X$ 
with unbounded geometric genus, unbounded $c_1^2$ and whose slope $K^2/\chi$ approaches $6$  as $p_g$ goes to infinity. The irregularity $q(X)$ is $0$ or $2$. 

\smallskip 
Quadruple abelian canonical covers of $\mathbb P^2$ and Hirzebruch surfaces not embedded as surfaces of minimal degree (see Subsection~\ref{section.bidouble.rational} and Table~\ref{table2} for details)
yield 
infinitely many families 
of surfaces $X$ with unbounded geometric genus,  unbounded $c_1^2$ and, as already mentioned, unbounded irregularity. For any integer $\alpha$, $\alpha \ge 2$, the number $\dfrac{8\alpha}{\alpha+1}$ is an accumulation point for the slopes $K^2/\chi$ of the surfaces $X$ of these families. Obviously, $8$ is also an accumulation point, since it is the limit of $K^2/\chi$ as both $p_g$ and $\alpha$ go to infinity; not only that, but are there are also infinitely many families with $K^2/\chi=8$. In sum, the limits of the slopes $K^2/\chi$ assume infinite, countably many distinct values in the range $[16/3,8]$ ($[4,8]$ if we include also the abelian canonical covers of \cite[Theorem 0.1 (II)]{quad1}). The set of these accumulation values strictly contains the set obtained  by Falluca and Pignatelli (see 
\cite[Theorem 3.1]{FallucaPignatelli}), since, for us, $\alpha +1$ is any integer bigger than $2$ (in particular,  $\alpha +1$ can be prime). Unlike theirs (except when the surface $X$ is a product of hyperelliptic curves; see Case 2 of Theorem~\ref{thm.Hirzebruch.Z2Z2.irregular} (1)), our surfaces are smooth when the cover is general, and the canonical map is a finite morphism.

\medskip

\begin{table}[!ht]\caption{Invariants of abelian canonical covers of smooth surfaces of minimal degree, $5 \le n \le 8$. }\label{table1}
\centerline{\vbox{\tabskip=0pt \offinterlineskip
\def\tablerule{\noalign{\hrule}}
\halign to 16truecm
{\strut
#& \vrule
\tabskip=0em plus 3em
\hfil
#
\hfil  
&
\hfil#\hfil 
& \vrule#&
\hfil#\hfil 
& \vrule#&
\hfil#\hfil 
& \vrule#&
\hfil#\hfil 
& \vrule#&
\hfil#\hfil 
& \vrule#&
\hfil#\hfil 
& \vrule#
&
\hfil#\hfil 
& \vrule#
&
\hfil#\hfil 
& \vrule#&
\hfil#\hfil 
& \vrule#&
\hfil#\hfil 
& \vrule#\tabskip=0pt
\cr\tablerule
&&
$n$ 
&&
$G$ 
&&
 $W$  
&&
  $W$ embedded by
  && $c_1^2$ && $p_g$ && $q$ &&$c_1^2/c_2$ &&$K^2/\chi$ &&  Sing$X$
&\cr\tablerule
\tablerule
\tablerule
&& 6
&& $\mathbb Z_6$
&& $\mathbb P^2$
&& $|\mathcal O_{\mathbb P^2}(1)|$
&& $6$
&& $3$
&& $0$
&& $1/7$ 
&& $3/2$
&&$4A_1$, $4A_2$
&\cr
\tablerule
&& $6$
&& $\mathbb Z_6$
&& $\mathbb F_0$
&& $|C_0+f|$
&& $12$
&& $4$
&& $0$
&& $1/4$
&& $12/5$
&& $6A_1$, $9A_2$
&\cr
\tablerule
 && 6
&& $\mathbb Z_6$
&& $\mathbb F_0$
&& $|C_0+mf|$, $m \ge 1$
&& $12m$
&& $2m+2$
&& $0$
&& \small $\dfrac{m}{m+3}$
&&\small  $\dfrac{12m}{2m+3}$
&&$(4m+2) A_1$, $4A_2$
&  \cr
\tablerule
&& $6$
&& $\mathbb Z_6$
&& $\mathbb F_0$
&& $|C_0+mf|$, $m \ge 1$
&& $12m$
&& $2m+2$
&& $2$
&& \small $\dfrac{m}{m+1}$
&&\small $\dfrac{12m}{2m+1}$
&&$(4m+8)A_1$$^*$
&\cr
\tablerule
&& $6$
&& $\mathbb Z_6$
&& $\mathbb F_1$
&& $|C_0+mf|$, $m \ge 2$
&& $12m-6$
&& $2m+1$
&& $0$
&&\small $\dfrac{2m-1}{2m+5}$
&&\small $\dfrac{6m-3}{m+1}$
&&$4mA_1$, $4A_2$
&\cr
\tablerule
\tablerule
\tablerule
&& $8$
&& $\mathbb Z_2^3$
&& $\mathbb P^2$
&& $|\mathcal O_{\mathbb P^2}(1)|$
&& $8$
&& $3$
&& $0$
&&$1/5$
&&$2$
&& smooth
&\cr
\tablerule
&& $8$
&& $\mathbb Z_2^3$
&& $\mathbb F_0$
&& $|C_0+mf|$, $m \ge 1$
&& $16m$
&& $2m+2$
&& $0$
&&\small $\dfrac{4m}{2m+9}$
&&\small $\dfrac{16m}{2m+3}$
&& smooth
&\cr
\tablerule
&& $8$
&& $\mathbb Z_2^3$
&& $\mathbb F_0$
&& $|C_0+mf|$, $m \ge 1$
&& $16m$
&& $2m+2$
&& $1$
&&\small $\dfrac{2m}{m+3}$
&&\small $\dfrac{8m}{m+1}$
&& smooth 
&\cr
\tablerule
&& $8$
&& $\mathbb Z_2^3$
&& $\mathbb F_0$
&& $|C_0+mf|$, $m \ge 1$
&& $16m$
&& $2m+2$
&& $3$
&&$2$
&&$8$
&& \ \ smooth $^*$
&\cr
\tablerule
&& $8$
&& $\mathbb Z_2^3$
&& $\mathbb F_1$
&& $|C_0+mf|$, $m \ge 2$
&& $16m-8$
&& $2m+1$
&& $0$
&&\small $\dfrac{2m-1}{m+4}$
&&\small $\dfrac{8m-4}{m+1}$
&& smooth
&\cr
\tablerule
\tablerule
\tablerule
&& $8$
&& $\mathbb Z_2 \times \mathbb Z_4$
&& $\mathbb P^2$
&& $|\mathcal O_{\mathbb P^2}(1)|$
&& $8$
&& $3$
&& $0$
&&$1/5$
&&$2$
&& $4A_1$
&\cr
\tablerule
&& $8$
&& $\mathbb Z_2 \times \mathbb Z_4$
&& $\mathbb F_0$
&& $|C_0+f|$
&& $16$
&& $4$
&& $1$
&&$1/2$
&&$4$
&&\ $16A_1$ $^*$
&\cr
\tablerule
&& $8$
&& $\mathbb Z_2 \times \mathbb Z_4$
&& $\mathbb F_0$
&& $|C_0+mf|$, $m \ge 1$
&& $16m$
&& $2m+2$
&& $0$
&&\small $\dfrac{4m}{2m+9}$
&&\small $\dfrac{16m}{2m+3}$
&& $4A_1$
&\cr
\tablerule
&& $8$
&& $\mathbb Z_2 \times \mathbb Z_4$
&& $\mathbb F_0$
&& $|C_0+mf|$, $m \ge 1$
&& $16m$
&& $2m+2$
&& $1$
&&\small $\dfrac{2m}{m+3}$
&&\small $\dfrac{8m}{m+1}$
&& \ $16A_1$ $^*$
&\cr
\tablerule
&& $8$
&& $\mathbb Z_2 \times \mathbb Z_4$
&& $\mathbb F_0$
&& $|C_0+mf|$, $m \ge 1$
&& $16m$
&& $2m+2$
&& $3$
&&$2$
&&$8$
&&\ \ smooth $^*$
&\cr
\tablerule
&& $8$
&& $\mathbb Z_2 \times \mathbb Z_4$
&& $\mathbb F_1$
&& $|C_0+mf|$, $m \ge 2$
&& $16m-8$
&& $2m+1$
&& $0$
&&\small $\dfrac{2m-1}{m+4}$
&&\small $\dfrac{8m-4}{m+1}$
&& $4A_1$ 
&\cr
\tablerule
\noalign{\smallskip}
}}}
 \vspace{0.1ex}

     {\raggedright {\small Invariants of columns 5 to 9 refer to $X$. Last column describes the singularities of $X$ when the cover $\pi$ is general. The as\-ter\-isk means that, in this case, these are the singularities of $X$ for all covers.}  \par}
\end{table}

\medskip

\begin{table}[!ht]\caption{Invariants of quadruple abelian canonical covers of Hirzebruch surfaces not embedded as surfaces of minimal degree.}\label{table2}
\renewcommand*{\arraystretch}{1.3}
\begin{tabular}{|c|c|c|c|}
\hline
$q$ &$c_1^2/c_2$ &$K^2/\chi$ &  Sing$X$ \\ \hline \hline 
 \rule{0pt}{3.75ex} $0$
& $\dfrac{2\alpha m}{\alpha m+3\alpha+3m+6}$ 
& $\dfrac{8\alpha m}{\alpha m+\alpha+m+2}$
&smooth \\ [8pt] 
\hline
\rule{0pt}{3.75ex} $\alpha$
& $\dfrac{2\alpha m}{\alpha m+3m+6}$ 
& $\dfrac{8\alpha m}{\alpha m+m+2}$
&smooth \\ [8pt] 
\hline
\rule{0pt}{3.75ex} $m$
& $\dfrac{2\alpha m}{\alpha m+3\alpha +6}$ 
& $\dfrac{8\alpha m}{\alpha m+\alpha +2}$
&smooth \\ [8pt] 
\hline
$\alpha + m +2$
& $2$ 
& $8$
&smooth$^*$ \\ 
\hline
\end{tabular}
\vspace{0.3cm}

     {\raggedright \small{For each row of this table, $n=4$, $G=\mathbb Z_2 \times \mathbb Z_2$,
     $W$ is embedded by $|\alpha C_0+mf|$ ($\alpha,m \ge 2$), 
 $c_1^2(X)=8\alpha m$
and $p_g(X)=(\alpha +1)(m+1)$. The meaning of the last column is the same as the meaning of the last column of Table~\ref{table1}.} \par}
\end{table}


\noindent {\bf More general results.}

\vskip 0.05truecm

\noindent Our results on abelian covers of $\mathbb P^2$ and Hirzebruch surfaces are based on the more general results we prove in Section~\ref{section.abelian.canonical}. These are in turn based on the even more general results of Section~\ref{section.split}. 
Besides their application to Section~\ref{section.abelian.rational}, the results of Sections~\ref{section.split} and \ref{section.abelian.canonical} are of independent interest and have applications to many kinds of varieties, 
{e.g., to give the classification of abelian canonical covers of elliptic ruled surfaces (see~\cite{elliptic.ruled}).}  In Section~\ref{section.abelian.canonical} we study abelian canonical covers $\pi: X \longrightarrow W$, with degree $n$ prime (see Theorem~\ref{thm.prime}, where we prove that $n$ prime implies $\pi$ is simple cyclic) and with $n \leq 8$ (see Theorems~\ref{thm.bidouble.alg.str}, \ref{thm.Z4.alg.str}, \ref{thm.Z6.alg.str}, \ref{thm.tridouble.alg.str}, \ref{thm.Z2Z4.alg.str}, \ref{thm.Z8.alg.str}), where $X$ and $W$ are projective varieties with $X$ normal, irreducible and locally Gorenstein and $W$ smooth with $p_g(W)=0$ (see Set-up~\ref{setup.beginning} and Definition~\ref{defi.canonical.cover}) and, occasionally, some extra condition on the torsion of Pic$W$.
In particular, the dimension of $X$ and $W$ is arbitrary.
The main point in  these results is that an abelian cover being canonical forces several of the divisors of its total branch locus to be trivial. This is the reason behind the fact that the singularities of the canonical covers of smooth surfaces of minimal degree with groups $\mathbb Z_6$ and $\mathbb Z_2 \times \mathbb Z_4$ are milder than a priori expected.  

\smallskip
In Section~\ref{section.split} we study (see Theorem~\ref{thm.splitting}) the structure of $\mathcal O_W$-module structure of $\pi_*\mathcal O_X$ for a canonical cover $\pi: X \longrightarrow W$, when $\pi_*\mathcal O_X$ splits a direct sum of line bundles but $\pi$ is not necessarily abelian, under very general conditions (see Set-up~\ref{setup.beginning} and Definition~\ref{defi.canonical.cover} (1); in particular, $X$ and $W$ have arbitrary dimension). We specialize this result to the case in which $W$ is $\mathbb P^2$ or a Hirzebruch surface, determining completely $\mathcal O_W$-module structure of $\pi_*\mathcal O_X$ (see Proposition~\ref{prop.split.minimal.degree}).

\section{Canonical covers whose trace zero module splits}\label{section.split}

 In this section we prove results for canonical covers $\pi$ which are not necessarily abelian (except in Proposition~\ref{prop.split.minimal.degree} (2) (b)) but whose trace zero modules splits as direct sum of line bundles (e.g, if $\pi$ is a natural deformation of an abelian cover; see \cite[Definition 5.1]{Pardini}, \cite[Definition 3.2]{FantechiPardini} or \cite[Definition 2.8]{Catanese}). We start by setting  notations and definitions that we will use in this section and in the remaining of the article.

\begin{setup}\label{setup.beginning}
{\rm {\bf Set-up and notation:} We work over an algebraically closed field of characteristic $0$. Throughout this article and unless otherwise said:
\begin{enumerate}
    \item 
$X$ will be an irreducible variety.

\item $W$ will be a smooth, projective variety  
and  $i: W \hookrightarrow \mathbf P^N$ will be an embedding of $W$ in $\mathbf P^N$ induced by a complete linear series. We will abridge $\mathcal O_W(1) := i^*\mathcal O_{\mathbf P^N}(1)$; $\omega_W(-1)= \omega_W \otimes \mathcal O_W(-1)$, etc. 
\end{enumerate}}
\end{setup}

\noindent We will adopt the following definition: 

\begin{definition}\label{defi.canonical.cover}
{\rm A morphism $\pi: X \longrightarrow W$ is a canonical cover of $W$ if 
\begin{enumerate}
\item $X$ is {a normal and  locally Gorenstein} variety of general type 
whose canonical bundle
is ample and base-point-free; 
\item the canonical morphism $\varphi: X \longrightarrow \mathbf P^N$ of $X$ factors as
$\varphi = i \circ \pi$.
\end{enumerate}} 
\end{definition}

\begin{remark}\label{remark.CMflatness}
{\rm Under the hypotheses of Definition~\ref{defi.canonical.cover}, $\pi$ is a flat, finite, surjective morphism (flatness follows, e.g., from \cite[Theorem 23.1%
]{Matsumura} or \cite[Theorem 18.16]{Eisenbud}). 
}
\end{remark}

\subsection{A general result} \ 

\smallskip
\noindent In this subsection we prove a result, namely, Theorem~\ref{thm.splitting},
describing the properties of canonical covers $\pi: X \longrightarrow W$ with split trace zero module, where the hypothesis on $X$ and $W$ are quite general (see Definition~\ref{defi.canonical.cover} (1) and \eqref{eq.p_g=0}; in particular $X$ and $W$ have arbitrary dimension). We also prove 
Lemma~\ref{lemma.criterion.canonical.cover}, which yields a sort of converse to Theorem~\ref{thm.splitting}. Theorem~\ref{thm.splitting} {(in particular, Theorem~\ref{thm.splitting} (II) (4))} and  Lemma~\ref{lemma.criterion.canonical.cover} will be crucial in the remaining of the article.

\begin{theorem}\label{thm.splitting}
   Let $\pi : X \longrightarrow  W$ be a canonical cover of degree $n$, $n \geq 2$. 
   \begin{enumerate}
       \item[(I)] Assume that  
       \begin{equation}\label{eq.p_g=0}
            h^0(\omega_W(-1))=0
       \end{equation}
       (e.g., if $p_g(W)=0$). Then  $\pi_*\mathcal O_X$ splits as 
   \begin{equation}\label{eq.trace.module.split}
 \pi_*\mathcal O_X \simeq   \mathcal{O}_{W} \oplus E' \oplus \omega_W(-1), 
 \end{equation}
 where $E'$ is a vector bundle of rank $n-2$ which is self-dual after twisting; precisely
 \begin{equation}\label{eq.self.duality}
     E'(1) \simeq \omega_W \otimes {E'}^*,
 \end{equation}
 and 
 \begin{equation}\label{eq.vanishing.global.sections}
     h^0(\omega_W \otimes {E'}^*)=0. 
 \end{equation}

 \item[(II)]  
 Assume in addition that $\pi_{*}\mathcal{O}_{X}$ splits as sum of line bundles.
In this case, we may set
\begin{equation}\label{formula.splitting}
    \pi_{*}\mathcal{O}_{X} = \mathcal{O}_{W} \oplus \mathcal L_{1}^{*} \oplus \cdots \oplus \mathcal L_{n-1}^{*}, 
\end{equation}
where $\mathcal L_1, \dots, \mathcal L_{n-1}$ are line bundles. Then the following holds: 

\begin{enumerate}
    \item[(1)]  $\mathcal L_{n-1} \simeq \omega_W^*(1)$. 
    \item[(2)] There is a permutation $\sigma$ of $S_{n-2}$, of order $1$ or $2$, such that \linebreak
    $\mathcal L_{\sigma(i)} \simeq \mathcal L_{i}^* \otimes \omega_W^*(1)$ for all $1 \leq i \leq n-2$.
    \item[(3)]  $p_g(W)=0$ and $h^0(\omega_{W} \otimes \mathcal L_{i})=0$ for all $1 \leq i \leq n-2$.
    \item[(4)] $\mathcal L_1^{\otimes 2} \otimes \cdots \otimes \mathcal L_{n-2}^{\otimes 2} \simeq \mathcal L_{n-1}^{\otimes n-2}$.
    \item[(5)] If, in addition, $n$ is even and the Picard group of $W$ is free of $2$-torsion, then 
    \begin{equation*}
   \mathcal L_1 \otimes \cdots \otimes \mathcal L_{n-2} \simeq \mathcal L_{n-1}^{\otimes \frac{n}{2}-1} \hskip -.2 cm.
    \end{equation*}
\end{enumerate}
   \end{enumerate}
\end{theorem}

\begin{proof}
  Since $\pi$ is flat, $\pi_{*}\mathcal{O}_{X}$ is a locally free $\mathcal{O}_{W}$ module of rank $n$.
  We have 
  \begin{equation}\label{eq.trace.zero.module}
     \pi_{*}\mathcal{O}_{X} = \mathcal{O}_{W} \oplus E, 
  \end{equation}
  where $E$ is the trace zero module of $\pi$. 
  By adjunction (see \cite[Theorem 6.4.9]{Liu}) we have 
  \begin{equation*}
  \omega_{X} = \omega_{X/W} \otimes\pi^{*}\omega_{W},    
  \end{equation*}
 and by projection formula and relative duality 
 we have
 \begin{equation}\label{eq.relative.duality}
  \pi_*\omega_{X} 
  = (\pi_*\mathcal O_X)^* \otimes \omega_{W}.
  \end{equation}
  Since $\varphi$ is the canonical morphism of $X$, 
  \begin{equation*}
      \omega_X=\varphi^*\mathcal O_{\mathbf{P}^N}(1)=\pi^*\mathcal O_W(1),
  \end{equation*}
  so, by projection formula, 
    \begin{equation}\label{eq.by.projection.formula}
      \pi_*\omega_X=\pi_*\mathcal O_X \otimes \mathcal O_W(1).
  \end{equation}
  Thus
 \begin{equation}\label{eq.cor.relative.duality}
      \pi_*\mathcal O_X \otimes \mathcal O_W(1)= (\pi_*\mathcal O_X)^* \otimes \omega_{W}, 
  \end{equation}
so, from \eqref{eq.trace.zero.module}, we get 
\begin{equation}\label{eq0}
    \mathcal O_W \oplus E \simeq \omega_W(-1) \oplus (\omega_W \otimes E^*(-1)).
\end{equation}
By hypothesis, $h^0(\omega_W(-1))=0$. Then,  by the Krull-Schmidt theorem (see e.g. \cite[Theorems 1 and 3]{Atiyah}), $E \simeq E' \oplus \omega_W(-1)$, where $E'$ is vector bundle on $W$ of rank $n-2$, so \eqref{eq.trace.module.split} holds. Then 
\eqref{eq0} becomes 
\begin{equation*}
    \mathcal O_W \oplus E' \oplus \omega_W(-1) \simeq \omega_W(-1) \oplus (\omega_W \otimes {E'}^*(-1)) \oplus \mathcal O_W.
\end{equation*}
so, by the Krull-Schmidt theorem, $E'$ satisfies \eqref{eq.self.duality}. 
Since the canonical morphism $\varphi$ of $X$ factors as $\varphi = i \circ \pi$, we get $H^0(\omega_X)=H^0(\mathcal O_W(1))$. By \eqref{eq.by.projection.formula}, 
    \begin{equation*}
        H^0(\omega_X)=H^0(\pi_*\omega_X)=H^0(\pi_*\mathcal{O}_X \otimes \mathcal O_W(1)),
    \end{equation*} so 
\eqref{eq.trace.module.split} and \eqref{eq.self.duality} imply \eqref{eq.vanishing.global.sections}.

\smallskip
\noindent Now assume $\pi_*\mathcal O_X$ splits as direct sum of line bundles and notation \eqref{formula.splitting}. Then, after a harmless renumbering if necessary, \eqref{eq.trace.module.split} yields (1). Since $\pi_*\mathcal O_X$ splits as direct sum of line bundles, \eqref{eq.self.duality} implies (2).

    \smallskip
    \noindent  By \eqref{eq.trace.module.split}, \eqref{formula.splitting} and (1), 
    \begin{equation*}
        E'= \mathcal L_1^* \oplus \cdots \oplus \mathcal L_{n-2}^*,
    \end{equation*}
so (3) follows from \eqref{eq.vanishing.global.sections} and (2).

 \smallskip
    
    \noindent By 
   (1) and (2) we have
    \begin{equation}\label{eq1}
    \mathcal O_{W} \oplus \mathcal L_{1}^{*} \oplus \cdots \oplus \mathcal L_{n-1}^{*}  = \omega_{W}(-1)  \oplus (\mathcal L_{1}  \otimes
    \omega_{W}(-1)) \oplus \cdots \oplus (\mathcal L_{n-1}  \otimes
    \omega_{W}(-1)).    
  \end{equation}
    Then  (4) follows from (1) and from taking determinant on both sides of  \eqref{eq1}.

     \smallskip
    \noindent Statement (5) is a straight forward consequence of (4).
\end{proof}

\begin{remark}\label{remark.thm.splitting}
\begin{enumerate}
    \item With the assumptions of Theorem~\ref{thm.splitting} (I), the conditions $h^0(\omega_W(-1))=0$ and $p_g(W)=0$ are equivalent, so, as hypothesis in the statement of Theorem~\ref{thm.splitting}, we could have used  the latter instead of the former.  
    \item With the assumptions of Theorem~\ref{thm.splitting} (I) and (II), $h^0(\mathcal L_i^*(1))=0$ for all $1 \leq i \leq n-1$. 
    \item If we do not assume $\pi$ to be a canonical cover as in Definition~\ref{defi.canonical.cover} but assume only $i \circ \pi$ to be induced by a, maybe incomplete, linear series of $|\omega_X|$ 
    (in this case, $i$ is not necessarily induced by a complete linear series either), then, except for (I) (1.4.4) and (II) (3), Theorem~\ref{thm.splitting} still holds.
\end{enumerate}
\end{remark}

\begin{corollary}\label{cor.234}
   With the hypotheses of Theorem~\ref{thm.splitting}:
   \begin{enumerate}
       \item If $n=2$, then $\pi_*\mathcal O_X=\mathcal O_W \oplus \omega_W(-1)$ and $|\omega_W^{-2}(2)|$ has a reduced divisor $D$ (if $X$ is smooth, then  $|\omega_W^{-2}(2)|$ has a smooth divisor $D$). 
       \item If $n=3$, then $\mathcal L_1^{\otimes 2}=\mathcal L_2=\omega_W^*(1)$; in particular, $\omega_W^*(1)$ is $2$-divisible in $\mathrm{Pic}(W)$.
       \item If $n=4$
       {and Pic$(W)$ is free of $2$-torsion}, then $\mathcal L_1 \otimes \mathcal L_2=\mathcal L_3=\omega_W^*(1)$ (compare  with \cite[Proposition 2.1 (2)]{quad1}, where the same result is stated only for smooth surfaces $W$ of minimal degree) 
   \end{enumerate}
\end{corollary}

\begin{proof}
    The corollary is a straightforward consequence of Theorem~\ref{thm.splitting} and, in the case of (1), of the fact that $\pi$ is a double cover, hence abelian, branched along a divisor $D$ of $|\omega_W^{\otimes 2}(2)|$; by the local equation of a simple cyclic cover (see e.g., the proof of \cite[Lemma I.17.1]{BHPV}), if $X$ is normal (respectively, smooth), then $D$ is reduced (respectively, smooth). 
   \end{proof}

\noindent We will need the following lemma in Sections~\ref{section.abelian.tridouble} and \ref{section.abelian.Z2xZ4}.

\begin{lemma}\label{lemma.criterion.canonical.cover}
Let $X$ be normal and locally Gorenstein and let $\pi: X \longrightarrow W$ be a finite cover of $W$. 
Assume that 
\begin{enumerate}
    \item there exists a numerically trivial line bundle $\mathcal L$ on $X$ such that $\omega_X = \pi^*\mathcal O_W(1) \otimes \mathcal L$;
    \item both \eqref{eq.trace.module.split} and \eqref{eq.self.duality} hold; and  
    \item both \eqref{eq.p_g=0} and \eqref{eq.vanishing.global.sections} hold.                  
\end{enumerate}
    Then $\mathcal L$ is trivial and $\pi$ is a canonical cover.
\end{lemma}

\begin{proof}
    By the projection formula, (1) implies 
    \begin{equation}\label{eq.push.forward.omega1}
        \pi_*\omega_X = \pi_*\mathcal L  \otimes  \mathcal O_W(1).
    \end{equation}
    By \eqref{eq.relative.duality} and (2), 
    \begin{equation}\label{eq.push.forward.omega2}
        \pi_*\omega_X = \pi_*\mathcal O_X \otimes  \mathcal O_W(1).
    \end{equation}
    Then \eqref{eq.push.forward.omega1} and \eqref{eq.push.forward.omega2} imply 
    $\pi_*\mathcal L = \pi_*\mathcal O_X$, so $h^0(\mathcal L)=1$, hence $\mathcal L$ is effective. Since $\mathcal L$ is numerically trivial and $X$ is projective, $\mathcal L$ is trivial. Then $\omega_X = \pi^*\mathcal O_W(1)$ and, in particular, $\omega_X$ is ample and base-point-free. By \eqref{eq.push.forward.omega2}, (2) and (3), $H^0(\omega_X)=H^0(\mathcal O_W(1))$,  
    so $|\omega_X| = \pi^*|\mathcal O_W(1)|$. Therefore the canonical map of $X$ factors as
$i \circ \pi$.
\end{proof}

\subsection{Canonical covers of  {$\mathbb P^2$ and Hirzebruch surfaces}}  \ 

\smallskip
\noindent In this section we apply Theorem~\ref{thm.splitting} to completely describe (see \eqref{eq.splitting.P^2}, \eqref{eq.regular.splitting}, \eqref{eq.regular.splitting.a>1} and \eqref{eq.irregular.splitting}) the trace zero module of a  canonical cover $\pi: X \longrightarrow W$, when $W$ is $\mathbb P^2$ or a  Hirzebruch surface. This description  will be crucial in Section~\ref{section.abelian.rational}. 
We give the description when $\pi_*\mathcal O_X$  splits as direct sum of line bundles. The latter condition happens of course when $\pi$ is abelian, but in the case $X$ is regular, do not requite this hypothesis. 

\smallskip
\noindent We start by setting the notation we use when $W$ is a Hirzebruch surface. 

\begin{notation}\label{notation.Hirzebruch}
 If $W$ is a Hirzebruch surface $\mathbb{F}_{e}$, we will use this notation: 
 
 \begin{enumerate}
 
\item Following \cite[Notation V.2.8.1]{Hart}, if $e > 0$, then $C_0$ will be the minimal section of $\mathbb{F}_{e}$ (minimal with respect to self-intersection, which is $C_0^2=-e$) and $f$ will be a fiber of $W$; if $e=0$, then $C_0$ and $f$ will be two non linearly equivalent fibers.

\item We will set  $\mathcal O_W(1)=\mathcal O_W(\alpha C_0+mf)$, with $\alpha \geq 1$. Since $W$ is a smooth surface, $m \geq \alpha e+1$. If $e=0$, after applying to $\mathbb{F}_0$, if necessary, the automorphism that swaps the factors of $\mathbb{P}^1 \times \mathbb{P}^1$, we will assume $m \ge \alpha$. 
\end{enumerate}
\end{notation}

\begin{prop}\label{prop.split.minimal.degree}
 Let $\pi : X \longrightarrow W$ be a canonical cover of degree $n$, $n > 2$, such that $\pi_{*} \mathcal{O}_{X}$ splits as
a direct sum of line bundles.

\begin{enumerate}
  \item Let $W$ be $\mathbb{P}^{2}$.  Then $X$ is regular, $\mathcal O_W(1)=\mathcal O_{\mathbb{P}^2}(1)$ and 
  \begin{equation}\label{eq.splitting.P^2}
      \pi_{*}\mathcal{O}_{X} = \mathcal{O}_{\mathbb{P}^{2}} \oplus \mathcal{O}_{\mathbb{P}^{2}}(-2)^{\oplus n-2} \oplus \mathcal{O}_{\mathbb{P}^{2}}(-4).
  \end{equation}

  \item  Let $W$ be a Hirzebruch surface. Then $n$ is even. 

  \smallskip
  \begin{enumerate}
  \item Let $X$ be regular.
  \item[] If $\alpha=1$, then 
  \begin{multline}\label{eq.regular.splitting}
      \pi_{*}\mathcal{O}_{X} = \mathcal{O}_{W} \oplus \mathcal{O}_{W}(-C_{0} -(m+1)f)^{\oplus \frac{n-2}{2}}\, \oplus \\ 
     \mathcal{O}_{W}(-2C_{0} -(e+1)f)^{\oplus \frac{n-2}{2}} \oplus \mathcal{O}_{W}(-3C_{0} - (m+e+2)f). 
           \end{multline}

           \item[]    If $\alpha \ge 2$, then $e=0$ and
        \begin{multline}\label{eq.regular.splitting.a>1}
      \pi_{*}\mathcal{O}_{X} = \mathcal{O}_{W} \oplus \mathcal{O}_{W}(-C_{0} -(m+1)f)^{\oplus \frac{n-2}{2}}\, \oplus \\ 
     \mathcal{O}_{W}(-(\alpha +1)C_{0} -f)^{\oplus \frac{n-2}{2}} \oplus \mathcal{O}_{W}(-(\alpha +2)C_{0} - (m+2)f). 
           \end{multline} 

            \medskip
            
     \item Let $X$ be irregular and let $\pi$ be abelian.  Then $e = 0$ and there are nonnegative integers  $n_0, n_0', n_1, n_1'$ that satisfy $n_0+ n_0'+n_1' >0$ and 
     $(n_0 + n_0' + n_1 + n_1')=\frac{n-2}{2}$, such that 
     \scriptsize{\begin{multline}\label{eq.irregular.splitting}
  \pi_{*}\mathcal{O}_{X} =  \mathcal{O}_{W} \oplus \mathcal{O}_{W}(-(m+1)f)^{\oplus n_0} \oplus \mathcal{O}_{W}(-(m+2)f)^{\oplus n_0'} \oplus \\
 \hskip -1 cm   
 \mathcal{O}_{W}(-C_{0} -(m+1)f)^{\oplus n_1} \oplus
    \mathcal{O}_{W}(-C_{0} -(m+2)f)^{\oplus n_1'} \oplus \\
   \qquad  \qquad  \qquad   \qquad  
    \mathcal{O}_{W}(-(\alpha +1)C_{0})^{\oplus n_1'} \oplus 
    \mathcal{O}_{W}(-(\alpha +1)C_{0}-f)^{\oplus n_1} \oplus 
    \mathcal{O}_{W}(-(\alpha+2)C_{0})^{\oplus n_0'}  \oplus \\
 \qquad  \qquad  \mathcal{O}_{W}(-(\alpha+2)C_{0}-f)^{\oplus n_0}
   \oplus \mathcal{O}_{W}(-(\alpha+2)C_{0} - (m+2)f).
    \end{multline}}

\normalsize
     \medskip

     \item Conversely, if \eqref{eq.regular.splitting} holds or, if \eqref{eq.regular.splitting.a>1}, with $e=0$, holds, then $X$ is regular; and, if \eqref{eq.irregular.splitting}, with $e=0$ and $n_0+ n_0'+n_1' >0$, holds, then $X$ is irregular. 
     \end{enumerate}

\end{enumerate}
\end{prop}

\begin{proof}
Let us prove (1). Let $l$ positive integer such that $\mathcal O_W(1)=\mathcal O_{\mathbb{P}^2}(l)$. For each $i$, $1 \leq i \leq n-2$ and $\mathcal L_i$ as in \eqref{formula.splitting}, let $\mathcal L_i=\mathcal O_{\mathbb P^2}(\alpha_i)$. Theorem~\ref{thm.splitting} (II) (3) implies $\alpha_i \leq 2$ and Theorem~\ref{thm.splitting} (II) (2) implies $\alpha_i \geq l+1$, hence $l=1$ 
and $\alpha_i=2$ for all $i$, $1 \leq i \leq n-2$. Theorem~\ref{thm.splitting} (II) (1) completes the proof of (1).

\smallskip
 \noindent Now assume that $W$ is a Hirzebruch surface. Recall that  $$\omega_{W} = \mathcal{O}_{W}(-2C_{0} -(e+2)f).$$ By Theorem~\ref{thm.splitting} (II) (1), the line bundle $\mathcal L_{n-1}$ of \eqref{formula.splitting} is  
 $$\mathcal L_{n-1}=\omega_W^*(1)=\mathcal O_W((\alpha+2)C_0+(m+e+2)f).$$
For any $1 \leq i \leq n-2$ and $\mathcal L_i$ of \eqref{formula.splitting}, we will set $\mathcal L_i=\mathcal O_W(\alpha_iC_0+\beta_if)$. According to Theorem~\ref{thm.splitting} (II) (2), we also have
 \begin{equation*}
    \mathcal L_{\sigma(i)}= \mathcal L_i^* \otimes \omega_W^*(1) = 
 \mathcal O_W((\alpha+2-\alpha_i)C_0+(m-\beta_i+e+2)f),   
 \end{equation*}
so
\begin{equation}\label{eq.self.duality.Hirz}
\begin{matrix}
  \alpha_{\sigma(i)} & = & \alpha+2-\alpha_i \\
 \beta_{\sigma(i)} & = & m-\beta_i+e+2.
\end{matrix}
\end{equation}

 \smallskip
 \noindent Now we prove (2) (a), including the claim that $n$ is even in this case. We are assuming that $X$ is regular. 
 We claim $\alpha_i \geq 1$ for all $1 \leq i \leq n-2$. Indeed, suppose the contrary and let $j$ be such that $\alpha_j \leq 0$. By Remark~\ref{remark.thm.splitting} (2), 
 \begin{equation}\label{eq.vanishing.L*(1)}
 H^0(\mathcal O_W((\alpha-\alpha_j)C_0 + (m-\beta_j)f))=H^0(\mathcal L_j^*(1))=0. 
 \end{equation}
 Then $\beta_j \geq m+1$. Since $X$ is regular, $H^1(\pi^*\mathcal O_X)=H^1(\mathcal O_X)=0$, so $H^1(\mathcal L_j^*)=0$. Since $\alpha_j \leq 0$, it follows that $\beta_j \leq 1$. But then $m+1 \leq \beta_j \leq 1$, so $m \leq 0$. This is impossible since $m \geq \alpha e+1 \geq 1$, as seen in Notation~\ref{notation.Hirzebruch}. 

 \noindent 
 Then, for all  $1 \leq i \leq n-2$, we have $\alpha+2-\alpha_i = \alpha_{\sigma(i)} \geq 1$, hence $1 \le \alpha_i \leq \alpha+1$. 
 If 
 $1 \le \alpha_i \le \alpha$, then, by Remark~\ref{remark.thm.splitting} (2), 
 \begin{equation}\label{eq.lower.bound.on.b_i}
    0=h^0(\mathcal L_i^*(1))\ge h^0(\mathcal O_{\mathbb P^1}(m-\beta_i)), 
 \end{equation}
so $\beta_i \geq m+1$ in this case. 
 
 \smallskip
 \noindent We now prove that there are no $i$, $1 \leq i \leq n-2$, such that $2 \le \alpha_i \le \alpha$. Suppose the contrary and let $i$ be such that $2 \le \alpha_i \le \alpha$. On the one hand, by \eqref{eq.self.duality.Hirz}, $2 \le \alpha_{\sigma(i)} \le \alpha$, so  $\beta_{i} \ge m+1$ and
 $\beta_{\sigma(i)} \ge m+1$. On the other hand, by \eqref{eq.self.duality.Hirz}, 
  $\beta_{\sigma(i)}=m-\beta_i+e+2$, so $\beta_{\sigma(i)} \le e+1$. Thus $m \le e$. This is impossible since $m \geq \alpha e+1$ (see Notation~\ref{notation.Hirzebruch}). Thus, if $i$, $1 \le i \le n-2$, then $\alpha_i=1$, in which case, by \eqref{eq.self.duality.Hirz}, $\alpha_{\sigma(i)}=\alpha+1 \neq 1$ (because $\alpha \ge 1$); or $\alpha_i=\alpha+1$, in which case $\alpha_{\sigma(i)}=1 \neq  \alpha+1 $. Therefore $n-2$ is even and there are exactly $\frac{n-2}{2}$ of the $\alpha_i$ equal to $1$ and exactly $\frac{n-2}{2}$ of the $\alpha_i$ equal to $\alpha+1$.
\smallskip

\noindent Let us show that, if $\alpha \ge 2$, then $e=0$ and \eqref{eq.regular.splitting.a>1} holds. Let   $i$, $1 \le i \le n-2$ such that $\alpha_i=1$ (hence  $\alpha_{\sigma(i)}=\alpha+1$). 
Since $X$ is regular, $H^1(\mathcal L_{\sigma(i)}^*)=0$ and,  
 by Serre duality, 
  $$0=h^{1}(\mathcal L_{\sigma(i)}^{*}) = h^{1}(\mathcal{O}_{W}((\alpha-1)C_0+(\beta_{\sigma(i)} - e - 2)f)) \geq 
  h^1(\mathcal O_{\mathbb P^1}(\beta_{\sigma(i)} -\alpha e-2)),$$ so $\beta_{\sigma(i)} \geq \alpha e+1$. Recall that $\beta_i \ge m+1$, so  $\beta_{\sigma(i)}=m-\beta_i+e+2 \le e+1$. Then
  \begin{equation}\label{eq.bound.b_j}
    \alpha e+1 \le \beta_{\sigma(i)} \le e+1,   
  \end{equation}
  so, if $\alpha \ge 2$, then $e=0$, $\beta_{\sigma(i)}=1$ and, by \eqref{eq.self.duality.Hirz}, $\beta_i=m+1$.

  \smallskip
  \noindent If $\alpha=1$, then \eqref{eq.bound.b_j} yields $\beta_{\sigma(i)}=e+1$ and, by \eqref{eq.self.duality.Hirz}, $\beta_i=m+1$, so \eqref{eq.regular.splitting} holds in this case. This completes the proof of (2) (a).

\smallskip

\noindent Let us now prove (2) (b), including the claim that $n$ is even in this case. We are assuming that $X$ is irregular. By Theorem~\ref{thm.split.abelian}, $\pi$ is under the assumption of Theorem~\ref{thm.splitting} (II).  
By Theorem~\ref{thm.Pardini}, there is an integer $d_i$, $d_i \geq 2$ (if $\mathcal L_i=L_\chi$, with $\chi$ as in Theorem~\ref{thm.split.abelian}, then $d_i$ is the order of $\chi$), such that $|d_i \alpha_iC_0+ d_i \beta_if|$ is non empty, so $\alpha_i, \beta_i \geq 0$. Then, by \eqref{eq.self.duality.Hirz}, $\alpha_i \leq \alpha+2$ for all $1 \leq i \leq n-2$. If $0 \leq \alpha_i \leq \alpha$, then $\beta_i \geq m+1$, because \eqref{eq.vanishing.L*(1)} and \eqref{eq.lower.bound.on.b_i} also hold for $X$ irregular. If $\alpha_j=\alpha+1$ or $\alpha+2$, then, by  \eqref{eq.self.duality.Hirz}, $\alpha_{\sigma(j)}=1$ or $0$ respectively, so $\beta_{\sigma(j)} \geq m+1$ and, again by \eqref{eq.self.duality.Hirz}, $\beta_j \leq e+1$. Arguing as in the case when $X$ is regular, we get  that there are no $i$, $1 \leq i \leq n-2$, such that $2 \le \alpha_i \le \alpha$, so, for any $i$, $1 \leq i \leq n-2$, $\alpha_i=0, 1, \alpha+1$ or $\alpha+2$.
\smallskip

\noindent
Let us show $e=0$. First assume that $\alpha_i \neq 0$ for all $1 \leq i \leq n-2$. Then, by \eqref{eq.self.duality.Hirz}, each $\alpha_i$ is either $1$ or $\alpha+1$. Since $H^1(\mathcal L_i^*)=0$ if $\alpha_i=1$ and $X$ is irregular, there is some $\alpha_j=\alpha+1$ such that $H^1(\mathcal L_j^*) \neq 0$. Since
$$h^1(\mathcal O_W(-(\alpha+1)C_0-\beta_jf))=h^1(\mathcal O_W((\alpha-1)C_0+(\beta_j-e-2)f),$$
$H^1(\mathcal L_j^*) \neq 0$ if and only if 
$H^1(\mathcal O_{\mathbb P^1}(\beta_j-\alpha e-2)) \neq 0$;
therefore there is $j$, $1 \leq j \leq n-2$, such that $\alpha_j=\alpha+1$ and $\beta_j \leq \alpha e$. 
 The total branch locus of $\pi$ is a divisor $D=\sum D_{H,\psi}$ (see \eqref{eq.branch.divisor}), where the $D_{H,\psi}$ are effective divisors on $W$ indexed by the cyclic subgroups $H$ of $G$ and, for each $H$, the generators $\psi$ of the group of characters of $H$.  By Theorem~\ref{thm.Pardini}, for each $1 \leq i \leq n-2$,
 the divisor $d_i \alpha_iC_0+ d_i \beta_if$ is linearly equivalent to a combination of the $D_{H,\psi}$ with  integer coefficients $\delta^i_{H,\psi}$ such that  $0 \leq \delta^i_{H,\psi} \leq d_i-1$. If $d_iC_0$ were in the fixed locus of 
  $|d_i \alpha_iC_0+ d_i \beta_if|$, then $C_0$ would appear in $D$ with coefficient greater than $1$, so, by  Theorem~\ref{thm.normal}, $X$ would be non normal. Therefore, $d_iC_0$ cannot be in the fixed locus of 
  $|d_i \alpha_iC_0+ d_i \beta_if|$, for any $1 \leq i \leq n-2$ and, in particular, for any $i=j$, $1 \leq j \leq n-2$, such that $\alpha_j=\alpha+1$ and $\beta_j \leq \alpha e$.
  Thus, for such $j$,  
  $$0 \leq ((\alpha d_j+1)C_0+d_j\beta_jf)C_0= -(\alpha d_j+1)e + d_j\beta_j \le -(\alpha d_j+1)e + \alpha d_je = -e,$$
  so $e=0$.

\smallskip
\noindent Now assume that, for some $i$, $1 \leq i \leq n-2$, $\alpha_i=0$. Then, by \eqref{eq.self.duality.Hirz}, for some $j$, $1 \leq j \leq n-2$, $\alpha_j=\alpha+2$. Since $X$ is normal, arguing as above, for such $j$, $d_jC_0$ cannot be in the fixed locus of $|(\alpha+2)d_jC_0+ d_j \beta_jf|$. Therefore 
\begin{multline*}
 0 \leq   ((\alpha d_j+d_j+1) C_0+ d_j \beta_jf)C_0 = -(\alpha d_j+d_j+1)e+d_j \beta_j \leq \\
 -(\alpha d_j+d_j+1)e+d_j (e+1)=d_j-e(\alpha d_j+1),
\end{multline*}
which implies $e \leq \dfrac{d_j}{\alpha d_j+1} < 1$, so $e=0$.

\smallskip

\noindent Now recall that, for any $1 \leq i \leq n-2$,  $\alpha_i=0, 1, \alpha+1$ or $\alpha+2$.  Since $e=0$, if $\alpha_j=\alpha+1$ or $\alpha+2$, then $\beta_j=0$ or $1$ and, by \eqref{eq.self.duality.Hirz}, if $\alpha_i=0$ or $1$, then  $\beta_i=m+1$ or $m+2$. More precisely, Theorem \ref{thm.splitting} II (2) implies that, for each $\lambda$, $\mu$, the sets $$\{i \ | \ \alpha_i=\lambda, 
\beta_i=\mu\}  \ \mathrm{and} \   \{j \ | \ \alpha_j=\alpha+2-\lambda, 
{\beta_j=m+2-\mu}\}$$ 
have the same cardinal. Since $\alpha \geq 1$, this implies that $n$ is even. This also proves the splitting \eqref{eq.irregular.splitting} except for the claim that $n_0+ n_0'+n_1' >0$.
To prove this claim, recall that, on the one hand, $$H^1(\mathcal{O}_{W})=H^1(\mathcal{O}_{W}(-C_{0} -(m+1)f))=
   H^1(\mathcal{O}_{W}(-(\alpha+1)C_{0} -f))=0.$$ On the other hand, by Serre's duality, $$h^1(-(\alpha+2)C_0-(m+2)f)=h^1(\mathcal{O}_{W}(\alpha C_0+mf))=0.$$
Since $H^1(\mathcal O_X)=H^1(\pi_*\mathcal O_X)$ and $X$ is irregular, then at least one among $n_0$, $n_0'$ and $n_1'$ is positive, hence $n_0+ n_0'+n_1' >0$. This completes the proof of (2) (b).

\smallskip

\noindent Now we prove (2) (c).  Assume now that $\pi_*\mathcal O_X$ satisfies \eqref{eq.regular.splitting}, or \eqref{eq.regular.splitting.a>1} and $e=0$. The group $H^1(\mathcal O_W(-C_0-(m+1)f)$  vanishes.  The cohomology group $H^1(\mathcal O_W(-(\alpha+1)C_0-f)$  vanishes if $e=0$. By Serre's duality, $$h^1(\mathcal O_W(-2C_0-(e+1)f)=
  h^1(\mathcal O_W(-f)= h^1(\mathcal O_{\mathbb P^1}(-1)),$$ which also vanishes.
  By Serre's duality, $$h^1(\mathcal O_W(-(\alpha+2)C_0-(m+e+2)f)=
  h^1(\mathcal O_W(\alpha C_0+mf),$$ which vanishes because $m \ge \alpha e+1$.
   Since $H^1(\mathcal O_X)=H^1(\pi_*\mathcal O_X)$, then $X$ is regular.  

   \smallskip
\noindent Now assume that $e=0$ and \eqref{eq.irregular.splitting} holds, with $n_0+ n_0'+n_1' >0$. This means that at least one among $n_0$, $n_0'$ and $n_1'$ is positive. On the one hand, since $m \geq \alpha e+1 = 1$, $H^1(\mathcal{O}_{W}(-(m+1)f))$ and 
$H^1(\mathcal{O}_{W}(-(m+2)f))$ do not vanish; on the other hand, $h^1(\mathcal{O}_{W}(-(\alpha +1)C_0)) = \alpha \geq 1$. Thus, $X$ is irregular.
\end{proof}

\section{Background on Abelian covers}

In the remaining of the article we deal with abelian covers of algebraic varieties. This class of covers were studied extensively and in detail by Pardini and we refer the reader to \cite{Pardini} (see also  Catanese's precursor description of bidouble covers in \cite[Section 2]{Catanese}). For convenience, we recall here some basic facts about them that we will use in Sections~\ref{section.abelian.canonical} and  \ref{section.abelian.rational}.

\begin{definition}\rm{(\cite[Definition 1.1]{Pardini}). Let $G$ be a finite abelian group and let $\mathcal X$, $\mathcal Y$ be {irreducible} algebraic varieties. An abelian cover with group $G$ is a finite morphism $\pi: \mathcal X \longrightarrow \mathcal Y$ together with a faithful action of $G$ on $\mathcal X$ such that $\pi$ exhibits $\mathcal Y$ as a quotient of $\mathcal X$ by $G$. 
}
\end{definition}

\noindent If an abelian cover $\pi$ is flat
(as observed in Remark~\ref{remark.CMflatness}, this happens if \linebreak $\pi: X \longrightarrow W$ is an abelian canonical cover), we have the following result:   

\begin{theorem}\label{thm.split.abelian}
    Let $\pi: \mathcal X \longrightarrow \mathcal Y$ be an abelian cover with group $G$ and let $G^*$ be the group of characthers of $G$. If $\pi$ is flat, then  
    \begin{equation}\label{eigenspaces} 
    \pi_*\mathcal O_\mathcal X = \mathcal O_\mathcal Y 
    \oplus \bigoplus_{\chi \in G^*, \chi \neq 1} L_\chi^*,
    \end{equation}
    where, for each non-trivial character $\chi$ of $G^*$, $L_\chi$ is a line bundle and $G$ acts on $L_\chi^*$ via $\chi$. 
\end{theorem}

\noindent  If $\pi:\mathcal X \longrightarrow \mathcal Y$ is an abelian cover with group $G$, the ring structure of $\pi_*\mathcal O_\mathcal X$ is compatible with the action of $G$, so the multiplication is determined by $\mathcal O_\mathcal Y$-linear maps
\begin{equation*}
\begin{matrix} 
 \mathcal O_\mathcal Y \otimes \mathcal O_\mathcal Y &\longrightarrow & \mathcal O_\mathcal Y \\ 
 \mathcal O_\mathcal Y \otimes L_\chi^* & \longrightarrow & L_\chi^* \\
L_\chi^* \otimes L_{\chi'}^* & \longrightarrow & L_{\chi\chi'}^* & \hskip -.2cm ,
\end{matrix}
\end{equation*}
for any non trivial $\chi,\chi' \in G^*$, 
where the first one among them is the ring multiplication on  $\mathcal O_\mathcal Y$ and the maps in second row are determined by the $\mathcal O_\mathcal Y$-module structure of the  $L_\chi$. 
For each $\chi,\chi' \in G^*$, the maps in the third row are given by the multiplication by a global section %
of the line bundle  $L_\chi \otimes L_{\chi'} \otimes L_{\chi\chi'}^*$, so that we have the following result (see \cite[Theorem 2.1]{Pardini}): 

\begin{theorem}\label{thm.Pardini}
Let $G$ be an abelian group, let $\mathcal Y$ be a smooth, irreducible variety, let $\mathcal X$ be an irreducible, normal variety and  let $\pi: \mathcal X \longrightarrow \mathcal Y$ be an abelian cover with group $G$. Then, for all $\chi, \chi'$ non trivial characters of $G$, there exists an effective divisor $D_{\chi, \chi'}$, 
such that 
    \begin{equation}\label{relation.ring.structure}
    L_\chi \otimes L_{\chi'} =  L_{\chi\chi'}(D_{\chi\chi'}),
\end{equation}
where $D_{\chi\chi'}$ is the sum 
\begin{equation}\label{relation.divisor.ring.structure}
D_{\chi\chi'}=  \sum_{H \in \mathfrak{C}} \sum_{\psi \in S_{H}} \epsilon_{\chi, \chi'}^{H, \psi} D_{H, \psi},
\end{equation}
with $\mathfrak{C}$, $S_{H}$, $\epsilon_{\chi, \chi'}^{H, \psi}$ and the (effective) divisors are as in $D_{H, \psi}$
as in \cite[(1.9), (2.1), (2.2)]{Pardini}.

\smallskip
\noindent Conversely, to any set of data $L_\chi, D_{H,\psi}$ satisfying \eqref{relation.ring.structure} and \eqref{relation.divisor.ring.structure} we can associate an abelian cover $\pi: \mathcal X \longrightarrow \mathcal Y$ with group $G$ in a natural way. If the cover so constructed is normal, $L_\chi, D_{H,\psi}$  is its building data (see \cite[Definition 2.1]{Pardini}). 

\noindent Moreover, if $\mathcal Y$ is complete, $L_\chi, D_{H,\psi}$ together with \eqref{relation.ring.structure} and \eqref{relation.divisor.ring.structure} determine the cover $\pi: \mathcal X \longrightarrow \mathcal Y$ up to isomorphism of abelian covers.
\end{theorem}

\noindent If $\pi$ is an abelian cover, the divisor 
\begin{equation}\label{eq.branch.divisor}
    D = \sum_{H \in \mathfrak{C}} \sum_{\psi \in S_{H}}  D_{H, \psi}
\end{equation}
is the total branch locus of $\pi$ (see \cite[page 193, (1.9) and Definition 2.2]{Pardini}).

\smallskip

\noindent Given data $L_\chi, D_{H,\psi}$ satisfying \eqref{relation.ring.structure} and \eqref{relation.divisor.ring.structure}, the abelian cover that, as claimed in Theorem~\ref{thm.Pardini}, can be associated to   $L_\chi, D_{H,\psi}$ (by defining an algebra structure on \eqref{eigenspaces}
as explained in the proof of \cite[Theorem 2.1]{Pardini}; see \cite[(2.8)]{Pardini}) is called the \emph{standard abelian cover} associated to $L_\chi, D_{H,\psi}$ (see \cite[Definition 2.2]{Pardini}). We recall now the characterization of standard abelian covers which are normal (see \cite[Corollary 3.1]{Pardini}):

\begin{theorem}\label{thm.normal}
    Let $\mathcal Y$ be a smooth, irreducible variety and let $\pi: \mathcal X \longrightarrow \mathcal Y$ be a standard abelian cover with building data $L_\chi, D_{H,\psi}$. Then $\mathcal X$ is normal if and only if the total branch locus $D$ $\mathrm{(}$see \eqref{eq.branch.divisor}$\mathrm{)}$ is a reduced divisor. 
\end{theorem}

\noindent We end this section by introducing the following notation that we will use thoughout the remaining of the paper:

\begin{notation}\label{notation.building.data}
 Let  $G$ be a finite abelian group and let $G=\mathbf{Z}_{n_1} \times \cdots \times  \mathbf{Z}_{n_l}$ be the invariant factor decomposition of $G$.
 For all $j=1, \dots, l$, let $\zeta_j$ be the primitive $n_j$-root of unity $$\zeta_j={e^{\frac{2\pi i}{n_j}}}.$$ 
 
 \begin{enumerate}
     \item  The elements of $G^*$ are the group homomorphisms from $G$ to $ \mathbb C^*$ of the form 
$$
\chi_{a_1a_2\dots a_l} :  (\bar x_1, \bar x_2, \dots \bar x_l)  \mapsto  \zeta_1^{a_1x_1}\zeta_2^{a_2x_2} \cdots \zeta_l^{a_lx_l},
$$ 
for all $a_j = 0,1, \dots, n_j-1$ and $j=1, \dots, l$.
 \end{enumerate}

 \smallskip
\noindent Let $\pi: \mathcal X \longrightarrow \mathcal Y$ be an abelian cover with group  $G$ 
 and let $a_j \in \{ 0,1, \dots, n_j-1\}$ for all $j=1, \dots, l$. 
  \smallskip

 \begin{enumerate}
     \item[(2)]  We will denote $L_{\chi_{a_1a_2\dots a_l}}$ by $L_{a_1a_2\dots a_l}$ (note that the notation $L_a$ and the notation $\mathcal L_1, \dots, \mathcal L_{n-1}$ 
     of Theorem~\ref{thm.splitting} (II) are different!).
     \item[(3)] Let $\sigma=(\bar b_1, \bar b_2, \dots, \bar b_l)$, with $b_j \in \{ 0,1, \dots, n_j-1\}$ for all $j=1, \dots, l$, be a  nonzero element of $G$ of order $\nu$ and let $H_\sigma=<\sigma>$. We will also denote $H_\sigma$ by $H_{b_1b_2 \dots, b_l}$. Let $\psi_\sigma$ be the generator of the group $H^*$ of characters of $H$ determined by $$\psi_\sigma(\sigma)=e^{\frac{2\pi i}{\nu}}.$$ We will also denote $\psi_\sigma$ by $\psi_{b_1b_2 \dots  b_l}$. The effective divisor $D_{H_\sigma,\psi_\sigma}$, introduced in \cite[(1.9)]{Pardini}, of the branch locus of $\pi$  will be denoted by $D_\sigma$ or $D_{b_1 b_2 \dots  b_l}$. Thus we index the divisors introduced in \cite[(1.9)]{Pardini} by the (nonzero) elements $\sigma$ of $G$.
 \end{enumerate}
\end{notation}

\begin{convention}
From now on, we use the letter $L$ (alone, or with some subscript) to denote line bundles {on $W$} and the letter $D$ (alone, or with some subscript) to denote \emph{effective} divisors {on $W$}.    
\end{convention}

\color{black}

\section{Abelian canonical covers up to degree $8$}\label{section.abelian.canonical}

In this section we characterize the building data of abelian canonical covers \linebreak $\pi: X \longrightarrow W$ (see Set-up and notation~\ref{setup.beginning} and Definition~\ref{defi.canonical.cover}; hence, in particular, $X$ has arbitrary dimension and $W$ is smooth), with $p_g(W)=0$ and degree $n$ of $\pi$ being $n \leq 8$ or $n$ prime. We also characterize those building data for which $X$ is smooth. Recall that we will use Notation~\ref{notation.building.data} throughout this section.

\subsection{Abelian canonical covers of prime degree} \

\smallskip
\noindent Even if the  main focus of the section is abelian covers  up to  degree $8$,  we start looking at abelian canonical covers of any prime degree $p$. We prove that these covers are simple cyclic (recall \cite[Definition 1.3]{Pardini}).  {The first case, $p=2$, is easy and well-known}:

\begin{remark}\label{thm.double}
   \textrm{If $p_g(W)=0$, by  Corollary~\ref{cor.234} (1), $L_1=\omega_W^*(1)$ and a canonical double cover (necessarily abelian, in fact, necessarily simple cyclic) of $W$ is determined by this and a choice of a reduced (respectively, smooth) divisor in $|\omega_W^{-2}(2)|$. Therefore, canonical double covers  $\pi: X \longrightarrow W$ (respectively, canonical double covers with $X$ smooth) exist if and only if  $|\omega_W^{-2}(2)|$ possesses reduced divisors (respectively, smooth divisors, e.g., if  $\omega_W^{-2}(2)$ is base-point-free). All this follows from the ramification formula} 
   and the local equation of a simple cyclic cover. 
\end{remark}

\noindent 
We will need the next lemma in the proof of  Theorem~\ref{thm.prime}.

\begin{lemma}\label{lemma.prime}
 Let $p$ be a prime number and let $\pi: \mathcal X \longrightarrow \mathcal Y$ be an abelian  cover 
with group $G=\mathbb Z_p$.   For $1 \le a, b \le p-1$, the exponent $i_{\chi_a}$, defined in \cite[(2.1)]{Pardini}, with respect to $(G, \psi_b)$  is the remainder of dividing $ab$ by $p$. 
\end{lemma}

\begin{proof}
Let $\zeta = e^{\frac{2\pi i}{p}}$. 
    Since $\psi_b(\bar b)=\zeta$, $\psi_b(\bar 1)=\zeta^{b'}$, where $b'$ is the inverse of $b$ modulo $p$.
    Let $c$ be the exponent $i_{\chi_a}$ with respect to $(G, \psi_b)$, i.e., $\chi_a=\psi_b^c$. Then $\zeta^a=\zeta^{b'c}$, so 
    \begin{equation*}
        a \equiv b' c \, (p).
    \end{equation*}
    Thus 
    \begin{equation*}
        c \equiv a b \, (p).
    \end{equation*}
\end{proof}

\begin{theorem}\label{thm.prime}
Let $p_g(W)=0$.
\begin{enumerate}
    \item If $\pi: X \longrightarrow W$ is an abelian canonical cover 
of prime degree $p \ge 3$, then $\pi$ is simple cyclic. Without loss of generality, we may assume

\begin{enumerate}
    \item[(i)] $L_{p-1}=\omega_W^*(1)$. 
\end{enumerate}

In this case, 

\begin{enumerate}
\item[(ii)] the total branch locus of $\pi$ is $D_1$; in particular, $D_2=\cdots = D_{p-1}=0$;
\item[(iii)] $D_1$ is reduced; 
\item[(iv)]  $L_1^{\otimes a}=L_a$ and  $L_1^{\otimes p} = \mathcal O_W(D_1)$; in particular $L_1^{\otimes p-1}=\omega_W^*(1) $; and 
\item[(v)] $h^0(\omega_W \otimes L_1^{\otimes a})=0$, 
for all $a=1, \dots, p-2$.
\end{enumerate}
   
\smallskip
\noindent If $X$ is smooth, then, in addition,

 \begin{enumerate}
\item[(iii')] {$D_{1}$ is  smooth.}
\end{enumerate}
    \smallskip
\item Conversely, if $p \ge 3$ is prime and $\{L_a,D_b\}$ $(a, b=1, \dots, p-1)$ satifies (i) to (v),   
then there exists a simple cyclic canonical cover 
$\pi: X \longrightarrow W$  of degree $p$ having $\{L_{a}, D_{b}\}$ as its building data with cyclic, with total branch locus $D_1$.

\item[] Moreover, if (iii') holds, then $X$ is smooth.
\end{enumerate}
\end{theorem}

\begin{proof}
First we prove (1). For any $1 \le a \le p-1$, there is an automorphism of $G^*$ sending $\chi_a$ and $\chi_{p-1}$, so, after possibly relabeling the $D_{b}$ by an automorphism of $G$ and $L_{a}$ by the inverse of its dual automorphism on $G^*$, we may assume $\omega_W^*(1)=L_{p-1}$. This proves (1) (i). 

\smallskip
\noindent Let $2 \le b \le p-1$. Then there is exist a (unique) $a$ with $1 \le a \le p-2$ such that $p-1$ is the remainder when dividing $a 
b$ by $p$. Then, for such $b$ and $a$, by Lemma~\ref{lemma.prime} the sum of the exponents $i_{\chi_a}$, $i_{\chi_{p-a-1}}$ with respect to the pair $(G,\psi_b)$ is not less than $p$. Therefore, by Theorem~\ref{thm.Pardini},
\begin{equation}\label{eq.D_b}
L_a \otimes L_{p-a-1} = L_{p-1} (D_b + D'),
\end{equation}
where $D'$ is a combination of $D_1, \dots, D_{b-1}, D_{b+1}, \dots, D_{p-1}$ whose coefficients are $0$ or $1$. On the other hand, since for $(G,\psi_1)$, $i_{\chi_a}=a$, the sum of the exponents $i_{\chi_a}$, $i_{\chi_{p-a-1}}$ with respect to $(G,\psi_1)$, for all $1 \le a \le p-2$, is $p-1$. This and Theorem~\ref{thm.Pardini} imply, for all $1 \le a \le p-2$
\begin{equation}\label{eq.D_1}
    L_a \otimes L_{p-a-1} = L_{p-1} (D''),
\end{equation}
where $D''$ is a combination of $D_2, \dots, D_{p-1}$ whose coefficients are $0$ or $1$. Then 
\eqref{eq.D_b} and \eqref{eq.D_1} imply 
\begin{equation}\label{eq.Z_p}
    (L_1 \otimes \dots \otimes L_{p-2})^{\otimes 2} = L_{p-1}^{\otimes p-2} (c_2D_2 + \cdots + c_{p-1}D_{p-1}), 
\end{equation}
where the coefficients $c_2, \dots, c_{p-1}$ are positive integers. Then Theorem~\ref{thm.splitting} (II) (4) and \eqref{eq.Z_p} imply 
\begin{equation*}
    c_2D_2 + \cdots + c_{p-1}D_{p-1} \sim 0, 
\end{equation*}
(note that $\mathcal L_{p-1} = L_{p-1}$ in this case). Since $D_2, \dots, D_{p-1}$ are effective and $W$ is projective, this implies
\begin{equation*}
    D_2 = \cdots = D_{p-1} = 0. 
\end{equation*}
Then the total branch divisor of $\pi$ is $D_1$ so $\pi$ is a simple cyclic cover and (1) (ii) is proved. 

\smallskip
\noindent
Since the sum of the exponents $i_{\chi_1}$, $i_{\chi_a}$  with respect to the pair $(G,\psi_1)$, for all $1 \le a \le p-1$, is $a+1$, Theorem~\ref{thm.Pardini} and (1) (ii) imply 
\begin{equation*}
   L_1 \otimes L_{a} = L_{a+1},  
\end{equation*}
if $1 \le a \le p-2$ and
\begin{equation*}
   L_1 \otimes L_{p-1} = \mathcal O_W(D_1).
\end{equation*}
Therefore, 
\begin{equation*}
   L_1^{\otimes a} = L_a
\end{equation*}
and 
\begin{equation*}
   L_1^{\otimes p} = \mathcal O_W(D_1).
\end{equation*}
In particular, 
\begin{equation*}
   L_1^{\otimes p-1} = \omega_W^*(1).
\end{equation*}
This completes the proof of (1) (iv). 

\smallskip
\noindent
Parts (iii) and (iii') of (1) follow from the local equation of a simple cyclic cover (see e.g.,  \cite[Section I.17]{BHPV}) and Part (v)  of (1) follows
from Theorem~\ref{thm.splitting} (I) (3).

\smallskip
\noindent Now we prove (2). If (i), (ii) and (iv) hold, 
then there exists an abelian cover of degree $p$ 
with building data  $\{L_{a}, D_{b}\}$, which is simple cyclic 
with total branch locus $D_1$. The ramification formula (see e.g.  \cite[Lemma I.17.1 (iii)]{BHPV}) and (1) (iv) imply 
\begin{equation*}
    \omega_X = \pi^*(\omega_W \otimes L_1^{\otimes p-1}) = \pi^*\mathcal O_W(1). 
\end{equation*}
Then the projection formula together with (1) (iv) and (1) (v) implies that $\pi$ is a canonical cover, since the local equation of a simple cyclic cover (see e.g.  \cite[Section I.17]{BHPV}) and  (1) (iii) imply that $X$ is normal and locally Gorenstein. Finally, because of the local equation of a simple cyclic cover, (1) (iii') implies $X$ is smooth.
\end{proof}

\subsection{Abelian canonical covers with group $\mathbf{Z}_2 \times \mathbf{Z}_2$}\label{section.abelian.bidouble} \ 

\smallskip
\noindent We start this subsection making explicit, for a group $\mathbb Z_2^r$, the algebra structure \eqref{relation.ring.structure}. This will be also used  in Subsection~\ref{section.abelian.tridouble}.

\begin{prop}\label{prop.alg.struct.bidouble}
  If $\pi: \mathcal X \longrightarrow \mathcal Y$ is an abelian cover with group $\mathbf{Z}_2^r$, $r \ge 1$, then \eqref{relation.ring.structure} is explicitly
 \begin{equation}\label{eq.bidouble.alg.str}
   L_{\chi} \otimes  L_{\chi'} = L_{\chi\chi'}\left (\sum D_{\sigma}\right ),  
 \end{equation}
  where the sum on the right hand side of \eqref{eq.bidouble.alg.str} ranges on  all the $\sigma$ of  $\mathbf{Z}_2^r$ such that $\chi(\sigma) = \chi'(\sigma) = -1$. 
  \end{prop}

\noindent After a routine computation, Proposition~\ref{prop.alg.struct.bidouble} yields  the following corollary:

\begin{corollary}
    If $\pi: \mathcal X \longrightarrow \mathcal Y$ is an abelian cover with group $\mathbf{Z}_2 \times \mathbb{Z}_2$, then we have, among others,   these relations:

    \begin{equation}\label{eq.2L_10}
   L_{10} \otimes  L_{10} =  \mathcal O_W(D_{10}+D_{11})
     \end{equation}
\begin{equation}\label{relations.ring.structure.Z_2^2}
   L_{10} \otimes  L_{01} =  L_{11}(D_{11})
     \end{equation}
     \begin{equation}\label{eq.2L_11}
         \hskip 1.5cm  L_{11} \otimes  L_{11}  \ \, =  \ \, \mathcal O_W(D_{10} + D_{01}). 
     \end{equation}
\end{corollary}

\begin{theorem}\label{thm.bidouble.alg.str}
Let $p_g(W)=0$. 

\begin{enumerate}
    \item If $\pi: X \longrightarrow W$ is an abelian canonical cover with group $\mathbf{Z}_2 \times \mathbf{Z}_2$, then, without loss of generality, we may assume  

\begin{enumerate}
    \item[(i)] $L_{11} = \omega_{W}^*(1)$.
\end{enumerate}
 In this case, we have 
 \begin{enumerate}
 \item[(ii)] $D_{11} = 0$;  
 \item[(iii)] {$D_{10} + D_{01}$ is reduced;}
 \item[(iv)] the $L_{a_1a_2}$ and the $D_{b_1b_2}$ satisfy \eqref{eq.bidouble.alg.str}; and  
 \item[(v)]
 $h^0(\omega_W \otimes L_{10})=h^0(\omega_W \otimes L_{01})=0$.
\end{enumerate}

\smallskip

If $X$ is smooth, then, in addition, 


\begin{enumerate}
 \item[(iii')] 
 {$D_{10}, D_{01}$ are smooth  and meet transversally.}

\end{enumerate}
\smallskip

\item Conversely, let $\{L_{a_1a_2}, D_{b_1b_2}\}$ $(a_1, a_2, b_1, b_2=0, 1; (a_1,a_2), (b_1, b_2) \neq (0,0))$  satisfy (i) to (v). Let
  $\pi: X \longrightarrow W$ be the  abelian  cover having $\{L_{a_1a_2}, D_{b_1b_2}\}$ as its building data and with algebra structure determined by \eqref{eq.bidouble.alg.str} and (ii). 

   \smallskip
  
  \begin{enumerate}
      \item If $X$ is locally Gorenstein, then $\pi$ is an abelian canonical cover with group $\mathbf{Z}_2 \times \mathbf{Z}_2$.

      \item If (iii') holds, then $\pi$ is an abelian canonical cover with group $\mathbf{Z}_2 \times \mathbf{Z}_2$ and $X$ is smooth.
  \end{enumerate} 
\end{enumerate}
\end{theorem}

\begin{proof}
First we prove (1), so we 
assume  that $\pi: X \longrightarrow W$ is an abelian canonical cover with group $\mathbf{Z}_2 \times \mathbf{Z}_2$. 
There are automorphisms of $G^*$ swapping any two order $2$ elements of $G$, so, after possibly relabeling the $D_{b_1b_2}$ by an automorphism of $G$ and $L_{a_1a_2}$ by the inverse of its dual automorphism on $G^*$, by Theorem~\ref{thm.splitting} (II) (1), we may assume without loss of generality that $L_{11}=\omega_W^*(1)$.
 Then, Theorem~\ref{thm.splitting} (II) (4) implies 
\begin{equation*}
 L_{10}^{\otimes 2} \otimes L_{01}^{\otimes 2} \ = L_{11}^{\otimes 2}.
\end{equation*}
Then, from \eqref{relations.ring.structure.Z_2^2} it follows that
$2D_{11} \sim 0$. 
Since  $D_{11}$ is effective and $W$ is projective, we get 
$D_{11} = 0$.
This finishes the proof of (ii). Part (iii) and (iii') follow from Theorem~\ref{thm.normal} and \cite[Proposition 3.1]{Pardini}, Part (iv) follows from Proposition~\ref{prop.alg.struct.bidouble} and Part (v) follows from Theorem~\ref{thm.splitting} (II) (3).

\smallskip
\noindent Assume now the hypotheses of (2). First we prove (a). By Theorems~\ref{thm.Pardini} and \ref{thm.normal} and (1) (iii), the $L_{a_1a_2}$ and the $D_{b_1b_2}$ are indeed the building data of an abelian cover 
$\pi: X \longrightarrow W$ with algebra structure determined by \eqref{eq.bidouble.alg.str} and (ii) and $X$ normal. 
To see that $\pi$ is a canonical cover we argue similarly to the proof of Theorem~\ref{thm.tridouble.alg.str} (2), taking now into account that the  ramification formula, (i), (ii) and \eqref{eq.2L_11} yield 
\begin{equation*}
    \omega_X^{\otimes 2} = \pi^*(\omega_W^{\otimes 2} \otimes \mathcal O_W(D_{10} + D_{01} ))= \pi^*(\mathcal O_W(2))
\end{equation*}
and that (1) (i), (1) (ii) and \eqref{relations.ring.structure.Z_2^2} imply \eqref{eq.self.duality}.

\smallskip
\noindent
Now we prove (b). 
Since $H_{10}$ and $H_{01}$ satisfy condition iii) a) of  \cite[Proposition 3.1]{Pardini}, this proposition and (iii') imply $X$ is smooth and, in particular, locally Gorenstein, so we conclude by (a).
\end{proof}

\subsection{Abelian canonical covers with group $\mathbf{Z}_4$}\label{section.abelian.Z4} \

\medskip
\noindent We start with subsection with a lemma that will be also used in Subsections~\ref{section.abelian.degree.6} and \ref{subsection.abelian.Z8}:

\begin{lemma}~\label{lemma.congruence}
    Let $G=\mathbb{Z}_n$, $n \ge 1$ and let $b \in \{1, \dots, n-1\}$. Consider $(H_b,\psi_b)$ and let $\nu = |H_b|$. If $\psi_b=\chi_c|_{H_b}$, then, for any $a \in \{1, \dots, n-1\}$ the exponent $i_{\chi_a}$ with respect to $(H_b,\psi_b)$, as defined in \cite[(2.1)]{Pardini},  satisfies 
    \begin{equation}\label{equ.congruence}
        i_{\chi_a} \equiv a c' (\nu),
    \end{equation}
  where $c'$ is the inverse of $c$ modulo $\nu$.  Moreover, if 
$n=2, 3, 4, 6, 8, 12, 24$, then 
 \begin{equation}\label{eq.congruence2}
        i_{\chi_a} \equiv a c \, (\nu).
    \end{equation}
\end{lemma}

\begin{proof}
We have 
\begin{equation*}
    \chi_a|_{H_b} = \psi_b^{i_{\chi_a}} = (\chi_c|_{H_b})^{i_{\chi_a}}, 
\end{equation*}
which is equivalent to
\begin{equation*}
    \chi_a(\bar b) =  \chi_c(\bar b)^{i_{\chi_a}}
\end{equation*}
By Notation~\ref{notation.building.data}, this is equivalent to
\begin{equation*}
    \zeta^a=\zeta^{ci_{\chi_a}},
\end{equation*}
where $\zeta$ is a primitive $\nu$-root of unity. This is equivalent to 
\begin{equation*}
  a \equiv   c \, i_{\chi_a} (\nu) 
    \end{equation*}
    and, hence, to \eqref{equ.congruence}. If $n=2, 3, 4, 6, 8, 12, 24$, then $\nu \in \{2, 3, 4, 6, 8, 12, 24\}$. 
     In this case,  $c'=c$, so  \eqref{equ.congruence} becomes \eqref{eq.congruence2}.
\end{proof}

\begin{prop}
    If $\pi: \mathcal X \longrightarrow \mathcal Y$ is an abelian cover with group $\mathbf{Z}_4$, then we have, among others,   these relations:
   \begin{equation}\label{eq.Z_4.2L_1}
   L_{1} \otimes  L_{1} =  L_{2}(D_2+D_3)
     \end{equation}
     \begin{equation}\label{eq.Z_4.L_1L_2}
        L_{1} \otimes  L_{2}   =   L_3(D_3)
     \end{equation}
         \begin{equation}\label{eq.Z_4.L_1L_3}
        L_{1} \otimes  L_{3}   =   \mathcal O_W(D_1+D_2+D_3)
     \end{equation}
     \begin{equation}\label{eq.Z_4.2L_2}
   L_{2} \otimes  L_{2} =  \mathcal O_W(D_1+D_3)
     \end{equation}
      \begin{equation}\label{eq.Z_4.2L_3}
   L_{3} \otimes  L_{3} =  L_{2}(D_1+D_2)
     \end{equation}
\end{prop}

\begin{proof}
    The above relations are a subset of those we obtain when we explicitly write, for the group $\mathbb Z_4$, the relations
    in \eqref{relation.ring.structure}.
    In order to do so, we use \eqref{eq.congruence2} of Lemma~\ref{lemma.congruence} to compute the coefficients of \eqref{relation.divisor.ring.structure}.  
\end{proof}

\begin{theorem}\label{thm.Z4.alg.str}
Let $W$ be such that $p_g(W)=0$.

\begin{enumerate}
    \item If  $\mathrm{Pic} \, W$ has no $2$-torsion or $\omega_W^{\otimes 2}$ is not very ample (e.g, if $W$
  is  not of general type) and $\pi: X \longrightarrow W$ is an abelian canonical cover with group $\mathbf{Z}_4$, then, without loss of generality, we may assume  

\begin{enumerate}
    \item[(i)] $L_{3} = \omega_{W}^*(1)$.
\end{enumerate}
 In this case, we have 
 \begin{enumerate}
 \item[(ii)] $D_{3} = 0$;   
 \item[(iii)] $D_1+D_2$ is a reduced divisor;
 \item[(iv)] the $L_{a}$ and the $D_{b}$ satisfy \eqref{relation.ring.structure}; and  
\item[(v)] $h^0(\omega_W \otimes L_1)=h^0(\omega_W \otimes L_2)=0$.
\end{enumerate}

\smallskip
 \item[] If $X$ is smooth, then, in addition, 
\smallskip

\begin{enumerate}
  \item[(iii')] 
  {$D_{1}, D_{2}$ are smooth  and disjoint.}
\end{enumerate}

\smallskip
\item Conversely, let $\{L_{a}, D_{b}\}$ $(a, b=1,2,3)$ satisfy (i) to (v) 
and let \linebreak
  $\pi: X \longrightarrow W$ be the  abelian  cover having $\{L_{a}, D_{b}\}$ as its building data and with algebra structure determined by \eqref{relation.ring.structure} and (ii). 
  \smallskip
  
  \begin{enumerate}
      \item If $X$ is locally Gorenstein, then $\pi$ is an abelian canonical cover with group $\mathbf{Z}_4$.

      \item If (iii') holds, then $\pi$ is an abelian canonical cover with group $\mathbf{Z}_4$ and $X$ is smooth.
  \end{enumerate} 
\end{enumerate}
\end{theorem}

\begin{proof}
   Let us prove (1). By \eqref{formula.splitting}, 
\begin{equation}\label{eq.character.omega.Z4}
    \omega_W^*(1)=L_{a}
\end{equation}
 for some $a\in \{1,2,3\}$. 

 \smallskip

\noindent
To prove (1) (i), we see first that the subindex $a$ in \eqref{eq.character.omega.Z4} cannot be $2$. Suppose the contrary. Then $\omega_W^*(1)=L_{2}$.  By Theorem~\ref{thm.splitting} (II) (4), \eqref{eq.Z_4.L_1L_3} and \eqref{eq.Z_4.2L_2},
\begin{equation*}
  \mathcal O_W(D_1+D_3)=L_2^{\otimes 2} = L_1^{\otimes 2} \otimes L_3^{\otimes 2} =  \mathcal O_W(2D_1+2D_2+2D_3).
\end{equation*}
Therefore
\begin{equation*}
    D_1+2D_2+D_3 \sim 0.
\end{equation*}
Then, since $W$ is projective, 
\begin{equation*}
D_1 =  D_2 = D_3 =  0.
\end{equation*}
Then \eqref{eq.Z_4.2L_2} and our assumption implies 
\begin{equation*}
   \omega_W^{-2}(2) = L_2^{\otimes 2} = \mathcal O_W.
\end{equation*}
This implies $\omega_W^{\otimes 2}$ is  very ample, hence, by hypothesis, Pic$W$ has no $2$-torsion. Then $L_2=\mathcal O_W$, so \eqref{formula.splitting} implies $h^0(\mathcal O_X) \ge 2$, so $X$ would not be connected, which is a contradiction. 

\smallskip
\noindent Then, in \eqref{eq.character.omega.Z4}, $a=1$ or $a=3$. 
There is an automorphism of $G^*$ swapping the two generators of $G^*$, so, after possibly relabeling the $D_{b}$ by an automorphism of $G$ and the $L_{a}$ by the inverse of its dual automorphism on $G^*$, by Theorem~\ref{thm.splitting} (II) (1), we may assume without loss of generality that $L_{3}=\omega_W^*(1)$, so (1) (i) is proved. 

\smallskip
\noindent
Now, Theorem~\ref{thm.splitting} (II) (5) implies 
\begin{equation*}
 L_{1} \otimes L_2  = L_{3}.
\end{equation*} 
Then, from \eqref{eq.Z_4.L_1L_2} it follows that
$2D_{3} \sim 0$. 
Since  $D_{3}$ is effective and $W$ is projective, we get 
$D_{3} = 0$, so (1) (ii) is proved. Parts (iii) and (v) of (1) are proved as in the proof of Theorem~\ref{thm.bidouble.alg.str} and Part (iv) of (1) follows from Theorem~\ref{thm.Pardini}. Finally, since $H_1$ and $H_2$ do not satisfy condition iii) a) of \cite[Proposition 3.1]{Pardini}, if $X$ is smooth, this proposition implies (1) (iii').

\smallskip
\noindent Assume now the hypotheses of (2). To prove (a), arguing as is the proof of Theorem~\ref{thm.bidouble.alg.str} (2), we see that $\{L_{a},D_{b}\}$ is the building data of an abelian cover
$\pi: X \longrightarrow W$ with algebra structure determined by \eqref{relation.ring.structure} and (ii).
To see that $\pi$ is a canonical cover we argue similarly to the proof of Theorem~\ref{thm.tridouble.alg.str} (2), taking now into account that the  ramification formula, (i), (ii), \eqref{eq.Z_4.2L_2} and \eqref{eq.Z_4.2L_3} yield 
\begin{equation*}
    \omega_X^{\otimes 4} = \pi^*(\omega_W^{\otimes 4} \otimes \mathcal O_W(3D_{1} + 2D_{2} ))= \pi^*(\mathcal O_W(4))
\end{equation*}
and that (1) (i), (1) (ii) and \eqref{eq.Z_4.L_1L_2} imply \eqref{eq.self.duality}.

\smallskip
\noindent
Now we prove (b). 
  \cite[Proposition 3.1]{Pardini} and (iii') imply $X$ is smooth and, in particular, locally Gorenstein, so we conclude by (a).
\end{proof}

\subsection{Abelian canonical covers of degree $6$}\label{section.abelian.degree.6} \

\smallskip

\noindent We start this subsection making explicit, for the group $\mathbb Z_6$, those relations in \eqref{relation.ring.structure} that we will use in the remaining of the subsection:

\begin{prop}\label{prop.alg.struct.cyclic6}
  If $\pi: \mathcal X \longrightarrow \mathcal Y$ is an abelian cover with group $\mathbf{Z}_6$, then its building data satisfy, among others, the following relations: 
  \begin{equation}\label{eq.alg.str.Z6.L_1L_1}
   L_1 \otimes L_1 = L_2(D_3+D_4+D_5)
\end{equation}
\begin{equation}\label{eq.alg.str.Z6.L_1L_2}
    L_1 \otimes L_2 = L_3(D_2+D_4+D_5)
\end{equation}
\begin{equation}\label{eq.alg.str.Z6.L_1L_4}
   L_1 \otimes L_4 = L_5(D_4+D_5)
\end{equation}
\begin{equation}\label{eq.alg.str.Z6.L_1L_5}
   L_1 \otimes L_5 = \mathcal O_W(D_1+D_2+D_3+D_4+D_5)
\end{equation}
\begin{equation}\label{eq.alg.str.Z6.L_2L_2}
   L_2 \otimes L_2 = L_4(D_2+D_5)
\end{equation}
\begin{equation}\label{eq.alg.str.Z6.L_2L_3}
L_2 \otimes L_3 = L_5(D_{5})
\end{equation}
\begin{equation}\label{eq.alg.str.Z6.L_2L_4}
L_2\otimes L_4
=\mathcal{O}_W(D_1+D_2+D_4+D_5)
\end{equation}
\begin{equation}\label{eq.alg.str.Z6.L_3L_3}
L_{3}\otimes L_3 = \mathcal{O}_{W}(D_{1} + D_{3} + D_{5})
\end{equation}
\begin{equation}\label{eq.alg.str.Z6.L_3L_5}
   L_3 \otimes L_5 = L_2(D_1+D_3)
\end{equation}
\begin{equation}\label{eq.alg.str.Z6.L_4L_4}
   L_4 \otimes L_4 = L_2(D_1+D_4)
\end{equation}
\begin{equation}\label{eq.alg.str.Z6.L_4L_5}
   L_4 \otimes L_5 = L_3(D_1+D_2+D_4)
\end{equation}
\begin{equation}\label{eq.alg.str.Z6.L_5L_5}
   L_5 \otimes L_5 = L_4(D_1 + D_2 + D_3)
\end{equation}
\end{prop}

\begin{proof}
The above relations are a subset of those we obtain when we explicitly write, for the group $\mathbb Z_6$, the relations
    in \eqref{relation.ring.structure}.
    In order to do so, we use \eqref{eq.congruence2} of Lemma~\ref{lemma.congruence} to compute the coefficients of \eqref{relation.divisor.ring.structure}. 
\end{proof}

\begin{theorem}\label{thm.Z6.alg.str}
  Let $W$ be such that $p_g(W)=0$.

  \smallskip

  \begin{enumerate}
    \item If $\mathrm{Pic} \, W$ is free of $2$-torsion and $3$-torsion and $\pi: X \longrightarrow W$ is an abelian canonical cover with group $G=\mathbb{Z}_6$, then, without loss of generality, we may assume  

\begin{enumerate}
    \item[(i)] $L_{5} = \omega_{W}^{*}(1)$.
\end{enumerate}
 In this case, we have 
 \begin{enumerate}
 \item[(ii)] 
 $D_{4} = D_{5} = 0$;    
 \item[(iii)] $D_{1}+D_{2}+D_{3}$ is a reduced divisor; 
 \item[(iv)] the $L_{a}$ and the $D_{b}$ satisfy \eqref{relation.ring.structure}; and 
 \item[(v)] $h^{0}(\omega_{W} \otimes L_{a}) = 0$ for all $a \neq 5$. 
 \end{enumerate}

\smallskip
 \item[] If $X$ is smooth, then, in addition, 
 \begin{enumerate}
     \item[(iii')] $D_{1}, D_{2}, D_{3}$ are smooth; $D_{1}$ meets neither  $D_{2}$ nor $D_{3}$;   and $D_{2}$ and $D_{3}$ meet transversely. 
 \end{enumerate}

\smallskip

\item Conversely, let $\{L_{a}, D_{b}\}$ $(a, b=1, \dots, 5)$ satisfy (i) to (v)
and let \linebreak
  $\pi: X \longrightarrow W$ be the  abelian  cover having $\{L_{a}, D_{b}\}$ as its building data and with algebra structure determined by \eqref{relation.ring.structure} and (ii). 
  
    \begin{enumerate}
      \item If $X$ is locally Gorenstein, then $\pi$ is an abelian canonical cover with group $\mathbf{Z}_6$.

      \item If (iii') holds, then $\pi$ is an abelian canonical cover with group $\mathbf{Z}_6$ and $X$ is smooth.
  \end{enumerate} 
\end{enumerate} 
\end{theorem}

\begin{proof}
 By \eqref{formula.splitting},  \begin{equation}\label{eq.character.omega.Z6}
    \omega_W^*(1)=L_{a}
\end{equation} for some $a \in \{1, \dots, {5}\}$.
We divide the proof of (1) (i) into several steps:
\smallskip

\noindent 
\underline{\emph{Step 1}}: The subindex $a$ in \eqref{eq.character.omega.Z6} cannot be ${2}$ or ${4}$.

\noindent 
Suppose the contrary. There are automorphisms of $G^*$ swapping $\chi_2$ and $\chi_4$, so, after possibly relabeling the $D_{b}$ by an automorphism of $G$ and $L_{a}$ by the inverse of its dual automorphism on $G^*$, we may assume $\omega_W^*(1)=L_{2}$. This assumption, together with Theorem~\ref{thm.splitting} (II) (4), 
\eqref{eq.alg.str.Z6.L_1L_1}, \eqref{eq.alg.str.Z6.L_3L_5} and \eqref{eq.alg.str.Z6.L_4L_4}, yields
\begin{equation*}
 L_2^{\otimes 4} =  L_1^{\otimes 2}\otimes L_3^{\otimes 2}\otimes L_4^{\otimes 2}\otimes L_5^{\otimes 2} = L_2^{\otimes 4}(3D_1 +3D_3+2D_4+3D_5),
\end{equation*}
which gives 
$$
3D_1+3D_3+2D_4+
{3}D_5 \sim 0.
$$
Since the $D_b$ are effective and $W$ is projective, this implies
$$
D_1=D_3=D_4=D_5=0.
$$
Then  \eqref{eq.alg.str.Z6.L_3L_3}
gives
$
L_3^{\otimes 2}=\mathcal{O}_W$.
Since $\mathrm{Pic}(W)$ has no $2$-torsion, 
we get $L_3=\mathcal{O}_W$, so 
$h^0(\mathcal{O}_X)\geq 2
$, 
which contradicts $X$ being connected. 

\smallskip

\noindent
\underline{\emph{Step 2}}: The subindex $a$ in \eqref{eq.character.omega.Z6} cannot be $3$.

\noindent Suppose the contrary. This assumption, together with Theorem~\ref{thm.splitting} (II) (5), 
\eqref{eq.alg.str.Z6.L_1L_2} and \eqref{eq.alg.str.Z6.L_4L_5}, yields 
$$L_3^{\otimes 2}=L_1\otimes L_2\otimes L_4\otimes L_5 = L_3^{\otimes 2}(D_1+2D_2+2D_4+D_5),$$
which gives 
$$D_1 + 2D_2 + 2D_4 + D_5 \sim 0.$$
Arguing as in Step 1, we get
$$
D_1=D_2=D_4=D_5=0.
$$
This together with \eqref{eq.alg.str.Z6.L_2L_2} and 
\eqref{eq.alg.str.Z6.L_2L_4} implies $L_2^{\otimes 3}=\mathcal O_W$.
Since $\mathrm{Pic}(W)$ has no $3$-torsion, we get
$h^0(\mathcal{O}_X)\geq 2
$,
which contradicts $X$ being connected. 

\smallskip
\noindent
\underline{\emph{Step 3}}:

\noindent There is an automorphism of $G^*$ swapping the two generators of $G^*$, so, after possibly relabeling the $D_{b}$ by an automorphism of $G$ and the $L_{a}$ by the inverse of its dual automorphism on $G^*$,  we may assume, without loss of generality, that $L_{5}=\omega_W^*(1)$. This completes the proof of 
(1) (i).

\smallskip
\noindent Now, Theorem~\ref{thm.splitting} (I) (5), \eqref{eq.alg.str.Z6.L_1L_4} and \eqref{eq.alg.str.Z6.L_2L_3} imply
 \begin{equation*}
     L_5^{\otimes 2}= L_1 \otimes L_2 \otimes L_3 \otimes L_4 = L_5^{\otimes 2}(D_4+2D_5).
 \end{equation*}
 This implies 
 \begin{equation*}
  D_4+2D_5 \sim 0.   
 \end{equation*}
 Since $D_4$ and $D_5$ are effective and $W$ is projective, 
  \begin{equation*}
  D_4=D_5=0,   
 \end{equation*}
 so (1) (ii) is proved. Parts (iii), (iv) and (v) of (1) are proved as in the proof of Theorem~\ref{thm.Z4.alg.str}. Since neither $H_1, H_2$ nor $H_1, H_3$ fulfill condition iii) a) of 
 \cite[Proposition 3.1]{Pardini}, if $X$ is smooth, then this proposition implies (1) (iii').

 \smallskip
\noindent Assume now the hypotheses of (2). To prove (a), arguing as is the proof of Theorem~\ref{thm.bidouble.alg.str} (2), we see that $\{L_{a},D_{b}\}$ is the building data of an abelian cover
$\pi: X \longrightarrow W$ with algebra structure determined by \eqref{relation.ring.structure} and (ii).
To see that $\pi$ is a canonical cover we argue similarly to the proof of Theorem~\ref{thm.tridouble.alg.str} (2), taking now into account that the  ramification formula, (1) (i), (2) (ii), \eqref{eq.alg.str.Z6.L_2L_4}, \eqref{eq.alg.str.Z6.L_4L_4} and \eqref{eq.alg.str.Z6.L_5L_5} yield 
\begin{equation*}
 \omega_{X}^{\otimes 6} = \pi^{*}(\omega_{W}^{\otimes 6} \otimes \mathcal O_W(5D_{1}+4D_{2}+3D_{3})) = \pi^{*}(\mathcal{O}_{W}(6))
\end{equation*}
and that (1) (i), (1) (ii),  \eqref{eq.alg.str.Z6.L_1L_4} and \eqref{eq.alg.str.Z6.L_2L_3} imply \eqref{eq.self.duality}. 

\smallskip
\noindent
Now we prove (b). 
Since $H_2$ and $H_3$ satisfy condition iii) a) of \cite[Proposition 3.1]{Pardini}, this proposition and (iii') imply $X$ is smooth  and, in particular, locally Gorenstein, so we conclude by (a).
\end{proof}

\subsection{Abelian canonical covers with group $\mathbf{Z}_2 \times \mathbf{Z}_2 \times  \mathbf{Z}_2$}\label{section.abelian.tridouble} \

\smallskip

\noindent We start this subsection making explicit, for the group $\mathbb Z_2 \times \mathbb Z_2 \times \mathbb Z_2$, those relations in Proposition~\ref{prop.alg.struct.bidouble} that we will use in the remaining of the subsection:

\begin{corollary}\label{cor.alg.struct.tridouble}
  If $\pi: \mathcal X \longrightarrow \mathcal Y$ is an abelian cover with group $\mathbf{Z}_2 \times \mathbf{Z}_2 \times  \mathbf{Z}_2$, then, among others, we have  these relations:
     \begin{equation}\label{relations.ring.structure.Z_2^3}
     \begin{matrix}
   L_{100} \otimes  L_{011} & = &  L_{111}(D_{110} + D_{101})\\
L_{010} \otimes L_{101} & = & L_{111}(D_{110} + D_{011}) \\
L_{001} \otimes L_{110} & = & L_{111}(D_{011} + D_{101})   
     \end{matrix}
     \end{equation}
     and 
     \begin{equation}\label{eq.2L_111}
         \hskip 1.5cm  L_{111} \otimes  L_{111}  \ \, =  \ \, \mathcal O_W(D_{100} + D_{010} + D_{001} + D_{111}). 
     \end{equation}   
\end{corollary}

\begin{proof}
    The above relations follow from Proposition~\ref{prop.alg.struct.bidouble} after a routine computation.
\end{proof}

\begin{theorem}\label{thm.tridouble.alg.str}
Let $p_g(W)=0$. 

\begin{enumerate}
    \item If $\pi: X \longrightarrow W$ is an abelian canonical cover with group $\mathbf{Z}_2 \times \mathbf{Z}_2 \times  \mathbf{Z}_2$, then, without loss of generality, we may assume
    \smallskip

\begin{enumerate}
    \item[(i)] $L_{111} = \omega_{W}^*(1)$.
\end{enumerate}
\smallskip

\noindent In this case, we have 
\smallskip

 \begin{enumerate}
 \item[(ii)] $D_{110} = D_{101} = D_{011} = 0$;     
 \item[(iii)] $D_{100}+D_{010}+D_{001}+D_{111}$ is a reduced divisor;  
 \item[(iv)] the $L_{a_1a_2a_3}$ and the $D_{b_1b_2b_3}$ satisfy \eqref{eq.bidouble.alg.str}; 
\item[(v)] $h^0(\omega_{W} \otimes L_{a_1a_2a_3})=0$ for any $(a_1,a_2,a_3) \neq (1,1,1)$.  
\end{enumerate}
\smallskip

\noindent 
 If $X$ is smooth, then we also have
 \smallskip

\begin{enumerate}
 \item[(iii')] 
 {$D_{100}, D_{010}, D_{001}, D_{111}$ are smooth,   $D_{100} \cap D_{010} \cap D_{001} \cap D_{111}= \emptyset$  and $D_{100} +D_{010} + D_{001} + D_{111}$ is a simple normal crossing divisor.}
\end{enumerate}

\smallskip
\item Conversely, let $\{L_{a_1a_2a_3}, D_{b_1b_2b_3}\}$ $(a_1, a_2, a_3, b_1, b_2, b_3=0, 1; (a_1,a_2,a_3),$ $(b_1,b_2,b_3) \neq (0,0,0))$ satisfy (i) to (v) above 
and let
  $\pi: X \longrightarrow W$ be the  abelian  cover having $\{L_{a_1a_2a_3}, D_{b_1b_2b_3}\}$ as its building data and with algebra structure determined by \eqref{eq.bidouble.alg.str} and (ii).
  \begin{enumerate}
      \item If $X$ is locally Gorenstein, then $\pi$ is an abelian canonical cover with group $\mathbf{Z}_2 \times \mathbf{Z}_2 \times \mathbf{Z}_2$.

      \item If (iii') holds, then $\pi$ is an abelian canonical cover with group $\mathbf{Z}_2 \times \mathbf{Z}_2 \times \mathbf{Z}_2$ and $X$ is smooth.
  \end{enumerate} 
\end{enumerate}
\end{theorem}

\begin{proof}
First we prove (1), so we 
assume  that $\pi: X \longrightarrow W$ is an abelian canonical cover with group $\mathbf{Z}_2 \times \mathbf{Z}_2 \times  \mathbf{Z}_2$. 
There are automorphisms of $G^*$ swapping any two order $2$ elements of $G^*$, so, after possibly relabeling the $D_{b_1b_2b_3}$ by an automorphism of $G$ and $L_{a_1a_2a_3}$ by the inverse of its dual automorphism on $G^*$, by Theorem~\ref{thm.splitting} (II) (1), we may assume without loss of generality that $L_{111} = \omega_{W}^*(1)$. This proves (1) (i).

\smallskip
\noindent Theorem~\ref{thm.splitting} (II) (4) implies 
\begin{equation*}
L_{100}^{\otimes 2} \otimes L_{010}^{\otimes 2} \otimes L_{001}^{\otimes 2} \otimes L_{110}^{\otimes 2} \otimes L_{101}^{\otimes 2} \otimes L_{011}^{\otimes 2} = L_{111}^{\otimes 6}.
\end{equation*}
From \eqref{relations.ring.structure.Z_2^3} it follows that
\begin{equation*}
4(D_{110} + D_{101} +   D_{011}) \sim 0. 
\end{equation*}
Since the $D_{a_1a_2a_3}$ are effective and $W$ is projective, we get 
\begin{equation*}
    D_{011} = D_{101} = D_{110} = 0.
\end{equation*}
This finishes the proof of (1) (ii). 
Parts (iii), (iv) and (v) are proved arguing as in the proof of Theorem~\ref{thm.bidouble.alg.str}.

\smallskip
\noindent
Finally, if $X$ is smooth, then $D_{100}, D_{010}, 
{D_{001}}, D_{111}$ 
are smooth outside their intersections by \cite[Proposition 3.1.ii)]{Pardini}; the divisor $D_{100} + D_{010}+ 
{D_{001}}+ D_{111}$ is a simple crossing divisor by \cite[Proposition 3.1.iii) b)]{Pardini} and, in particular, $D_{100}, D_{010}, 
{D_{001}}, D_{111}$ are smooth at their intersections; and, since the 
{quadruplet} $H_{100}$, 
{$H_{010}$},
${H_{001}}$, $H_{111}$ does not satisfy \cite[Proposition 3.1.iii) a)]{Pardini}, there cannot be points in $D_{100}  \cap 
{D_{010}}  %
\cap {D_{001}} \cap D_{111}$.

\smallskip
\noindent Assume now the hypotheses of (2). Then, by Theorems~\ref{thm.Pardini} and \ref{thm.normal}, the $L_{a_1a_2a_3}$ and the $D_{b_1b_2b_3}$ are the building data of an abelian cover 
$\pi: X \longrightarrow W$, with $X$ normal, 
 whose algebra structure is determined by \eqref{eq.bidouble.alg.str}.  

\smallskip
\noindent The ramification formula (see e.g. \cite[I.(20), Lemma I.16.1]{BHPV} or, directly, \cite[Proposition 2.5]{Liedtke}), (i), (ii) and \eqref{eq.2L_111} yield 
\begin{equation}\label{eq.2.torsion}
    \omega_X^{\otimes 2} = \pi^*(\omega_W^{\otimes 2} \otimes \mathcal O_W(D_{100} + D_{010} + D_{001} + D_{111}))= \pi^*(\mathcal O_W(2)).
\end{equation}
Since $X$ is assumed to be locally Gorenstein, $\pi^*\mathcal O_W(1) \otimes \omega_X^*$
is a line bundle $\mathcal L$ on $X$, which is $2$-torsion,  hence (1) of Lemma~\ref{lemma.criterion.canonical.cover} holds. Since (1) (i) holds, the abelian cover $\pi$ satisfies \eqref{eq.trace.module.split}. By (1) (iv) we have \eqref{relations.ring.structure.Z_2^3}. Then $D_{110}=D_{101}=D_{011}=0$ implies Theorem~\ref{thm.splitting} (II) (2) and, hence, \eqref{eq.self.duality}. Thus, (2) of Lemma~\ref{lemma.criterion.canonical.cover} is satisfied. Finally, by (v), Theorem~\ref{thm.splitting} (II) (3) is satisfied and, hence, so are \eqref{eq.p_g=0} and \eqref{eq.vanishing.global.sections}. Thus, (3) of Lemma~\ref{lemma.criterion.canonical.cover} holds. Therefore, we can apply Lemma~\ref{lemma.criterion.canonical.cover}, so $\pi$ is a canonical cover. This completes the proof of (a).

\smallskip

\noindent Now we prove (b), so we assume (iii') holds.  Since any three among $H_{100}$, $H_{010}$, $
{H_{001}}$, $H_{111}$ satisfy \cite[Proposition 3.1.iii) a)]{Pardini}, condition $D_{100}  \cap D_{010} \cap D_{001} \cap D_{111}= \emptyset$ and the other conditions of (iii') imply, by \cite[Proposition 3.1]{Pardini}, that $X$ is smooth and, in particular, locally Gorenstein, so, by (a), $\pi$ is an abelian canonical cover and we conclude the proof of (b). 
\end{proof}

 \subsection{Abelian canonical covers with group $\mathbb{Z}_2 \times \mathbb{Z}_4$}
  \label{section.abelian.Z2xZ4} \ 

  \smallskip

\noindent We start this subsection making explicit, for the group $\mathbb Z_2 \times \mathbb Z_4$, those relations in Proposition~\ref{prop.alg.struct.bidouble} that we will use in the remaining of the subsection:

\begin{prop}\label{prop.alg.struct.Z2Z4}
  If $\pi: \mathcal X \longrightarrow \mathcal Y$ is an abelian cover with group $\mathbb{Z}_2 \times \mathbb{Z}_4$, then its building data satisfy, among others, the following relations: 
  \begin{equation}\label{eq.alg.str.L_01L_02}
    L_{01} \otimes L_{02} =L_{03}(D_{03} + D_{13})
\end{equation}
  \begin{equation}\label{eq.alg.str.L_01L_03}
    L_{01} \otimes L_{03} = \mathcal O_W(D_{01} + D_{02} + D_{03} + D_{11} + D_{12} + D_{13}  )
\end{equation}
\begin{equation}\label{eq.alg.str.L_01L_11}
   L_{01} \otimes L_{11} = L_{12}(D_{02} + D_{03} + D_{11} + D_{13})
\end{equation}
\begin{equation}\label{eq.alg.str.L_01L_12}
    L_{01} \otimes L_{12} =L_{13}(D_{03} + D_{12})
\end{equation}
\begin{equation}\label{eq.alg.str.L_02L_02}
    L_{02} \otimes L_{02} = \mathcal O_W(D_{01} + D_{03} + D_{11} + D_{13})
\end{equation}
\begin{equation}\label{eq.alg.str.L_02L_03}
    L_{02} \otimes L_{03} =L_{01}(D_{01} + D_{11})
\end{equation}
\begin{equation}\label{eq.alg.str.L_02L_11}
    L_{02} \otimes L_{11} =L_{13}(D_{03} + D_{11})
\end{equation}
\begin{equation}\label{eq.alg.str.L_03L_10}
    L_{03} \otimes L_{10} = L_{13}(D_{11} + D_{12})
\end{equation}
\begin{equation}\label{eq.alg.str.L_03L_13}
    L_{03} \otimes L_{13} = L_{12}(D_{01} + D_{02} + D_{11} + D_{13} )
\end{equation}
\begin{equation}\label{eq.alg.str.L_10L_12}
    L_{10} \otimes L_{12} = L_{02}(D_{10} + D_{12})
\end{equation}
\begin{equation}\label{eq.alg.str.L_11L_13}
    L_{11} \otimes L_{13} = \mathcal O_W(D_{01}+ D_{02}  + D_{03}  + D_{10} + D_{11} + D_{13})
\end{equation}
\begin{equation}\label{eq.alg.str.L_13L_13}
    L_{13} \otimes L_{13} = L_{02}(D_{01}+ D_{02} + D_{10}  + D_{13} )
\end{equation}
\end{prop}

\begin{proof}
    The above relations are a subset of the relations that we obtain when we explicitly write the relations
    in \eqref{relation.ring.structure}.
    In order to write these relations explicitly,  we need to know, for every $\chi_{a_1a_2} \in G^*$ and every pair $(H_{b_1b_2}, \psi_{b_1b_2})$ as in \ref{notation.building.data}, the exponent $i_{\chi_{a_1a_2}}$ of \cite[(2.1)]{Pardini}. 
 This can be achieved from 
Theorem~\ref{thm.Pardini} by an elementary, routine computation.
\end{proof}

\begin{theorem}\label{thm.Z2Z4.alg.str}
  Let $W$ be such that $p_g(W)=0$. 

  \begin{enumerate}
    \item If  $\mathrm{Pic} \, W$ has no $2$-torsion and $\pi: X \longrightarrow W$ is an abelian canonical cover with group $G=\mathbb{Z}_2 \times \mathbb{Z}_4$, then, without loss of generality, we may assume  

\begin{enumerate}
    \item[(i)] $L_{13} = \omega_{W}^*(1)$.
\end{enumerate}
 In this case, we have 
 \begin{enumerate}
 \item[(ii)] $D_{03} = D_{11} = D_{12} = 
 0$;  
 \item[(iii)] $D_{01}+D_{02}+  D_{10}+ D_{13}$ is a reduced divisor; 
 \item[(iv)] the $L_{a_1a_2}$ and the $D_{b_1b_2}$ satisfy 
 \eqref{relation.ring.structure}; 
 \item[(v)] $h^0(\omega_W \otimes L_{a_1a_2})=0$ for all $(a_1,a_2) \neq (1,3)$.
\end{enumerate}
If $X$ is smooth, then we also have
 \begin{enumerate}
\item[(iii')] 
{$D_{01}, D_{02},  D_{10}, D_{13}$ are smooth; $D_{01}, D_{02}, D_{13}$ are pairwise disjoint; and $D_{10}$ meets transversally each one of $D_{01}, D_{02}, D_{13}$.}
\end{enumerate}

\smallskip

\item Conversely, let  $\{L_{a_1a_2}, D_{b_1b_2}\}$ $(a_1, b_1 = 0,1$, $a_2, b_2 =0,1, 2, 3; (a_1,a_2),$ $(b_1,b_2) \neq (0,0))$ satisfy (i) to (v) above 
and let 
  $\pi: X \longrightarrow W$ be the  abelian  cover having $\{L_{a_1a_2}, D_{b_1b_2}\}$ as its building data and with algebra structure determined by \eqref{relation.ring.structure} and (ii).
  \begin{enumerate}
      \item If $X$ is locally Gorenstein, then $\pi$ is an abelian canonical cover with group $\mathbf{Z}_2 \times \mathbf{Z}_4$.

      \item If (iii') holds, then $\pi$ is an abelian canonical cover with group $\mathbf{Z}_2 \times \mathbf{Z}_4$ and $X$ is smooth.
  \end{enumerate} 
 \end{enumerate} 
\end{theorem}

\begin{proof}
   By \eqref{formula.splitting}, 
\begin{equation}\label{eq.character.omega.Z2Z4}
    \omega_W^*(1)=L_{a_1a_2}
\end{equation}
 for some $a_1\in \{0,1\}$ and $a_2 \in \{0,1,2,3\}$ with $(a_1, a_2) \neq (0,0)$. We divide the proof of (1) (i) into several steps.

\smallskip

\noindent
\underline{\emph{Step 1.}} The pair $(a_1, a_2)$ in \eqref{eq.character.omega.Z2Z4} cannot be $(0,2)$: 
\smallskip

\noindent Suppose the contrary. 
By Theorem~\ref{thm.splitting} (II) (5), 
\begin{equation*}
    L_{01}\otimes L_{03} \otimes L_{10} \otimes L_{11} \otimes L_{12} \otimes L_{13} = L_{02}^{\otimes 3}.
\end{equation*}
Then, on the one hand,  
\eqref{eq.alg.str.L_01L_03}, \eqref{eq.alg.str.L_10L_12}, \eqref{eq.alg.str.L_11L_13} imply 
\begin{equation*}
L_{02}^{\otimes 2}= 
\mathcal O_W(2(D_{01}+D_{02}+D_{03}+D_{10}+D_{11}+D_{12}
+D_{13})),
\end{equation*}
 and, on the other hand, 
 \begin{equation*}
L_{02}^{\otimes 2}=  \mathcal O_W((D_{01}+D_{03}+D_{11}
+D_{13})), 
 \end{equation*}
 by \eqref{eq.alg.str.L_02L_02}.
 Then
\begin{equation*}
    D_{01}+2D_{02}+D_{03}+2D_{10}+D_{11}+2D_{12}+D_{13} \sim 0.
\end{equation*}
Then, since $W$ is projective, 
\begin{equation*}
D_{01}=D_{02}=D_{03}=D_{10}=D_{11}=D_{12}
=D_{13} = 0.
\end{equation*}
Then \eqref{eq.alg.str.L_02L_02} implies $L_{02}^{\otimes 2}= \mathcal O_W$.
Since Pic$W$ has no $2$-torsion, $L_{02}=\mathcal O_W$. Then \eqref{formula.splitting} implies $h^0(\mathcal O_X) \ge 2$, so $X$ would not be connected, which is a contradiction.

\smallskip

\noindent
\underline{\emph{Step 2.}} The pair $(a_1, a_2)$ in \eqref{eq.character.omega.Z2Z4} cannot be $(1,0)$ or $(1,2)$:
\smallskip

\noindent Suppose the contrary. There are automorphisms of $G^*$ swapping $\chi_{10}$ and $\chi_{12}$, so, after possibly relabeling the $D_{b_1b_2}$ by an automorphism of $G$ and $L_{a_1a_2}$ by the inverse of its dual automorphism on $G^*$, we may assume $\omega_W^*(1)=L_{12}$.  By Theorem~\ref{thm.splitting} (II) (5) and \eqref{eq.alg.str.L_01L_11}, \eqref{eq.alg.str.L_03L_13}, \eqref{eq.alg.str.L_10L_12},
\begin{equation*}
    L_{12}^{\otimes 3}= L_{01}\otimes L_{02} \otimes L_{03} \otimes L_{10} \otimes L_{11} \otimes L_{13} = L_{12}^{\otimes 3}(D_{01} + 2D_{02} + D_{03}+ 3D_{11} + 3D_{13}),
\end{equation*}
so 
\begin{equation*}
    D_{01} + 2D_{02} + D_{03}+ 3D_{11} + 3D_{13} \sim 0.
\end{equation*}
Since $W$ is projective, 
\begin{equation*}
D_{01}=D_{02}=D_{03}=D_{11}=D_{13} = 0.
\end{equation*}
Then \eqref{eq.alg.str.L_02L_02} implies $L_{02}^{\otimes 2} = \mathcal O_W$ and we get a  contradiction as in Step 1.

\smallskip
\noindent \underline{\emph{Step 3.}} We may assume without loss of generality $a_1=1, a_2=3$ in \eqref{eq.character.omega.Z2Z4}:
\smallskip

\noindent Steps 1 and 2 imply $(a_1,a_2)=(0,1), (0,3), (1,1)$ or $(1,3)$. 
There are automorphisms of $G^*$ taking $\chi_{13}$ to any of  $\chi_{01}, \chi_{03}$ and $\chi_{11}$ so, after possibly relabeling the $D_{b_1b_2}$ by an automorphism of $G$ and $L_{a_1a_2}$ by the inverse of its dual automorphism on $G^*$, we may assume $\omega_W^*(1)=L_{13}$. This completes the proof of (1) (i). 

\medskip
\noindent Now we prove (1) (ii). By Theorem~\ref{thm.splitting} (II) (5) and \eqref{eq.alg.str.L_01L_12}, \eqref{eq.alg.str.L_02L_11}, \eqref{eq.alg.str.L_03L_10}
\begin{equation*}
 L_{13}^{\otimes 3}=   L_{01} \otimes L_{02}\otimes L_{03} \otimes L_{10} \otimes L_{11} \otimes L_{12} = L_{13}^{\otimes 3}(2(D_{03}+D_{11}+D_{12})).
\end{equation*}
Then 
\begin{equation*}
   2(D_{03}+D_{11}+D_{12})\sim 0. 
\end{equation*}
Since $W$ is projective, 
\begin{equation*}
 D_{03}  = D_{11} = D_{12} = 0.
\end{equation*}
This finishes the proof of (1) (ii).

\smallskip
\noindent Since $X$ is normal, (1) (ii) and Theorem~\ref{thm.normal} imply (1) (iii). Theorem~\ref{thm.Pardini} implies (1) (iv) and Theorem~\ref{thm.splitting} implies (1) (v).

\smallskip
\noindent Finally, if $X$ is smooth, since  $H_{01}, H_{02}$; $H_{01}, H_{13}$; $H_{02}, H_{13}$ are pairs of cyclic subgroups of $G$ having nonzero intersection, \cite[Proposition 3.1 iii) a)]{Pardini} implies $D_{01}, D_{02}, D_{13}$ are pairwise disjoint. In addition,   
\cite[Proposition 3.1 ii), iii) b)]{Pardini} implies $D_{01}, D_{02},  D_{10}, D_{13}$ are smooth 
and $D_{10}$ meets transversally each one of $D_{01}, D_{02}, D_{13}$.

\smallskip

\noindent Assume now the hypotheses of (2). We first prove (a). 
By Theorem~\ref{thm.normal} and (1) (iii),  $X$ is normal.
 To see that $\pi$ is a canonical cover we argue similarly to the proof of Theorem~\ref{thm.tridouble.alg.str} (2), taking now into account that the  ramification formula, (i), (ii), \eqref{eq.alg.str.L_02L_02}, \eqref{eq.alg.str.L_13L_13}  yield 
\begin{equation*}
    \omega_X^{\otimes 4} = \pi^*(\omega_W^{\otimes 4} \otimes \mathcal O_W(3D_{01} + 2D_{02} + 2D_{10}+3D_{13}))= \pi^*(\mathcal O_W(4))
\end{equation*}   
and that \eqref{eq.alg.str.L_01L_12}, \eqref{eq.alg.str.L_02L_11}, \eqref{eq.alg.str.L_03L_10} and (1) (ii) imply \eqref{eq.self.duality}.

\smallskip
\noindent If  (iii') holds, 
since  the pair $H_{01}, H_{10}$ satisfies \cite[Proposition 3.1 iii) a)]{Pardini} but the pairs  $H_{01}, H_{02}$; $H_{01}, H_{13}$; $H_{02}, H_{13}$ do not, the condition that $D_{01}, D_{02}, D_{13}$ are pairwise disjoint and  the other conditions of (iii') imply, by \cite[Proposition 3.1]{Pardini}, that $X$ is smooth and, in particular, locally Gorenstein, so we conclude by (a).
\end{proof}

  \subsection{Abelian canonical covers with group $\mathbb{Z}_8$}
  \label{subsection.abelian.Z8} \

\smallskip
\noindent We start this subsection by making explicit, for the group $\mathbb Z_8$, those relations in \eqref{relation.ring.structure} that we will use in the remaining of this section: 

\begin{prop}\label{prop.alg.struct.cyclic}
  If $\pi: \mathcal X \longrightarrow \mathcal Y$ is an abelian cover with group $\mathbf{Z}_8$, then its building data satisfy, among others, the following relations: 
  \begin{equation}\label{eq.alg.str.Z8.L_1L_2}
    L_1 \otimes L_2 = L_3(D_3+D_6+D_7)
\end{equation}
\begin{equation}\label{eq.alg.str.Z8.L_1L_6}
    L_1 \otimes L_6 = L_7(D_5+D_6+D_7)
\end{equation}
\begin{equation}\label{eq.alg.str.Z8.L_1L_7}
    L_1 \otimes L_7 = \mathcal O_W(D_1+D_2+D_3+D_4+D_5+D_6+D_7)
\end{equation}
\begin{equation}\label{eq.alg.str.Z8.L_2L_2}
    L_2 \otimes L_2 = L_4(D_2+D_3+D_6+D_7)
\end{equation}
\begin{equation}\label{eq.alg.str.Z8.L_2L_4}
    L_2 \otimes L_4 = L_6(D_3+D_7)
\end{equation}
\begin{equation}\label{eq.alg.str.Z8.L_2L_5}
    L_2 \otimes L_5 = L_7(D_3+D_6+D_7)
\end{equation}
\begin{equation}\label{eq.alg.str.Z8.L_2L_6}
    L_2 \otimes L_6 = \mathcal O_W(D_1+D_2+D_3+D_5+D_6+D_7)
\end{equation}
\begin{equation}\label{eq.alg.str.Z8.L_3L_4}
    L_3 \otimes L_4 = L_7(D_5+D_7)
\end{equation}
\begin{equation}\label{eq.alg.str.Z8.L_3L_5}
    L_3 \otimes L_5 = \mathcal O_W(D_1+D_2+D_3+D_4+D_5+D_6+D_7)
\end{equation}
\begin{equation}\label{eq.alg.str.Z8.L_4L_4}
    L_4 \otimes L_4 = \mathcal O_W(D_1+D_3+D_5+D_7)
\end{equation}
\begin{equation}\label{eq.alg.str.Z8.L_6L_6}
    L_6 \otimes L_6 =  L_{4}(D_1+D_2+D_5+D_6)
\end{equation}
\begin{equation}\label{eq.alg.str.Z8.L_7L_7}
    L_7 \otimes L_7 = L_6(D_1+D_2+D_3+D_4).
\end{equation}
\end{prop}

\begin{proof}
    The above relations are a subset of the relations that we obtain when we explicitly write the relations
    in \eqref{relation.ring.structure}.
    In order to do so,  we need to know, for every $a \in \{1, \dots, 7\}$ and every pair $(H_b, \psi_b)$, the exponent $i_{\chi_a}$ of \cite[(2.1)]{Pardini}. To do so, we use \eqref{eq.congruence2} and the relations follow easily from Theorem~\ref{thm.Pardini}.
\end{proof}

\begin{theorem}\label{thm.Z8.alg.str}
  Let $W$ be such that $p_g(W)=0$.

  \begin{enumerate}
    \item If  $\mathrm{Pic} \, W$ has no $2$-torsion or 
{$\omega_W^{\otimes 4}$} is not very ample (e.g, if $W$ 
  is  not of general type) and $\pi: X \longrightarrow W$ is an abelian canonical cover with group $G=\mathbb{Z}_8$, then, without loss of generality, we may assume  

\begin{enumerate}
    \item[(i)] $L_{7} = \omega_{W}^*(1)$.
\end{enumerate}
 In this case, we have 
 \begin{enumerate}
 \item[(ii)] $D_{3} = D_{5} = D_{6} = D_7=0$;    
 \item[(iii)] $D_{1}+D_{2}+D_{4}$ is a reduced divisor;
 \item[(iv)] The $L_{a}$ and the $D_{b}$ satisfy \eqref{relation.ring.structure}; and 
 \item[(v)] $h^0(\omega_W \otimes L_a)=0$ 
 for all $a \neq 7$.
\end{enumerate}
If $X$ is smooth, then we also have
\begin{enumerate}
 \item[(iii')] 
 {$D_{1}, D_{2}, D_{4}$ are smooth and pairwise disjoint.}
\end{enumerate}

\smallskip

\item Conversely, let  $\{L_{a}, D_{b}\}$ $(a, b=1, \dots, 7)$ satisfy (i) to (v) above 
and let 
  $\pi: X \longrightarrow W$ be the  abelian  cover having $\{L_{a}, D_{b}\}$ as its building data and with algebra structure determined by \eqref{relation.ring.structure} and (ii).
  \begin{enumerate}
      \item If $X$ is locally Gorenstein, then $\pi$ is an abelian canonical cover with group $\mathbf{Z}_8$.

      \item If (iii') holds, then $\pi$ is an abelian canonical cover with group $\mathbf{Z}_8$ and $X$ is smooth.
  \end{enumerate} 
\end{enumerate} 
\end{theorem}

\begin{proof}
By \eqref{formula.splitting}, 
\begin{equation}\label{eq.character.omega}
    \omega_W^*(1)=L_{a}
\end{equation}
 for some $a\in \{1,\dots, 7\}$. We divide the proof of (1) (i) into several steps.

\smallskip

\noindent
\underline{\emph{Step 1:}} The subindex $a$ in \eqref{eq.character.omega} cannot be $2$ or $6$. 

\noindent Suppose the contrary. There are automorphisms of $G^*$ swapping $\chi_2$ and $\chi_6$, so, after possibly relabeling the $D_{b}$ by an automorphism of $G$ and $L_{a}$ by the inverse of its dual automorphism on $G^*$, we may assume $\omega_W^*(1)=L_{2}$.  By Theorem~\ref{thm.splitting} (II) (4) and \eqref{eq.alg.str.Z8.L_2L_2}, \eqref{eq.alg.str.Z8.L_2L_4},
\begin{equation*}
    L_1^{\otimes 2} \otimes L_3^{\otimes 2} \otimes L_4^{\otimes 2} \otimes L_5^{\otimes 2} \otimes L_6^{\otimes 2} \otimes L_7^{\otimes 2} = L_2^{\otimes 6} = L_6^{\otimes 2}(2D_2 + 4D_3 + 2D_6 + 4D_7)
\end{equation*}
and, by \eqref{eq.alg.str.Z8.L_1L_7}, \eqref{eq.alg.str.Z8.L_3L_5}, \eqref{eq.alg.str.Z8.L_4L_4}, 
\begin{equation*}
    L_1^{\otimes 2} \otimes L_3^{\otimes 2} \otimes L_4^{\otimes 2} \otimes L_5^{\otimes 2} \otimes L_6^{\otimes 2} \otimes L_7^{\otimes 2} =  L_6^{\otimes 2}(5D_1+ 4D_2 +5D_3+ 4D_4 +5D_5+4D_6+5D_7).
\end{equation*}
Then 
\begin{equation*}
    5D_1+ 2D_2 + D_3 + 4D_4 + 5D_5 + 2D_6 + D_7 \sim 0.
\end{equation*}
Then, since $W$ is projective, 
\begin{equation*}
D_1 =  D_2 = D_3 = D_4 = D_5 = D_6 = D_7 = 0.
\end{equation*}
Then \eqref{eq.alg.str.Z8.L_2L_2}, \eqref{eq.alg.str.Z8.L_4L_4} and our assumption implies 
\begin{equation*}
   \omega_W^{-4}(4) = L_2^{\otimes 4} = \mathcal O_W.
\end{equation*}
This implies  
{$\omega_W^{\otimes 4}$} is  very ample, 
hence, by hypothesis, Pic$W$ has no $2$-torsion. Then $L_2=\mathcal O_W$, so \eqref{formula.splitting} implies $h^0(\mathcal O_X) \ge 2$ (in fact, $h^0(\mathcal O_X) \ge 4$), so $X$ would not be connected, which is a contradiction. 

\smallskip
\noindent \underline{\emph{Step 2:}} The subindex $a$ in \eqref{eq.character.omega} cannot be $4$.

\noindent Suppose the contrary. Arguing as in Step 1, Theorem~\ref{thm.splitting} (II) (4) and \eqref{eq.alg.str.Z8.L_1L_7}, \eqref{eq.alg.str.Z8.L_2L_6}, \eqref{eq.alg.str.Z8.L_3L_5}, \eqref{eq.alg.str.Z8.L_4L_4} imply
\begin{equation*}
    3D_1 + 6D_2 + 3D_3 + 4D_4 + 3D_5 + 6D_6 + 3D_7 \sim 0, 
\end{equation*}
so, as in Step 1, 
\begin{equation*}
D_1 = D_2 = D_3 = D_4 = D_5 = D_6 = D_7 = 0.
\end{equation*}
Then \eqref{eq.alg.str.Z8.L_4L_4} implies 
\begin{equation*}
  \omega_W^{-2}(2) = L_4^{\otimes 2} = \mathcal O_W.
\end{equation*}
Arguing as in Step 1, by the theorem hypothesis on $W$, we arrive at a contradiction.

\smallskip
\noindent \underline{\emph{Step 3:}} We may assume without loss of generality $a=7$ in \eqref{eq.character.omega}.

\noindent Steps 1 and 2 imply $a=1, 3, 5$ or $7$. 
There are automorphisms of $G^*$ taking $\chi_7$ to any of  $\chi_1, \chi_3$ and $\chi_5$ so, after possibly relabeling the $D_{b}$ by an automorphism of $G$ and $L_{a}$ by the inverse of its dual automorphism on $G^*$, we may assume $\omega_W^*(1)=L_{7}$. This completes the proof of (1) (i). 

\medskip
\noindent Now we prove (1) (ii). By Theorem~\ref{thm.splitting} (II) (4), 
\begin{equation*}
    L_1^{\otimes 2} \otimes L_2^{
    {\otimes 2}} \otimes L_3^{\otimes 2} \otimes L_4^{\otimes 2} \otimes L_5^{\otimes 2} \otimes L_6^{\otimes 2} = L_7^{\otimes 6}.
\end{equation*} 
This together with \eqref{eq.alg.str.Z8.L_1L_6}, \eqref{eq.alg.str.Z8.L_2L_5}, \eqref{eq.alg.str.Z8.L_3L_4} implies  
\begin{equation*}
    2D_3 + 4D_5 + 4D_6 + 6D_7 \sim 0. 
\end{equation*}
Then, since $W$ is projective, 
\begin{equation*}
 D_3  = D_5 = D_6 = D_7 = 0.
\end{equation*}
This finishes the proof of (1) (ii). Since $X$ is normal, (1) (ii) and Theorem~\ref{thm.normal} imply (1) (iii). Theorem~\ref{thm.Pardini} implies (1) (iv)  and Theorem~\ref{thm.splitting} (I) (3) implies (1) (v). 

\smallskip
\noindent Finally, if $X$ is smooth, since any two nonzero subgroups of $\mathbb Z_8$ have nonzero intersection, \cite[Proposition 3.1 iii) a)]{Pardini} implies $D_1$, $D_2$ and $D_4$ are disjoint pairwise. 
\cite[Proposition 3.1 ii), iii) b)]{Pardini} imply $D_1$, $D_2$ and $D_4$ are smooth.

\smallskip

\noindent Assume now the hypotheses of (2). We  first prove (a). Note first that  Theorem~\ref{thm.normal} and (1) (iii) imply $X$ is normal.  
To see that $\pi$ is a canonical cover we argue similarly to the proof of Theorem~\ref{thm.tridouble.alg.str} (2), taking now into account that the  ramification formula, (i), (ii), \eqref{eq.alg.str.Z8.L_4L_4}, \eqref{eq.alg.str.Z8.L_6L_6} and \eqref{eq.alg.str.Z8.L_7L_7} yield 
\begin{equation*}
    \omega_X^{\otimes 8} = \pi^*(\omega_W^{\otimes 8} \otimes \mathcal O_W(7D_{1} + 6D_{2} + 4D_{4}))= \pi^*(\mathcal O_W(8)) \end{equation*}
    and that \eqref{eq.alg.str.Z8.L_1L_6}, \eqref{eq.alg.str.Z8.L_2L_5}, \eqref{eq.alg.str.Z8.L_3L_4} and (1) (ii) imply \eqref{eq.self.duality}.

    \smallskip
\noindent Now we prove (b). If  (iii') holds, since $D_1, D_2, D_4$ are pairwise disjoint, this and  the other conditions of (iii') imply, by \cite[Proposition 3.1]{Pardini}, that $X$ is smooth and, in particular, locally Gorenstein, so we conclude by (a).
\end{proof}

\section{Abelian canonical covers up to degree $8$ of $\mathbb P^2$ and Hirzebruch surfaces}\label{section.abelian.rational}

 In this last section  we use the results of Sections~\ref{section.split} and \ref{section.abelian.canonical} to give a complete classification of 
 abelian canonical covers $\pi: X \longrightarrow W$ up to degree $8$, when $W$ is  $\mathbb P^2$ or a Hirzebruch surface and $W$, embedded by a complete linear series but not necessarily as a surfaces of minimal degree.

 \smallskip
\noindent
Recall that, if $W$ is  a Hirzebruch surface, we will use Notation~\ref{notation.Hirzebruch}.

\begin{remark}
   Recall that, by Proposition~\ref{prop.split.minimal.degree} (1), for a canonical cover \linebreak $\pi: X \longrightarrow W$ of degree $n > 2$, if $W=\mathbb P^2$, then $\mathcal O_W(1) = \mathcal O_{\mathbb P^2}(1)$ and, by Proposition~\ref{prop.split.minimal.degree} (2) (b), if $\pi$ is abelian and $X$ is irregular, then $W$ is $\mathbb F_0$. This is why, in the statements of this section, we will only consider $\mathcal O_W(1) = \mathcal O_{\mathbb P^2}(1)$ whenever $W=\mathbb P^2$ and $W=\mathbb F_0$ whenever $X$ is irregular. 
\end{remark} 

\subsection{Covers of prime degree}\label{subsection.prime.rational} \

\smallskip
\noindent  Even if the  main focus of the section are  covers  up to  degree $8$,  we start looking at  covers of any prime degree $p$. 

 \smallskip 
 \noindent If $p=2$, canonical covers are well-known (see \cite{Horikawa}, \cite{Persson}) and easy to handle. As pointed our in Remark~\ref{thm.double}, the existence of canonical double covers $\pi: X \longrightarrow W$ amounts to the existence of reduced divisors in $|\omega_W^{-2}(2)|$ (or smooth divisors, for $X$ to be smooth). Since Pic$W$ has no torsion, $\pi$ is determined by a choice of a member of $|\omega_W^{-2}(2)|$.
General members of $|\omega_W^{-2}(2)|$ are smooth if $W$ is is $\mathbb P^2$ or a Hirzebruch surface $\mathbb F_e$ with $e \le 3$
(recall Notation~\ref{notation.Hirzebruch}) or if $\mathcal O_W(1)$ is sufficiently ample if $e \ge 4$.

\smallskip
\noindent If $p \geq 3$ we show in Theorems~\ref{cor.prime.P^2} and \ref{cor.prime.ruled} that there exist abelian canonical covers of $\mathbb P^2$ and Hirzebruch surfaces if and only if $p=3$ and $(W, \mathcal O_W(1))=(\mathbb P^2, \mathcal O_{\mathbb P^2}(1))$.

\begin{theorem}\label{cor.prime.P^2}
Let $W=\mathbb P^2$. 

\smallskip
\noindent If $\pi: X \longrightarrow W$ is an abelian canonical cover 
of prime degree $p$, $p \ge 3$,  then $p=3$ and $\mathcal O_W(1)= \mathcal O_{\mathbb P^2}(1)$. In this case, 
$\pi$ is a simple cyclic  triple cover, branched along a non reduced plane sextic $D_1$. 
If, in addition, $X$ is smooth, then $D_1$ is a smooth plane sextic.

\smallskip
\noindent Conversely, if $D_1$ is a reduced plane sextic,
 then there exists a simple cyclic, canonical  triple cover of $\mathbb P^2$, branched along  $D_1$. 
If $D_1$ is a general plane sextic, then $D_1$ is smooth and so is $X$. 
\end{theorem}

\begin{proof}
By Proposition~\ref{prop.split.minimal.degree} (1),  $\mathcal O_W(1)=\mathcal O_{\mathbb P^2}(1)$. Then,  Theorem~\ref{thm.prime} (1) (iv) implies $p-1$ divides $4$, so $p=3$ 
or $p=5$, but \eqref{eq.splitting.P^2} and Theorem~\ref{thm.prime} (1) (iv) imply that $p=5$ cannot happen.
The rest of the result follows from Theorem~\ref{thm.prime} (1) (iii), (1) (iii'), (1) (iv), (2) and from Bertini's theorem. 
\end{proof}

\noindent Now we look at covers of Hirzebruch surfaces and, more generally, of (geometrically) ruled surfaces (compare with the results of \cite[Section 3]{Seshadri}, which are, on the one hand more general, because they deal with varieties of arbitrary dimension, but, on the other hand, less general, since they require $X$ to be smooth and $W$ to be of minimal degree or a scroll).

\begin{theorem}\label{cor.prime.ruled}
Let $W$ be a (geometrically) ruled surface over a curve of genus $g$. We will use the notation $C_0$ and $f$ as in Notation~\ref{notation.Hirzebruch}.  Let the hyperplane section of $i(W)$ be numerically equivalent to   $\alpha C_0+ \beta f$, where $\alpha$ and $\beta$ are (necessarily positive) integers. 
 There are no  abelian canonical covers $\pi: X \longrightarrow W$ 
of prime degree $p$, $p \geq 3$, if one of the following happens:
\begin{enumerate}
    \item $p-1$ does not divide $\alpha+2$ or $p-1$ does not divide $\beta+e+2-2g$. 
    \item $W$ is a Hirzebruch surface. 
\end{enumerate}
In particular, there are no  abelian canonical covers $\pi: X \longrightarrow W$ 
of prime degree $p$ if $\alpha$ is odd (e.g., this happens if 
  $W$  is embedded as a scroll (i.e., the hyperplane section of $i(W)$ is numerically equivalent to   $C_0+ \beta f$)) or if $\alpha \le p-4$.
\end{theorem}

\begin{proof}
    Recall that the canonical divisor of $W$ is numerically equivalent to $-2C_0+(2g-2-e)f$. Thus 
    Theorem~\ref{thm.prime} (1) (iv) implies that, if $\pi$ is an abelian canonical cover, then 
$p-1$  divides $\alpha+2$ and $\beta+e+2-2g$, so $\pi$ is not an abelian canonical cover if (1) holds.

\smallskip
\noindent Now let us see that $\pi$ is not  an abelian canonical cover if (2) holds. By (1), we may assume $\alpha \ge 2$ and we may also rule out the case $g=e=0$, $\beta=1$. Let us see if $L_1^*(1)$ has nonzero global sections. By  Theorem~\ref{thm.prime} (1) (iv), a  divisor associated to $L_1^*(1)$ is linearly equivalent to 
\begin{equation}\label{eq.num}
    \left (\alpha - \frac{\alpha+2}{p-1}\right ) C_0 + \left (\beta - \frac{\beta+e+2}{p-1}\right )f.
\end{equation}
Since $\alpha \ge 2$ and $p \ge 3$, the coefficient of $C_0$ in \eqref{eq.num} is non negative. Let us now look at the coefficient of $f$ in \eqref{eq.num}. If $e=0$, then $\beta \geq 2$, so $\beta(p-2) -e -2 \ge 0$, so the coefficient of $f$ in \eqref{eq.num} is non negative. If $e \ge 1$, then
\begin{equation*}
   \beta(p-2) -e-2  
    \ge (\alpha-1)e-1 \ge 0, 
\end{equation*}
so the coefficient of $f$ in \eqref{eq.num} is also non negative.
Then
$h^0(L_1^*(1)) > 0$, so Theorem~\ref{thm.splitting} (II) (2) and (3) imply $\pi$ cannot be an abelian canonical cover.
\end{proof}

 \subsection{Covers  with group $\mathbb{Z}_2 \times \mathbb{Z}_2$}\label{section.bidouble.rational} \ 

\smallskip
 \noindent Abelian canonical covers $\pi: X \longrightarrow W$ of degree $4$, in particular, those with group $\mathbb Z_2 \times \mathbb Z_2$, were completely classified in \cite{quad1}, if $W$ is a smooth surface of minimal degree (i.e., when $W=\mathbf P^2$ and $\mathcal O_W(1) = \mathcal O_{\mathbb P^2}(1)$ or $W$ is a Hirzebruch surface embedded as a rational normal scroll). We refer  the reader to \cite[Theorem 0.1 (B) (II)]{quad1} for the precise description of these
  covers (note that the notation there is ad-hoc  and differs from Notation~\ref{notation.building.data}). 
  
  \smallskip
  \noindent In this subsection we extend this classification and give the complete classification of abelian canonical covers $\pi: X \longrightarrow W$ with group $\mathbb Z_2 \times \mathbb Z_2$, when $W$ is a Hirzebruch surface not embedded as a rational normal scroll. Since by Proposition~\ref{prop.split.minimal.degree} (1), there are no canonical covers of $W$ if $W=\mathbf P^2$ and $\mathcal O_W(1) = \mathcal O_{\mathbb P^2}(\alpha),  \alpha \geq 2$, this together with \cite{quad1} gives the complete classification of 
  abelian canonical covers with group $\mathbb Z_2 \times \mathbb Z_2$, when $W$ is $\mathbf P^2$ or a Hirzebruch surface.

\begin{theorem}\label{thm.bidouble.nonscroll.regular}
  Let $W$ be a Hirzebruch surface $\mathbb{F}_e$ 
  embedded not as {a rational normal scroll} (i.e., $\alpha \ge 2$). 
  Let $G= \mathbb{Z}_2 \times \mathbb{Z}_2$.
\begin{enumerate}
    \item
    If $\pi: X \longrightarrow W$  is an abelian canonical cover with group $G$ and $X$ is regular, then $e=0$ and, after possibly relabeling $D_{b_1b_2}$ by  an automorphism of $G$ and $L_{a_1a_2}$ by the inverse of its dual automorphism on $G^*$ 
    and after possibly swapping $C_0$ and $f$, the building data of $\pi$ is the following:
 \begin{enumerate}
 \item[(i)]  $L_{11}=\mathcal O_W((\alpha+2)C_0+(m+2)f)$,
 \item[] $L_{10}=\mathcal O_W(C_0+(m+1)f)$, 
 \item[] $
 {L_{01}}=\mathcal O_W((\alpha+1)C_0+f)$;  
  \smallskip
  
 \item[(ii)] $D_{11}=0$; 
 \item[(iii)]  $D_{10}  \sim 2C_0+(2m+2)f$, 
  $D_{01} \sim (2\alpha+2)C_0+2f$;
  \item[(iv)] $D_{10},  D_{01}$ are reduced divisors without common components.  
 \end{enumerate}

  \smallskip
 \noindent
The $L_{a_1a_2}$ and $D_{b_1b_2}$ in (i) to (iii) above satisfy \eqref{eq.bidouble.alg.str}.

  \smallskip
 \noindent If $X$ is smooth, then, in addition, 
  \begin{enumerate}
 \item[(iv')]  $D_{10},  D_{01}$ are smooth divisors that meet transversally.
  \end{enumerate}

\smallskip
 \color{black}
\item 
 \noindent Conversely, let $e=0$ and let $L_{a_1a_2}$ and $D_{b_1b_2}$ be as in (i), (ii), (iii) above.
If  $(D_{10}, D_{01})$ is general in $|D_{10}| \times  |D_{01}|$, then (iv') above holds. Let $\pi: X \longrightarrow W$ be the abelian cover with group $G$ and  building data $\{L_{a_1a_2}, D_{b_1b_2}\}$. This cover $\pi$ is a canonical cover and $X$ is smooth and regular. 
\end{enumerate}
 \end{theorem}

 \begin{proof}
 By Proposition~\ref{prop.split.minimal.degree} (2) (a), we have $e=0$.
Now we prove (1). 
  After  possibly relabeling $D_{b_1b_2}$ by  an automorphism of $G$ and $L_{a_1a_2}$ by the inverse of its dual automorphism on $G^*$, we may assume by \eqref{eq.regular.splitting.a>1}
that $L_{10}=\mathcal O_W(C_0+(m+1)f)$ and $L_{11}=\omega_W^*(1)=\mathcal O_W((\alpha+2)C_0+(m+2)f)$; therefore $
{L_{01}} = \mathcal O_W((\alpha+1)C_0+f)$. 
This proves (1) (i). 

\smallskip
\noindent Part (ii) of (1) follows from Theorem~\ref{thm.bidouble.alg.str} (1) (ii). Part (iii) of (1) follows from 
{\eqref{eq.2L_10}, \eqref{eq.2L_11} and (1) (i), (ii)}. Parts (iv) and (iv') of (1) follow from Theorem~\ref{thm.bidouble.alg.str} (1) (iii) and (iii'). 

 \smallskip
 \noindent Let us now prove (2). Since $\{L_{a_1a_2}, D_{b_1b_2}\}$ satisfies (1) (i), (ii), (iii), it also satisfies \eqref{eq.bidouble.alg.str} and 
 Theorem~\ref{thm.bidouble.alg.str} (1) (v). 
 If   $(D_{10}, D_{01})$ is general in $|D_{10}| \times  |D_{01}|$, then (1) (iv')  holds by Bertini. 
 Then (2) follows from Theorem~\ref{thm.bidouble.alg.str} (2). 
 \end{proof}

 \begin{theorem}\label{thm.Hirzebruch.Z2Z2.irregular}
Let $W$ be the Hirzebruch surface $\mathbb{F}_0$ 
{embedded not as {a rational normal scroll} (i.e., $\alpha \ge 2$)}. 
{Let $G= \mathbb{Z}_2 \times \mathbb{Z}_2$}.
\begin{enumerate}
    \item
    If $\pi: X \longrightarrow W$  is an abelian canonical cover
    and $X$ is irregular, then, 
    after possibly relabeling $D_{b_1b_2}$ by  an automorphism of $G$ and $L_{a_1a_2}$ by the inverse of its dual automorphism on $G^*$ 
    and after possibly swapping $C_0$ and $f$, the building data of $\pi$ is the following:
 \begin{enumerate}
 \item[(i)] $L_{11}=\mathcal O_W((\alpha+2)C_0+(m+2)f)$; 
 
 \item[(ii)] $D_{11}=0$;
 
 \item[(iii)]  one of these  
 {three} cases happens: 
 \item[] \underline{Case 1}: $L_{10}=\mathcal O_W((m+1)f)$, $L_{01}=\mathcal O_W((\alpha+2)C_0+f)$; $D_{10} \sim (2m+2)f$, $D_{01} \sim (2\alpha+4)C_0+2f$; 
 \item[] \underline{Case 2}:  $L_{10}=\mathcal O_W((m+2)f)$, $L_{01}=\mathcal O_W((\alpha+2)C_0)$; $D_{10}  \sim (2m+4)f$, $D_{01} \sim (2\alpha+4)C_0$;
 \item[] \underline{Case 3}:  $L_{10} =\mathcal O_W(C_0+(m+2)f)$, $L_{01} =\mathcal O_W((\alpha+1)C_0)$;  $D_{10} \sim 2C_0 + (2m+4)f$, $D_{01} \sim (2\alpha+2)C_0$.

 \item[(iv)] $D_{01}$ and $D_{10}$ are reduced divisors with no common components;

 \item[(v)] in Case 1, $q(X)=m$ and, if $X$ is smooth, then $D_{01}$ is a smooth divisor and  $D_{10}$ consists of $2m+2$ distinct lines linearly equivalent to $f$ and meeting $D_{01}$  transversally; 

 \item[(vi)] in Case 2, $q(X)=\alpha+m+2$, $D_{01}$ consists of $2\alpha+4$ distinct lines linearly equivalent to $C_0$, $D_{10}$ consists of $2m+4$ distinct lines linearly equivalent to $f$ and  $X$ is smooth; in fact $X$ is the product of (smooth) hyperelliptic curves of genus $\alpha+1$ and $m+1$;

 \item[(vii)] In Case 3,  $q(X)=\alpha$ and, if $X$ is smooth, then $D_{10}$ is a smooth divisor and  $D_{01}$ consists of $2\alpha+2$ distinct lines linearly equivalent to $C_0$ and meeting $D_{10}$  transversally.
\end{enumerate}

 \noindent
The $L_{a_1a_2}$ and $D_{b_1b_2}$ in (i), (ii) and (iii) satisfy \eqref{eq.bidouble.alg.str}. 
\smallskip

 \color{black}
\item 
 \noindent Conversely, 
 let $L_{a_1a_2}$ and $D_{b_1b_2}$ be as in (i), (ii), (iii) above. If $(D_{01}, D_{10})$ is general in $|D_{01}| \times |D_{10}|$, then
 $L_{a_1a_2}$ and $D_{b_1b_2}$ satisfy (iv) and, in Cases 1 and 3 of (iii) above, they satisfy the description given in (v) and (vii) above respectively; moreover, $L_{a_1a_2}$ and $D_{b_1b_2}$ satisfy \eqref{eq.bidouble.alg.str}. Let   $\pi: X \longrightarrow W$ be the abelian cover with group  $G$ and building data $L_{a_1a_2}, D_{b_1b_2b}$. This cover $\pi$ is canonical and $X$ is smooth and irregular, with irregularity as in (v), (vi) or (vii) above, accordingly.
\end{enumerate}
 \end{theorem}

 \begin{proof}
First we prove (1). 
  After  possibly relabeling $D_{b_1b_2}$ by  an automorphism $\gamma$ of $G$ and $L_{a_1a_2}$ by the  automorphism $(\gamma^*)^{-1}$ of $G^*$, we may assume by \eqref{eq.regular.splitting.a>1} that 
  $$L_{11} = \omega_W^*(1)=\mathcal O_W((\alpha+2)C_0+(m+2)f).$$ This proves (1) (i). Then Theorem~\ref{thm.bidouble.alg.str} (1) (ii) implies (1) (ii). Now, since the degree of $\pi$ is $n=4$,  Proposition~\ref{prop.split.minimal.degree} (2) (b) one of $n_0, n_0', n_1'$ equals $1$ while the other two of them are $0$, and $n_1$ is also $0$. 

  \smallskip
  \noindent If $n_0=1$, we may choose $\gamma$ above so that $L_{10}=\mathcal O_W((m+1)f)$. In this case, by \eqref{eq.irregular.splitting}, $
{L_{01}}=\mathcal O_W((\alpha+2)C_0+f)$. 
  Then 
{\eqref{eq.2L_10} and \eqref{eq.2L_11}} imply that $D_{10}$ and $D_{01}$ are as in Case 1 of (1) (iii).

\smallskip
  \noindent If $n_0'=1$, arguing as in Case 1, we see we get Case 2 of (1) (iii)  this time. 
Finally, if $n_1'=1$, arguing as in Case 1, we see we get Case 3 of (1) (iii)  this time. 

\smallskip
  \noindent Part (iv) of (1) follows from Theorem~\ref{thm.bidouble.alg.str} (1) (iii). The statements about $q(X)$ in (1) (v), (vi), (vii) follow by computing the first cohomology group of \eqref{eq.irregular.splitting}, taking into account the description of $L_{10}, L_{01}, L_{11}$ given in (1) (i), (iii). The remaining of (1) (v), 
  (respectively, (vii)) follows from Theorem~\ref{thm.bidouble.alg.str} (1) (iii'), taking into account the description of $D_{10}$, 
  (respectively, $D_{01}$) given in (1) (iii).  
  For the remaining of (1) (vi), since $D_{10}$ and $D_{01}$ are reduced by (1) (iv), they are necessarily as described in (1) (vi), so $X$ is smooth by Theorem~\ref{thm.bidouble.alg.str} (2). This completes the proof of (1). 

  \smallskip
  \noindent The proof of (2) follows the same lines of the proof of Theorem~\ref{thm.bidouble.nonscroll.regular} (2), noting that a general element $(D_{01},D_{10})$ in
  $|D_{01}| \times |D_{10}|$ is as the $(D_{01},D_{10})$ in (1) (v), (vi), (vii) and that $q(X)$ is stated also in (1) (v), (vi), (vii).
 \end{proof}

 \subsection{Covers  with group $\mathbb{Z}_4$}\label{section.Z4.rational} \ 

 \smallskip
 \noindent As mentioned in Subsection~\ref{section.bidouble.rational},  
abelian canonical covers $\pi: X \longrightarrow W$  with group $\mathbb Z_4$ were completely classified in \cite{quad1}, if $W$ is a smooth surface of minimal degree (i.e., when $W=\mathbf P^2$ and $\mathcal O_W(1) = \mathcal O_{\mathbb P^2}(1)$ or $W$ is a Hirzebruch surface embedded as a rational normal scroll). We refer  the reader to \cite[Theorem 0.1 (A), (I)]{quad1} for the precise description of these
  covers (note that the assignment of the subindexes of the $D_b$ done there is ad-hoc and differs from the assignment in Notation~\ref{notation.building.data}). 
  
  \smallskip
  \noindent Recall that, by Proposition~\ref{prop.split.minimal.degree} (1), there are no canonical covers of $W$ if $W=\mathbf P^2$ and $\mathcal O_W(1) = \mathcal O_{\mathbb P^2}(\alpha),  \alpha \geq 2$. In this subsection we show that there are no abelian canonical covers $\pi: X \longrightarrow W$ with group $\mathbb Z_4$, when $W$ is a Hirzebruch surface not embedded as a rational normal scroll. Then, the only abelian canonical covers with group $\mathbb Z_4$, when $W$ is $\mathbf P^2$ or a Hirzebruch surface, are the ones classified in \cite{quad1}.

 \begin{theorem}\label{thm.Z4.non.scroll.regular}
  There are no abelian canonical covers  $\pi: X \longrightarrow W$ with group $\mathbf{Z}_4$, if $W$ is  a Hirzebruch surface $\mathbb{F}_e$  embedded not as a rational normal scroll (i.e., $\alpha \ge 2$).  
 \end{theorem}

 \begin{proof}
 Let $W$ be  a Hirzebruch surface $\mathbb{F}_e$  embedded not as a rational normal scroll (then $\alpha \ge 2$, by Notation~\ref{notation.Hirzebruch}) and let $\pi: X \longrightarrow W$ be an abelian canonical cover with group $\mathbf{Z}_4$.
     By Theorem~\ref{thm.Z4.alg.str} (1) (i), we may assume $L_3 = \omega_W^*(1)$. 

\smallskip
    \noindent We suppose first that $X$ is regular. In that case, \eqref{eq.regular.splitting.a>1} implies $L_1=\mathcal O_W(C_0+(m+1)f$ or $L_1=\mathcal O_W((\alpha +1)C_0+f)$. If 
$L_1=\mathcal O_W(C_0+(m+1)f$, then $L_2=\mathcal O_W((\alpha+1)C_0+f)$, so, by \eqref{eq.Z_4.2L_1}, $D_2 \sim (1-\alpha)C_0  + (2m+1)f$. Since $D_2$ is effective, $\alpha=1$, which contradicts our hypothesis. 

\smallskip
\noindent If $L_1=\mathcal O_W((\alpha+1)C_0+f)$, then $L_2=\mathcal O_W(C_0+(m+1)f$. Arguing as above, we get $D_2 \sim (2\alpha+1)C_0 + (1-m)f$ and $m=1$. By Notation~\ref{notation.Hirzebruch}, this contradicts our hypothesis $\alpha \ge 2$. Then $X$ cannot be regular. 

\smallskip
\noindent Suppose now that $X$ is irregular. According to \eqref{eq.irregular.splitting}, $L_1$ is isomorphic to one the following line bundles: $\mathcal O_W((m+1)f)$, $\mathcal O_W((m+2)f)$, $\mathcal O_W((\alpha+1)C_0)$, $\mathcal O_W((\alpha+2)C_0)$, $\mathcal O_W(C_0+(m+2)f)$ and $\mathcal O_W((\alpha+2)C_0+f)$. In each case, $L_2$ is determined by \eqref{eq.irregular.splitting}. Arguing as when we supposed $X$ regular, we see that, in the first four cases, the coefficient of $C_0$ or $f$ in the linear equivalence class of $D_2$ is negative, so $D_2$ cannot effective, which is a contradiction. Likewise, in the fifth case, we get $\alpha=1$ and in the last case, $m=1$, which also contradicts our hypothesis $\alpha \ge 2$. Therefore $X$ cannot be irregular either and, indeed, covers as in the statement cannot exist.
 \end{proof}

\subsection{Covers of degree $6$}\label{section.degree.6.rational} 

\begin{theorem}\label{thm.P2.Z6}
Let $W=\mathbb P^2$, $\mathcal{O}_W(1)=\mathcal{O}_{\mathbb{P}^2}(1)$ and $G= \mathbb{Z}_6$.

\begin{enumerate}
    \item If $\pi: X \longrightarrow W$ is an abelian canonical cover with group $G$, then,
    \begin{enumerate}
   \item  $X$ is singular;
        \item  the mildest possible set of singular points of such a surface $X$ consists of $4$ singularities of type $A_1$ and $4$ singularities of type $A_2$; this happens if $D_1, D_2, D_3$ below are smooth and meet one another transversally.
    \end{enumerate}

    \smallskip
    \item[] Moreover, 
    after possibly relabeling $D_{b}$ by  an automorphism of $G$ and $L_{a}$ 
    by the inverse of its dual automorphism on $G^*$, the building data of $\pi$ is the following:
    \begin{enumerate}
        \item[(i)] $L_1=L_2=L_3=L_4=\mathcal{O}_{\mathbb{P}^2}(2)$; $L_5 = \mathcal{O}_{\mathbb{P}^2}(4)$; 
        \item[(ii)] $D_4=D_5=0$;
        \item[(iii)] $D_1, D_2, D_3$ are conics;  
        \item[(iv)] $D_1, D_2, D_3$ are reduced divisors without common components. 
    \end{enumerate}

     \smallskip
 \noindent
The $L_{a}$ and $D_{b}$ in (i), (ii) and (iii) satisfy \eqref{relation.ring.structure}. 
\smallskip
  
\smallskip
\item Conversely, let $\{L_{a},D_{b}\}$ $(a,b= 1, \dots, 5)$ be as in (i), (ii), (iii) above. If   $(D_1, D_2, D_3)$ is a general element of $|\mathcal{O}_{\mathbb{P}^2}(2)| \times |\mathcal{O}_{\mathbb{P}^2}(2)| \times |\mathcal{O}_{\mathbb{P}^2}(2)|$, then $\{L_{a},D_{b}\}$ satisfies (iv) above and \eqref{relation.ring.structure}. Let   $\pi: X \longrightarrow W$ be the abelian cover with group  $G$, building data $\{L_{a},D_{b}\}$. Then $\pi$ is a canonical cover and  $X$ has canonical singularities; precisely, $X$ has exactly these singular points: $4$ singularities of type $A_1$ and $4$ singularities of type $A_2$.
\end{enumerate}
\end{theorem}

\begin{proof}

We first prove (1). 
Part (i) of (1) follows from  Proposition~\ref{prop.split.minimal.degree} (1), since we can assume, by Theorem~\ref{thm.Z6.alg.str} (1) (i), that $L_{5}=\omega_W^*(1)= \mathcal O_{\mathbb P^2}(4)$. Theorem~\ref{thm.Z6.alg.str} (1) (ii) then says that $D_{4}= D_{5} =0$. This together with \eqref{eq.alg.str.Z6.L_1L_1}, \eqref{eq.alg.str.Z6.L_1L_2}, \eqref{eq.alg.str.Z6.L_1L_5}, implies (1) (iii). Since $D_{1}, D_{2}, D_3$ are conics, $D_1$ meets both $D_2$ and $D_3$, so $X$ is singular by Theorem \ref{thm.Z6.alg.str} (1) (iii'). The mildest possible singularities for $X$ happen when $D_1, D_2, D_3$ are smooth and meet one another transversally. In such a case, $D_1$ and $D_2$ (which are branch divisors that correspond, respectively, to $H_1$ and $H_2$) meet at $4$ points;  by \cite[Proposition 3.3]{Pardini}, these points originate $4$ singularities of type $A_2$ in $X$. Likewise, $D_1$ and $D_3$ (which are branch divisors that correspond, respectively, to $H_1$ and $H_3$) meet at $4$ points, that originate $4$ singularities of type $A_1$ in $X$. \cite[Proposition 3.1]{Pardini} implies that $X$ is smooth away from these $8$ points. This completes the proof of (1) (b). 
 
    \smallskip
    \noindent Now we prove (2). If 
    $\{L_a,D_b\}$ satisfies (1) (i), (ii), (iii), then it also satisfies \eqref{relation.ring.structure}. By Theorem~\ref{thm.Pardini}, there is an abelian cover $\pi: X \longrightarrow W$ associated to $\{L_a,D_b\}$. If  $(D_{1}, D_{2}, D_{3})$ is general in $|\mathcal O_{\mathbb{P}^2}(2)| \times |\mathcal O_{\mathbb{P}^2}(2)| \times 
    |\mathcal O_{\mathbb{P}^2}(2)|$, then $D_1, D_2, D_3$ are smooth and meet one another transversally, so they satisfy (1) (iv) and, as argued in the proof of (1) (b), the singular points of $X$ are exactly $4$ singularities of type $A_1$ and $4$ singularities of type $A_2$. These are canonical singularities, so $X$ is locally Gorenstein. Finally, since  $\{L_a,D_b\}$ satisfies (1) (i), (ii), (iii), Theorem~\ref{thm.Z6.alg.str} (1) (v) holds. Then Theorem~\ref{thm.Z6.alg.str} (2) implies that $\pi$ is a canonical abelian cover with group $G$.
\end{proof}

\begin{theorem}\label{thm.Hirzebruch.Z6.regular}
  Let $W$ be a Hirzebruch surface $\mathbb{F}_e$ and let $G= \mathbb{Z}_6$.
\begin{enumerate}
    \item
    If $\pi: X \longrightarrow W$  is an abelian canonical cover with group $G$ and $X$ is regular, then 
    \begin{enumerate}
        \item $e=0$ or $e=1$;
        \item  $\alpha=1$ (i.e., $W$ is embedded by $|C_0+mf|$, hence, a rational normal scroll);
        \item  $X$ is singular; 
    \end{enumerate}
        \smallskip
 
 \item[] Moreover, after possibly relabeling $D_{b}$ by  an automorphism of $G$ and $L_{a}$ by the inverse of its dual automorphism on $G^*$ 
    and after possibly swapping $C_0$ and $f$ when $W=\mathbb{F}_0$, the building data of $\pi$ is the following:
 \begin{enumerate}
 \item[(i)]  $L_{5}=\mathcal O_W(3C_0+(m+e+2)f)$;
 \item[(ii)] $D_4=D_5=0$; 
\item[(iii)] one of these two cases happens: 
\item[] \underline{\emph{Case 1:}} $L_1=L_3=\mathcal{O}_W(C_0+(m+1)f)$, $L_2=L_4=\mathcal O_W(2C_0+(e+1)f)$; $D_1 \sim D_2 \sim 2C_0+(e+1)f$, $D_3 \sim (2m-e+1)f$; 

\item[] \underline{\emph{Case 2:}} $e=0$ and $m=1$; $L_1=L_2=\mathcal{O}_W(2C_0+f)$, $L_3=L_4=\mathcal O_W(C_0+2f)$; $D_1 \sim 3f$, $D_2 \sim 3C_0$, $D_3 \sim 2C_0+f$; 

 \item[(iv)] $D_{1}, D_{2}, D_{3}$ are reduced divisors without common components; 

 \item[(v)] the mildest possible singularities for $X$ are
 $4m-2e+2$ singularities of type $A_1$ and $4$ singularities of type $A_2$ in Case 1 above and $6$ singularities of type $A_1$ and $9$ singularities of type $A_2$ in Case 2 above. 
 \end{enumerate}
 
 \smallskip
 \noindent
The $L_{a}$ and $D_{b}$ in (i), (ii) and (iii) satisfy \eqref{relation.ring.structure}. 
\smallskip

\item Conversely, assume (1) (a) and (1) (b) and let $\{L_{a},D_{b}\}$ $(a,b= 1, \dots, 5)$ be as in (i), (ii), (iii) above. If   $(D_1, D_2, D_3)$ is a general element in  $|D_{1}| \times |D_{2}| \times |D_{3}|$, then (iv) above holds and $\{L_{a},D_{b}\}$ satisfies \eqref{relation.ring.structure}. Let  $\pi: X \longrightarrow W$ be the abelian cover with group  $G$ and building data $\{L_{a},D_{b}\}$. Then $\pi$ is canonical and   $X$ is regular and has canonical singularities; precisely, $X$ has exactly these singular points: 
 \begin{enumerate}
     \item[(a)]  In Case 1 of (iii) above, $4m-2e+2$ singularities of type $A_1$ and $4$ singularities of type $A_2$. 
       \item[(b)]  In Case 2 of (iii) above,  $6$ singularities of type $A_1$ and $9$ singularities of type $A_2$. 
 \end{enumerate}
 \end{enumerate}
 \end{theorem}

 \begin{proof}
 First we prove (1). 
    Theorem~\ref{thm.Z6.alg.str} (1) 
    (ii) imply 
    (1) (ii). Let us suppose now $\alpha \ge 2$.  
    The splitting \eqref{eq.regular.splitting.a>1} and  Theorem~\ref{thm.Z6.alg.str} (1) (i)   imply there are four possible choices for $L_1$ and $L_3$, namely, $L_{1}$ equals $\mathcal O_W(C_0+(m+1)f)$ or $\mathcal O_W((\alpha+1)C_0+f)$ and so does $L_3$.
    If $L_{1} = L_{3} = \mathcal{O}_{W}((\alpha+1)C_{0} + f)$,
   then  \eqref{eq.alg.str.Z6.L_1L_1}, \eqref{eq.alg.str.Z6.L_2L_3}, (1) (ii) and Theorem~\ref{thm.Z6.alg.str} (1) (i)  imply $D_3 \sim (2a+1)C_0+(1-m)f$. Since $D_3$ is effective, $m=1$, but this  is not possible since $m \ge \alpha \ge 2$ by 
   Notation~\ref{notation.Hirzebruch}. For the other three choices of $L_1$ and $L_3$, \eqref{eq.alg.str.Z6.L_1L_1}, \eqref{eq.alg.str.Z6.L_1L_2}, \eqref{eq.alg.str.Z6.L_1L_4}, \eqref{eq.alg.str.Z6.L_2L_3}, \eqref{eq.alg.str.Z6.L_4L_4}, (1) (ii) and Theorem~\ref{thm.Z6.alg.str} (1) (i) imply that one of $D_1$, $D_2$, $D_3$ is linearly equivalent to $(1-\alpha)C_0+(2m+1)f$. Since $D_1$, $D_2$, $D_3$ are effective, $\alpha=1$, in contradiction with our assumption $\alpha \ge 2$. This completes the proof of (1) (b) and, by Theorem~\ref{thm.Z6.alg.str} (1) (i), we have (1) (i). 

   \smallskip
   \noindent Then $\alpha=1$ and, as above, by \eqref{eq.regular.splitting} and (1) (i), there are four choices for $L_1$ and $L_3$, namely, $L_{1}$ equals $\mathcal O_W(C_0+(m+1)f)$ or $\mathcal O_W(2C_0+f)$ and so does $L_3$. We split the proof in four cases accordingly.

   \smallskip
   \noindent \underline{Case 1}: $L_{1} = L_{3} = \mathcal{O}_{W}(C_{0} + (m+1)f)$. 

   \noindent In this case \eqref{eq.regular.splitting} implies $L_2=L_4=\mathcal O_W(2C_0+(e+1)f)$. Arguing as above, $D_1 \sim D_2 \sim 2C_0+(e+1)f$ and $D_3=(2m-e+1)f$, so we get Case 1 of Part (1) (iii) of the statement. Moreover, note that $$(2C_0 + (e+1)f) \cdot C_0 = 1-e$$ is negative if $e \ge 2$. Thus, if $e \ge 2$, then $C_0$ would be in the fixed part of both $|D_1|$ and $|D_2|$, so $D_1+D_2$ would not be reduced, and this would contradict Theorem~\ref{thm.Z6.alg.str} (1) (iii). This proves (1) (a) in this case. 

   \smallskip
   \noindent \underline{Case 2}: $L_{1} = \mathcal O_W(2C_0+(e+1)f)$ and $L_{3} =  \mathcal{O}_{W}(C_{0} + (m+1)f)$.

   \noindent In this case, \eqref{eq.alg.str.Z6.L_1L_4},  \eqref{eq.alg.str.Z6.L_2L_3} and (1) (i), (ii) imply $L_2=\mathcal O_W(2C_0+(e+1)f)$ and $L_4=\mathcal{O}_{W}(C_{0} + (m+1)f)$. Arguing as above, $D_2 \sim 3C_0 + (2e-m+1)f$.  Since 
   $$(D_2-C_0) \cdot C_0 = (2C_0 + (2e-m+1) f) \cdot C_0 = 1-m,$$
   if $m > 1$, 
   then $2C_0$ would be in the fixed part of $|D_2|$, contradicting Theorem~\ref{thm.Z6.alg.str} (1) (iii). Then $m=1$. Since $m \ge e+1$ (see Notation~\ref{notation.Hirzebruch} (2)), $e=0$. Then 
   we also have  $D_1 \sim 3f$ and $D_3 \sim 2C_0+f$, so we get Case 2 of Part (1) (iii) of the statement, and, in particular, (1) (a) also holds in this case.

   \smallskip
   \noindent \underline{Case 3}: $L_1=L_3= \mathcal{O}_{W}(2C_{0} + (e+1)f)$.

   \noindent In this case, \eqref{eq.regular.splitting} implies $L_{2} = L_{4} = \mathcal{O}_{W}(C_{0} + (m+1)f)$. Theorem~\ref{thm.Z6.alg.str} (1) (iii) implies $D_3$ is reduced, so arguing as in Case 2, we see $m=1$ and $e=0$ and that $D_1 \sim D_2 \sim C_0+2f$ and $D_3 \sim 3C_0$. After swapping $C_0$ and $f$, this case is Case 1 when $m=1$ and $e=0$.

   \smallskip
   \noindent \underline{Case 4}: $L_1= \mathcal{O}_{W}(C_{0} + (m+1)f)$ and $L_{3} = \mathcal{O}_{W}(2C_{0} + (e+1)f)$. 

   \noindent In this case \eqref{eq.alg.str.Z6.L_2L_3}, \eqref{eq.alg.str.Z6.L_1L_4} and  (1) (i), (ii)  imply $L_2=\mathcal{O}_{W}(C_{0} + (m+1)f)$ and $L_4=\mathcal{O}_{W}(2C_{0} + (e+1)f)$. Arguing with $D_1$ as we did with $D_2$ in Case 2, we conclude that $m=1$ and $e=0$. In addition, $D_1 \sim 3C_0$, $D_2 \sim 3f$ and $D_3 \sim C_0+2f$. After swapping $C_0$ and $f$, this is Case 2. This completes the proof of (1) (iii).

   \smallskip
   \noindent The set $\{L_a,D_b\}$ satisfies \eqref{relation.ring.structure} by Theorem~\ref{thm.Z6.alg.str} (iv) and Theorem~\ref{thm.Z6.alg.str} (1) (iii) implies (1) (iv).
   Since $D_1$ meets both $D_2$ and $D_3$, Theorem~\ref{thm.Z6.alg.str} (1) (iii') implies $X$ is singular. The mildest singularities of such a surface $X$ occur when $D_1$, $D_2$ and $D_3$ are smooth and meet transversally. When this happens, 
   \begin{equation*}
\begin{matrix}
    D_1 \cdot D_2  =  (2C_0+(e+1)f)^2  = 4 \\
    D_1 \cdot D_3  =  (2C_0+(e+1)f) \cdot (2m-e+1)f  =  4m-2e+2
\end{matrix}
   \end{equation*}
   in Case 1 and 
      \begin{equation*}
\begin{matrix}
    D_1 \cdot D_2  =  (3f) \cdot (3C_0)  = 9 \\
    D_1 \cdot D_3  =  (3f) \cdot (2C_0+f)   =  6
\end{matrix}
  \end{equation*}
  in Case 2. Then we obtain Part (1) (v) of the statement arguing  as in the proof of Theorem~\ref{thm.P2.Z6} (1) (b).

\smallskip
\noindent To prove (2), note that, by (1) (iii), in Case 1, the divisors $D_1$ and $D_2$ and, in Case 2, the divisor $D_1$, are big and base-point-free, so, by Bertini, for a general element $(D_1, D_2, D_3)$  in $|D_1| \times |D_2| \times |D_3|$, the divisors $D_1$, $D_2$ and $D_3$ are smooth and meet one another transversally. Then the proof mimics the proof of Theorem~\ref{thm.P2.Z6} (2). 
\end{proof}

 \begin{theorem}\label{thm.Hirzebruch.Z6.irregular}
Let $W$ be  the Hirzebruch surface $\mathbb{F}_0$ and let $G=\mathbb Z_6$. 
\begin{enumerate}
    \item
    If $\pi: X \longrightarrow W$  is an abelian canonical cover with group $G$ and $X$ is irregular, then 
    \begin{enumerate}
        \item $W$ is embedded as a rational normal scroll (i.e., $\alpha=1$); 
        \item $X$ is singular, with canonical singularities; precisely, the singular points of $X$ are exactly $4m+8$ singularities of type $A_1$; 
        \item The irregularity of $X$ is $q(x)=2$.
    \end{enumerate}
    \smallskip
    
    \item[] After possibly relabeling $D_{b}$ by  an automorphism of $G$ and $L_{a}$ by the inverse of its dual automorphism on $G^*$ 
    and after possibly swapping $C_0$ and $f$, the building data of $\pi$ is the following:
 \begin{enumerate}
 \item[(i)] $L_{1}=L_{3}=\mathcal O_W(C_0+(m+2)f)$, 
 \item [] $L_{2}=L_{4}=\mathcal O_W(2C_0)$;
 \item[]  $L_{5}=\mathcal O_W(3C_0+(m+2)f)$; 
 
 \item[(ii)] $D_{4}=D_{5}=0$;
 
 \item[(iii)]  $D_{1}  \sim D_{2} \sim 2C_0$, $D_{3} \sim (2m+4)f$; 

  \item[(iv)] The divisor $D_{1}+ D_{2}$ consists of $4$ distinct lines linearly equivalent to $C_0$ and the divisor $D_3$  consists of $2m+4$ distinct lines linearly equivalent to $f$.
\end{enumerate}

\smallskip
 \noindent
The $L_{a}$ and $D_{b}$ in (i), (ii) and (iii) satisfy \eqref{relation.ring.structure}.

\smallskip
 \item Conversely, assume (1) (a) and let $\{L_{a},D_{b}\}$ $(a,b= 1, \dots, 5)$ be as in (i), (ii), (iii) above. Then $\{L_{a},D_{b}\}$ satisfies \eqref{relation.ring.structure} and, if $D_1, D_2, D_3$ satisfies (iv) above, let  $\pi: X \longrightarrow W$ be the abelian cover with group  $G$ and building data $\{L_{a},D_{b}\}$. Then, $\pi$ is canonical, and  $X$ satisfies (1) (b) and (1) (c).
\end{enumerate}
 \end{theorem}

 \begin{proof}
We first prove (1). By Theorem \ref{thm.Z6.alg.str} (1) (i)(ii), we can assume 
\begin{equation}\label{eq.L_5.omega}
 L_5=\omega_W^*(1)=\mathcal O_W((\alpha+2)C_0+(m+2)f)   
\end{equation}
 and we also get (1) (ii). Now we prove (1) (i). By Proposition \ref{prop.split.minimal.degree} (2) (b), $L_1, L_2$ should be equal to one or two of $8$ non trivial line bundles in \eqref{eq.irregular.splitting}, different from $\mathcal O_W((\alpha+2)C_0+(m+2)f)$. However, it is easy to rule out most of the possible combinations for $L_1$ and $L_2$, as we outline now. On the one hand, in the notation of Proposition~\ref{prop.split.minimal.degree} (2) (b), $n_0+n_0'+n_1' >0$. Therefore, if $L_{1}$ equal $\mathcal{O}_{W}(C_{0} + (m+1)f)$ or $\mathcal{O}_{W}((\alpha+1)C_{0}+f)$, then $L_{2}$ cannot be $\mathcal{O}_{W}(C_{0} +(m+1)f)$ or $
\mathcal{O}_{W}((\alpha+1)C_{0}+f)$ (recall that $L_3, L_4$ are determined from $L_1, L_2$ by \eqref{eq.alg.str.Z6.L_1L_4}  \eqref{eq.alg.str.Z6.L_2L_3}, \eqref{eq.L_5.omega} and (1) (ii)). On the other hand, \eqref{relation.ring.structure} implies the linear systems $|L_1^{\otimes 5}  \otimes L_5^*|$,  $|L_1 \otimes L_2^{\otimes 2}  \otimes L_5^*|$ and  $|L_1^{\otimes 2} \otimes L_3  \otimes L_5^*|$
are not empty. This rules out all the combinations of $L_1$ and $L_2$ for which the coefficient of $C_0$ or $f$ in the linear equivalence class of one of the line bundles $L_1^{\otimes 5}  \otimes L_5^*$,  $L_1 \otimes L_2^{\otimes 2}  \otimes L_5^*$ or $L_1^{\otimes 2} \otimes L_3  \otimes L_5^*$ is negative. 
Finally, using \eqref{eq.alg.str.Z6.L_1L_4}, \eqref{eq.alg.str.Z6.L_2L_3}, \eqref{eq.alg.str.Z6.L_4L_4}, \eqref{eq.L_5.omega} and (1) (ii) to compute $D_1$, we can rule out $8$ more combinations of $L_1, L_2$ for which $D_1$ is not effective. 
 At the end, there are only two possibilities left for $L_1, L_2$:

\smallskip
\noindent \underline{Case 1}: $L_{1} = \mathcal{O}_{W}(C_{0} + (m+2)f)$ and $L_{2} =  \mathcal{O}_{W}((\alpha+1)C_{0})$.

\noindent In this case, by \eqref{eq.alg.str.Z6.L_1L_1}, \eqref{eq.alg.str.Z6.L_1L_4}, \eqref{eq.alg.str.Z6.L_2L_3}, \eqref{eq.L_5.omega} and (1) (ii), we have $L_{3} = \mathcal{O}_{W}(C_{0} + (m+2)f)$, $L_{4} = \mathcal{O}_{W}((\alpha+1)C_{0})$ and  $D_{3} \sim (-\alpha+1)C_{0} +(2m+4)f$. Since $D_3$ is effective, $\alpha=1$, so, in this case (1) (a)  is  proved. In addition, \eqref{eq.alg.str.Z6.L_1L_2}, \eqref{eq.alg.str.Z6.L_4L_4} and (1) (ii) imply $D_1 \sim D_2 \sim 2C_0$. Thus, (1) (i), (iii) are proved in this case.

\smallskip
\noindent \underline{Case 2}: $L_{1} = \mathcal{O}_{W}((\alpha+2)C_{0} + f)$ and $L_{2} =  \mathcal{O}_{W}((m+1)f)$.

\noindent Arguing as in Case 1, $L_{3} = \mathcal{O}_{W}((\alpha+2)C_{0} + f)$, $L_{4} = \mathcal{O}_{W}((m+1)f)$ and $D_{3} \sim (2\alpha+4)C_{0} + (1-m)f$. Since $D_3$ is effective, $m=1$, so $\alpha=1$ by Notation~\ref{notation.Hirzebruch} (2). After swapping $C_0$ and $f$, this a special case of Case 1.

\smallskip
\noindent By Theorem~\ref{thm.Z6.alg.str} (1) (iii), $D_1+D_2$ and $D_3$ are reduced, so (1) (iv) follows from (1) (iii). Since $D_1$ and $D_3$ meet, Theorem~\ref{thm.Z6.alg.str} (1) (iii') implies $X$ is singular. Since $D_1, D_2, D_3$ are smooth, $D_1$ and $D_2$ are disjoint, $D_2$ and $D_3$ meet transversally and $D_1$  and $D_3$ also meet transversally at exactly $4m+8$ distinct points, the remaining of (1) (b) follows from \cite[Propositions 3.1, 3.3]{Pardini}. Finally, (1) (c) follows from (1) (i).

\smallskip

\noindent Now we prove (2). That $\{L_{a}, D_{b}\}$  satisfies \eqref{relation.ring.structure} is straightforward. Moreover, (1) (iv) implies Theorem~\ref{thm.Z6.alg.str} (1) (iii) and (1) (i) implies Theorem~\ref{thm.Z6.alg.str} (1) (v), so (2) follows from Theorem~\ref{thm.Z6.alg.str} (2) and the above arguments to prove (1) (b), (c).
 \end{proof}

\subsection{Covers  with group $\mathbf{Z}_2 \times \mathbf{Z}_2 \times  \mathbf{Z}_2$}\label{section.tridouble.rational}

\begin{theorem}\label{thm.tridouble.P2}
 Let $W=\mathbb P^2$, $\mathcal O_W(1)=\mathcal O_{\mathbb P^2}(1)$ and let $G=\mathbf{Z}_2 \times \mathbf{Z}_2 \times  \mathbf{Z}_2$.

 \smallskip
\begin{enumerate}
    \item
    If $\pi: X \longrightarrow W$  is an abelian canonical cover with group $G$, then,  after possibly relabeling $D_{b_1b_2b_3}$ by  an automorphism of $G$ and $L_{a_1a_2a_3}$ by the inverse of its dual automorphism on $G^*$:
 \begin{enumerate}
 \item[(i)] we may assume $L_{111}=\mathcal O_{\mathbb{P}^2}(4)$; 
 \item[] in this case:
     \item[(ii)] $L_{a_1a_2a_3}=\mathcal O_{\mathbb{P}^2}(2)$, for all $(a_1,a_2,a_3) \neq (0,0,0), (1,1,1)$; 
     \item[(iii)] $D_{110}=D_{101}=D_{011}=0$;  $D_{100}, D_{010}, D_{001}, D_{111}$ are conics;
     \item[(iv)] $D_{100}, D_{010}, D_{001}, D_{111}$  are  reduced  without common components.
      \end{enumerate}
        \smallskip
      \item[] The $L_{a_1a_2a_3}$ and $D_{b_1b_2b_3}$ in (i) to (iii)
    satisfy \eqref{eq.bidouble.alg.str}.

      \smallskip
      \noindent If $X$ is smooth, then we also have: 
     \begin{enumerate}  
     \item[(iv')] $D_{100}, D_{010}, D_{001}, D_{111}$ are smooth conics that meet pairwise transversally
     and such no three of them meet.  
 \end{enumerate}

 \noindent
 
 \smallskip

\item 
 \noindent Conversely,
let $\{L_{a_1a_2a_3}, D_{b_1b_2b_3}\}$ $(a_1, a_2, a_3, b_1, b_2, b_3 = 0,1, (a_1,a_2,a_3),$ $(
{b_1},b_2,b_3) \neq (0,0,0))$
as in (i) to (iii). Then, 
$\{L_{a_1a_2a_3}, D_{b_1b_2b_3}\}$ satisfies \eqref{eq.bidouble.alg.str}. If $(D_{100}, D_{010}, D_{001}, D_{111})$ is general in $|\mathcal{O}_{\mathbb{P}^2}(2)|^4$, then the abelian cover $\pi: X \longrightarrow W$ associated to 
$\{L_{a_1a_2a_3}, D_{b_1b_2b_3}\}$ is canonical and $X$ is smooth.
\end{enumerate}
 \end{theorem}

\begin{proof} We first prove (1). 
    Statements  (i) and (ii) follow from  Proposition~\ref{prop.split.minimal.degree} (1), since we may assume, as in Theorem~\ref{thm.splitting} (II) (1), that $L_{111}=\omega_W^*(1)= \mathcal O_{\mathbb P^2}(4)$.  Statement (iii) follow from Theorem~\ref{thm.tridouble.alg.str} (1) (ii). By Proposition~\ref{prop.alg.struct.bidouble}, the $L_{a_1a_2a_3}$ and $D_{abc}$ satisfy \eqref{eq.bidouble.alg.str}. In particular, (iii) and Corollary~\ref{cor.alg.struct.tridouble} imply
     \begin{equation}\label{eq.Dabc}
        \begin{matrix}
        \mathcal O_{W}(D_{100}) =
        L_{100} \otimes L_{101} \otimes  L_{001}^*,\\
         \mathcal O_{W}(D_{010}) =
        L_{010} \otimes L_{011} \otimes  L_{001}^*,\\
         \mathcal O_{W}(D_{001}) =
        L_{001} \otimes L_{011} \otimes  L_{010}^*,\\
         \mathcal O_{W}(D_{111}) =
        L_{100} \otimes L_{010} \otimes  L_{110}^*.
        \end{matrix}
    \end{equation}
     Therefore, by (i) and (ii), $D_{100}, D_{010}, D_{001}$ and $D_{111}$ are conics. They are
    reduced  without common components by 
    Theorem~\ref{thm.tridouble.alg.str} (1) (iii). If, in addition, $X$ is smooth, then Theorem~\ref{thm.tridouble.alg.str} (1) (iii') implies  $D_{100}, D_{010}, D_{001}$ and $D_{111}$ are smooth and meet transversally. 

    \smallskip
    \noindent Now we prove (2). 
    It is  easy to check that $L_{a_1a_2a_3}$ and $D_{b_1b_2b_3}$    satisfy \eqref{eq.bidouble.alg.str} and Theorem~\ref{thm.tridouble.alg.str} (1) (v). If 
   $(D_{100}, D_{010}, D_{001}, D_{111})$ is general, then  $D_{100}, D_{010}, D_{001},$ $D_{111}$ satisfy (1) (iv') above, so  Theorem~\ref{thm.tridouble.alg.str} (1) (iii') holds. Then we conclude by Theorem~\ref{thm.tridouble.alg.str} (2).
\end{proof}

\begin{theorem}\label{thm.scroll.regular}
  Let $W$ be a Hirzebruch surface $\mathbb{F}_e$, embedded as {a rational normal scroll}. 
  Let $G= \mathbb{Z}_2 \times \mathbb{Z}_2 \times  \mathbb{Z}_2$.
\begin{enumerate}
    \item
    If $\pi: X \longrightarrow W$  is an abelian canonical cover with group $G$ and $X$ is regular, then $e=0$ or $e=1$ and, after possibly relabeling $D_{b_1b_2b_3}$ by  an automorphism of $G$ and $L_{a_1a_2a_3}$ by the inverse of its dual automorphism on $G^*$ 
    and after possibly swapping $C_0$ and $f$ when $W=\mathbb{F}_0$, the building data of $\pi$ is the following:
 \begin{enumerate}
 \item[(i)]  $L_{111}=\mathcal O_W(3C_0+(m+e+2)f)$,
 \item[] $L_{100}=L_{010}=L_{001}=\mathcal O_W(C_0+(m+1)f)$, 
 \item[] $L_{110}=L_{101}=L_{011}=\mathcal O_W(2C_0+(e+1)f)$;   
  \smallskip
  
 \item[(ii)] $D_{110}=D_{101}=D_{011}=0$; 
 \item[(iii)]  $D_{100} \sim D_{010} \sim D_{001} \sim 2C_0+(e+1)f$, 
  $D_{111} \sim (2m-e+1)f$;
  \item[(iv)] $D_{100}, D_{010}, D_{001}, D_{111}$ are reduced divisors without common components.  
 \end{enumerate}

  \smallskip
 \noindent
The $L_{a_1a_2a_3}$ and $D_{b_1b_2b_3}$ in (i) to (iv) above satisfy \eqref{eq.bidouble.alg.str}.

  \smallskip
 \noindent If $X$ is smooth, then, in addition, 
  \begin{enumerate}
 \item[(iv')]  $D_{100}, D_{010}, D_{001}, D_{111}$ are smooth divisors such that the intersection of any three of them is empty and  any two of them meet transversally.
  \end{enumerate}

\smallskip
 \color{black}
\item 
 \noindent Conversely, let $e=0,1$ and let $L_{a_1a_2a_3}$ and $D_{b_1b_2b_3}$ be as in (i), (ii), (iii) above. 
 Then, $L_{a_1a_2a_3}$ and $D_{b_1b_2b_3}$ satisfy \eqref{eq.bidouble.alg.str} and,
 \begin{enumerate}
     \item[(a)] if $D_{111}$ is the union of $2m-e+1$ lines and $(D_{100}, D_{010}, D_{001})$ is general in $|2C_0+(e+1)f|^3$, then   (iv') above is satisfied;
     \item[(b)]  if (iv') above is satisfied, then the  abelian cover $\pi: X \longrightarrow W$ with group  $G$ and building data $\{L_{a_1a_2a_3}, D_{b_1b_2b_3}\}$ is canonical and  $X$ is smooth and regular.
 \end{enumerate}

\end{enumerate}
 \end{theorem}

\begin{proof} 
Recall that, by Notation~\ref{notation.Hirzebruch} (2), $\alpha =1$, so $\mathcal O_W(1)=\mathcal O_W(C_0+mf)$. 

\smallskip
\noindent 
We first prove (1). By \eqref{eq.regular.splitting}, there are $3$ line bundles of the splitting of $\pi_*\mathcal O_X$ which are isomorphic to $\mathcal O_W(C_0+(m+1)f)$. Then we may, after  possibly relabeling $D_{b_1b_2b_3}$ by  an automorphism of $G$ and $L_{a_1a_2a_3}$ by the inverse of its dual automorphism on $G^*$, assume 
that $L_{100}=L_{010}=\mathcal O_W(C_0+(m+1)f)$ and $L_{111}=\omega_W^*(1)=\mathcal O_W(3C_0+(m+e+2)f)$. 
 Then   (ii) follows from Theorem~\ref{thm.tridouble.alg.str} (1), (ii). 

 \smallskip
 \noindent Now  $L_{001}$ is isomorphic to either  $\mathcal O_W(C_0+(m+1)f)$ or $\mathcal O_W(2C_0+(e+1)f)$. We split the argument in those two cases.

 \smallskip
 \noindent \emph{\underline{Case  1: $L_{001}=\mathcal O_W(C_0+(m+1)f)$}}: In this case, by 
 \eqref{eq.regular.splitting},
 $L_{110}=L_{101}=L_{011}=\mathcal O_W(2C_0+(e+1)f)$, so, in this case (i) is proved. By Proposition~\ref{prop.alg.struct.bidouble}, the $L_{a_1a_2a_3}$ and the $D_{b_1b_2b_3}$ satisfy \eqref{eq.bidouble.alg.str} and, in particular, \eqref{eq.Dabc}, 
so 
\begin{equation*}
 \begin{matrix}
    D_{100} \sim D_{010} \sim D_{001} \sim 2C_0+(e+1)f, \\
    D_{111} \sim (2m-e+1)f,
    \end{matrix}
\end{equation*} 
so (iii) is proved in this case.
If $e \geq 2$, then $C_0$ is in the fixed part of $|2C_0+(e+1)f|$ and appears with multiplicity $3$ in $D_{100}+D_{010}+D_{001}+D_{111}$.  
This contradicts Theorem~\ref{thm.tridouble.alg.str} (1) (iii). 
Then $e = 0$ or $e=1$. Theorem~\ref{thm.tridouble.alg.str} (1) (iii) also implies (iv) in this case. Finally, Theorem~\ref{thm.tridouble.alg.str} (1) (iii') implies, if $X$ is smooth, condition (iv') in this case.

    \smallskip
 \noindent \emph{\underline{Case  2: $L_{001}=\mathcal O_W(2C_0+(e+1)f)$}}: In this case, by \eqref{relations.ring.structure.Z_2^3},  
 \eqref{eq.regular.splitting}
 and Theorem~\ref{thm.tridouble.alg.str} (1) (ii), $L_{110}=\mathcal O_W(C_0+(m+1)f)$ and 
 $L_{101}=L_{011}=\mathcal O_W(2C_0+(e+1)f$. By Proposition~\ref{prop.alg.struct.bidouble}, the $L_{a_1a_2a_3}$ and the $D_{b_1b_2b_3}$ satisfy \eqref{eq.bidouble.alg.str}. Then, arguing as above we get 
 \begin{equation*}
 \begin{matrix}
 D_{100} \sim D_{010} \sim D_{111} \sim C_0+(m+1)f, \\
 D_{001} \sim 3C_0+(2e-m+1)f. 
     \end{matrix}
\end{equation*} 
Note that $(2C_0+(2e-m+1)f)C_0=-m+1 \le -e$. Then, if $e > 0$,  the divisor $2C_0$ would be in the base locus of $|3C_0+(2e-m+1)f|$ and 
this would contradict Theorem~\ref{thm.tridouble.alg.str} (1) (iii).
If  $e=0$, since $D_{001}$ is effective and $m \ge 1$, then  $m=1$. Therefore
\begin{equation*}
 \begin{matrix}
 D_{100} \sim D_{010} \sim D_{111} \sim C_0+2f, \\
 D_{001} \sim 3C_0. 
     \end{matrix}
\end{equation*} 
Now we swap $C_0$ and $f$. Then, on the one hand, the line bundles $L_{100}, L_{010}, L_{001}$,  $L_{110}, L_{101}, L_{011}$  become respectively the line bundles $L_{011}, L_{101}, L_{001}, L_{110}, L_{010}$, $L_{100}$ of Case 1, when  $e=0$ and $m=1$,  while $L_{111}$ remains the same. On the other hand, the divisors $D_{100}, D_{010}, D_{001}, D_{111}$  become respectively the divisors
$D_{010}, D_{100}, D_{111}, D_{001}$ of Case 1 when $e=0$ and $m=1$. This is a harmless relabeling, because it can be accomplished by applying a suitable automorphism of $G$ to the subindexes of the $D_{b_1b_2b_3}$ and the inverse of the dual of this automorphism 
to the subindexes of the $L_{a_1a_2a_3}$, so it preserves  \eqref{eq.bidouble.alg.str}. Thus Case 2 is a subcase of Case 1. This completes the proof of (1).

    \smallskip
    \noindent Now we prove (2).  Note that part (1) of the theorem states that, for $L_{a_1a_2a_3}$ and $D_{b_1b_2b_3}$ satisfying (i), (ii), (iii),  the relations \eqref{eq.bidouble.alg.str} hold. Since $2C_0+(e+1)f$ is base-point-free, 
    Bertini's theorem 
   implies (2) (a) and, therefore, Theorem~\ref{thm.tridouble.alg.str} (1) (iii').
     By an elementary cohomology computation, (1) (i) implies that $L_{a_1a_2a_3}$ and $D_{b_1b_2b_3}$ also satisfy Theorem~\ref{thm.tridouble.alg.str} (1) (v) (recall that  $\omega_{W}=\mathcal{O}_{W}(-2C_0-(e+2)f)$). Then (2) (b) follows from Theorem~\ref{thm.tridouble.alg.str} (2) (b), the regularity of $X$ coming from the splitting \eqref{eq.regular.splitting}, which holds by (1) (i).
\end{proof}

\begin{theorem}\label{thm.non.scroll.regular}
 Let $W$ be  a Hirzebruch surface $\mathbb{F}_e$, not embedded as a rational normal scroll. If $X$ is regular, then there are no abelian canonical covers  $\pi: X \longrightarrow W$ with group $\mathbf{Z}_2 \times \mathbf{Z}_2 \times  \mathbf{Z}_2$. 
 \end{theorem}

\begin{proof}  Let $C_0$, $f$, $\alpha$ and $m$ be as in Notation~\ref{notation.Hirzebruch}. Recall that, by  Notation~\ref{notation.Hirzebruch} (2), $m \ge \alpha \ge 2$. We will use \eqref{eq.regular.splitting.a>1} in our arguments.

\smallskip
\noindent Arguing as in the proof of Theorem~\ref{thm.scroll.regular} we may assume $L_{111}=\mathcal O_W((\alpha+2)C_0+(m+2)f)$ and, in this case, 
$D_{110}=D_{101}=D_{011}=0$ and 
 $D_{100}, D_{010}, D_{001}, D_{111}$ 
 {would be effective divisors}. 
Continuing using the arguments of the proof of Theorem~\ref{thm.scroll.regular},  we may assume without loss of generality that, in addition to $L_{111} = \omega_{W}^*(1)$, $L_{100}=L_{010}=\mathcal O_W(C_0+(m+1)f)$, so $L_{001}$ is isomorphic either to $\mathcal O_W(C_0+(m+1)f)$ or $\mathcal O_W((\alpha+1)C_0+f)$. We see that neither of these two cases can occur:

 \smallskip
 \noindent \emph{\underline{Case  $L_{001}=\mathcal O_W(C_0+(m+1)f)$}}: In this case, by Proposition~\ref{prop.split.minimal.degree} (2), 
 $L_{110}=L_{101}=L_{011}=\mathcal O_W((\alpha+1)C_0+f)$. By Proposition~\ref{prop.alg.struct.bidouble}, the $L_{a_1a_2a_3}$ and the $D_{b_1b_2b_3}$ satisfy \eqref{eq.bidouble.alg.str} and, in particular, \eqref{eq.Dabc}, 
 so 
\begin{equation*}
    D_{111} \sim (1-\alpha)C_0+(2m+1)f.
\end{equation*} 
Since $\alpha \ge 2$, $D_{111}$ would not be effective, so we would arrive at a contradiction. Therefore, this case cannot occur.

    \smallskip
 \noindent \emph{\underline{Case  {$L_{001}=\mathcal O_W((\alpha+1)C_0+f)$}}}: In this case, by Proposition~\ref{prop.split.minimal.degree} (2), \eqref{relations.ring.structure.Z_2^3} and Theorem~\ref{thm.tridouble.alg.str} (1) (ii), $L_{110}=\mathcal O_W(C_0+(m+1)f)$
 and $L_{101}=L_{011}=\mathcal O_W((\alpha+1)C_0+f)$. By Proposition~\ref{prop.alg.struct.bidouble}, the $L_{a_1a_2a_3}$ and the $D_{b_1b_2b_3}$ satisfy \eqref{eq.bidouble.alg.str}. Then from \eqref{eq.Dabc} we get 
 \begin{equation*}
 D_{001} \sim (2\alpha+1)C_0+(1-m)f. 
\end{equation*}
Since $m \ge 2$, $D_{001}$ would not be effective,  so we would arrive at a contradiction. Therefore, this case cannot occur either. 
\end{proof}

\begin{theorem}\label{thm.Hirzebruch.irregular}
Let $W$ be  the Hirzebruch surface $\mathbb{F}_0$ and 
let $G=\mathbb{Z}_2 \times \mathbb{Z}_2 \times  \mathbb{Z}_2$.
\begin{enumerate}
    \item
    If $\pi: X \longrightarrow W$  is an abelian canonical cover with group $G$ and $X$ is irregular, then 
    $W$ is embedded as a rational normal scroll and, after possibly relabeling $D_{b_1b_2b_3}$ by  an automorphism of $G$ and $L_{a_1a_2a_3}$ by the inverse of its dual automorphism on $G^*$ 
    and after possibly swapping $C_0$ and $f$, the building data of $\pi$ is the following:
 \begin{enumerate}
 \item[(i)] $L_{111}=\mathcal O_W(3C_0+(m+2)f)$; 
 
 \item[(ii)] $D_{110}=D_{101}=D_{011}=0$; 
 \item [(iii)] $D_{100}, D_{010}, D_{001}, D_{111}$ are reduced divisors without common components; 
 
 \item[(iv)]  one of these two cases happens: 
 \item[] \underline{Case 1}: $L_{100}=L_{010}=L_{001}=\mathcal O_W(C_0+(m+2)f)$, $L_{110}=L_{101}=L_{011}=\mathcal O_W(2C_0)$; $D_{100} \sim D_{010} \sim D_{001} \sim 2C_0$ and $D_{111} \sim (2m+4)f$. 
 \item[] In this case, $q(X)=3$.
 \smallskip
 \item[] \underline{Case 2}: $L_{100}=L_{010}=\mathcal O_W(C_0+(m+1)f)$, $L_{001}=\mathcal O_W(C_0+(m+2)f)$, $L_{110}=\mathcal O_W(2C_0)$, $L_{101}=L_{011}=\mathcal O_W(2C_0+f)$; $D_{100} \sim D_{010} \sim 2C_0$, $D_{001} \sim 2C_0+2f$, $D_{111}= (2m+2)f$.
 \item[] In this case, $q(X)=1$.
 \smallskip
 
 \item[(v)] In Case 1 of (iv), $X$ is smooth. 
 
 \item[(vi)] In Case 2 of (iv), if $X$ is smooth, then, in addition to (iii), $D_{001}$ is smooth, passes neither through $D_{100} \cap D_{111}$ nor through ${D_{010}} \cap D_{111}$, and meets each $D_{100}, D_{010}, D_{111}$ transversely. 
\end{enumerate}

\smallskip
 \noindent
The $L_{a_1a_2a_3}$ and $D_{b_1b_2b_3}$ in (i) to (iv) above satisfy \eqref{eq.bidouble.alg.str}.

\smallskip
 \color{black}
\item 
 \noindent Conversely, let 
 $\alpha=1$ and let $L_{a_1a_2a_3}$ and $D_{b_1b_2b_3}$ be as in (i), (ii), (iv) above. Then, $L_{a_1a_2a_3}$ and $D_{b_1b_2b_3}$ satisfy \eqref{eq.bidouble.alg.str}. If 
 \begin{enumerate}
     \item[(a)] 
    in Case 1 above,  $D_{100}+D_{010}+D_{001}$ consists of $6$ distinct lines  and $D_{111}$ consists of $2m+4$ distinct lines; and 
     \item[(b)] in Case 2 above,  $D_{100}+D_{010}$ consists of $4$ distinct lines; $D_{111}$ consists of $2m+2$ distinct lines; and $D_{001}$ is general in $|2C_0+2f|$; 
      \end{enumerate}
     \item[] then  the abelian  cover $\pi: X \longrightarrow W$ with group  $G$ and  building data $L_{a_1a_2a_3}, D_{b_1b_2b_3}$ is canonical and $X$ is smooth and irregular, with $q(X)=3$ if (a) holds, and $q(X)=1$ if (b) holds. 
\end{enumerate}
 \end{theorem}

\noindent We will use the following lemma to prove Theorem~\ref{thm.Hirzebruch.irregular}:

 \begin{lemma}\label{lemma.independence.characters}
     Let $\mathcal X, \mathcal Y$ be 
    irreducible algebraic 
     varieties and let $\pi: \mathcal X \longrightarrow \mathcal Y$
be a 
{flat} abelian cover with group $G = \mathbb Z_2^r$. If $\chi_1, \chi_2$ and $\chi_3$ are three  distinct, nontrivial characters of $G^*$ 
such that $|L_{\chi_1} \otimes L_{\chi_2} \otimes L_{\chi_3}^*|=\emptyset$, then $\chi_1, \chi_2$ and $\chi_3$ are independent.
  \end{lemma}

  \begin{proof}
      Since $G$ is a finite abelian group, $G^*$ is isomorphic to $G$. If $\chi_1, \chi_2$ and $\chi_3$ are not independent, then $<\chi_1, \chi_2, \chi_3>$ is isomorphic to $\mathbb Z_2 \times \mathbb Z_2$, so $\chi_1\chi_2= \chi_3$. 
      Then, the relations in \eqref{eq.bidouble.alg.str} imply that there is an effective divisor in   $|L_{\chi_1} \otimes L_{\chi_2} \otimes L_{\chi_3}^*|$  and this contradicts our hypothesis. 
  \end{proof}

  \begin{proof} \emph{of Theorem~\ref{thm.Hirzebruch.irregular}}.
    We first prove (1). Parts (ii) and (iii) of (1) follow from Theorem~\ref{thm.tridouble.alg.str} (1) (ii), (iii). By Proposition~\ref{prop.split.minimal.degree} (2) (b), $e=0$ and $\pi_*\mathcal O_X$ splits as in \eqref{eq.irregular.splitting}. More precisely, $\pi_*\mathcal O_X$ has 19 non isomorphic possible splittings that correspond to the 19 (ordered) partitions $(n_0, n_0', n_1, n_1')$ of $3$ by non negative integers $n_0, n_0', n_1, n_1'$ such that $n_0 + n_0' + n_1' > 0$. %
    Having in account Lemma~\ref{lemma.independence.characters}, after  possibly relabeling $D_{b_1b_2b_3}$ by  an automorphism of $G$ and $L_{a_1a_2a_3}$ by the inverse of its dual automorphism on $G^*$, for these different partitions we may assume  $L_{111}=\omega_W^*(1)=\mathcal O_W((\alpha+2)C_0+(m+2)f)$, so (1) (i) is proved, and  that $L_{100}$ and $L_{010}$, in the cases below, are as follows: 

 \smallskip
    \noindent {\emph{Case Ia} ($n_0=3$)}:
$L_{100}=L_{010}=\mathcal O_W((m+1)f)$.

   \smallskip
    \noindent {\emph{Case Ib}  ($n_0'=3$):} $L_{100}=L_{010}=\mathcal O_W((m+2)f)$.

 \smallskip
    \noindent {\emph{Case Ic} ($n_1'=3$):} $L_{100}=L_{010}=\mathcal O_W(C_0+(m+2)f)$.

 \smallskip
    \noindent {\emph{Case IIa} ($n_0=2$):} $L_{100}=L_{010}=\mathcal O_W((m+1)f)$. 

    \smallskip
    \noindent {\emph{Case IIb}  ($n_0'=2$):} $L_{100}=L_{010}=\mathcal O_W((m+2)f)$.

   \smallskip
    \noindent {\emph{Case IIc}  ($n_1=2$, $n_0=1$):} $L_{100}=L_{010}=\mathcal O_W(C_0+(m+1)f)$.

\smallskip
    \noindent \emph{Case IId}  ($n_1=2$, $n_0'=1$): $L_{100}=L_{010}=\mathcal O_W(C_0+(m+1)f)$

\smallskip
    \noindent \emph{Case IIe}  ($n_1=2$, $n_1'=1$): $L_{100}=L_{010}=\mathcal O_W(C_0+(m+1)f)$

    \smallskip
    \noindent {\emph{Case IIf}  ($n_1'=2$):} $L_{100}=L_{010}=\mathcal O_W(C_0+(m+2)f)$

\smallskip
    \noindent {\emph{Case IIIa}  ($n_0=n_0'=1$):} $L_{100}=\mathcal O_W((m+1)f), L_{010}=\mathcal O_W((m+2)f)$.

\smallskip
    \noindent {\emph{Case IIIb}  ($n_1=n_1'=1$):} $L_{100}=\mathcal O_W(C_0+(m+1)f), L_{010}=\mathcal O_W(C_0+(m+2)f)$. 

\smallskip
\noindent Once it is set what $L_{100}, L_{010}, L_{111}$ are, Formulae
 \eqref{eq.irregular.splitting}, \eqref{relations.ring.structure.Z_2^3} and item (ii) in the statement allow several possibilities for $L_{001}$: two possibilities in each one of Cases Ia, Ib, Ic, IIc, IId, IIe;  six possibilities in each one of Cases IIa, IIb, IIf; and four possibilities in each one of Cases IIIa, IIIb. Once $L_{001}$ is fixed,  the remaining line bundles $L_{110}, L_{101}, L_{011}$ are determined by  \eqref{relations.ring.structure.Z_2^3} and  item (ii) in the statement. After all this process, there are in total  38 possibilities that can occur for the line bundles $L_{a_1a_2a_3}$, but not all of them actually happen because the divisors $D_{100}, D_{010}, D_{001}, D_{111}$, whose linearly equivalence class depends on the $L_{a_1a_2a_3}$, should be effective. Precisely, \eqref{eq.Dabc} implies that, in Cases Ia, Ib, IIa, IIb, IId, IIf, IIIa and IIIb, there are no effective divisors linearly equivalent to  one of 
$D_{100}, D_{010}, D_{001}, D_{111}$, so these cases do not occur. We now deal with Cases Ic, IIc and IIe.

\smallskip
    \noindent  Recall that, in Case Ic, $L_{111}=\mathcal O_W((\alpha+2)C_0+(m+2)f)$ and $L_{100}=L_{010}=\mathcal O_W(C_0+(m+2)f)$. By Formulae
 \eqref{eq.irregular.splitting}, \eqref{relations.ring.structure.Z_2^3} and  item (ii) in the statement, $L_{001}$ is isomorphic to $\mathcal O_W(C_0+(m+2)f)$ or to $\mathcal O_W((\alpha+1)C_0)$  and the remaining among the line bundles $L_{a_1a_2a_3}$ are determined by \eqref{relations.ring.structure.Z_2^3} and  item (ii) in the statement. If $L_{001}=\mathcal O_W((\alpha+1)C_0)$, it follows from \eqref{eq.Dabc} that there would be  no effective divisors linearly equivalent to $D_{001}$, hence this option for $L_{001}$ cannot occur. If $L_{001}=\mathcal O_W(C_0+(m+2)f)$, then  \eqref{eq.Dabc} implies 
$D_{111} \sim (1-\alpha)C_0+(2m+4)f$. Then $|(1-\alpha)C_0+(2m+4)f|$ is non empty if and only if $\alpha\leq 1$, i.e., if and only if $\alpha=1$, hence $W$ is a rational normal scroll in this case.  Then $D_{111} \sim (2m+4)f$ and \eqref{eq.Dabc} also implies $D_{100} \sim D_{010} \sim D_{001} \sim 2C_0$. In addition, \eqref{relations.ring.structure.Z_2^3} and  item (ii) in the statement  implies $L_{110}=L_{101}=L_{011}=\mathcal O_W(2C_0)$. Thus we get Case 1 of Theorem~\ref{thm.Hirzebruch.irregular} (1) (iv).

\smallskip
\noindent   Recall that, in Case IIe, $L_{111}=\mathcal O_W((\alpha+2)C_0+(m+2)f)$ and
$L_{100}=L_{010}=\mathcal O_W(C_0+(m+1)f)$. By Formulae
 \eqref{eq.irregular.splitting}, \eqref{relations.ring.structure.Z_2^3} and  item (ii) in the statement, $L_{001}$ is isomorphic to $\mathcal O_W(C_0+(m+2)f)$ or to 
$\mathcal O_W((\alpha+1)C_0)$ and the remaining among the line bundles $L_{a_1a_2a_3}$ are determined by \eqref{relations.ring.structure.Z_2^3} and  item (ii) in the statement. If $L_{001}$ is isomorphic to $\mathcal O_W((\alpha+1)C_0)$, then \eqref{eq.Dabc} implies that  there are no effective divisors linearly equivalent to $D_{001}$, so this option for $L_{001}$ cannot happen. If $L_{001}$ is isomorphic  to 
$\mathcal O_W(C_0+(m+2)f)$, 
then  \eqref{eq.Dabc} implies 
$D_{111} \sim (1-\alpha)C_0+(2m+2)f$. Then $|(1-\alpha)C_0+(2m+2)f|$ is non empty if and only if $\alpha\leq 1$, i.e., if and only if $\alpha=1$, hence $W$ is a rational normal scroll in this case.  Then $D_{111} \sim (2m+2)f$ and \eqref{eq.Dabc} also implies $D_{100} \sim D_{010} \sim 2C_0$ and $D_{001} \sim 2C_0+2f$.  In addition, \eqref{relations.ring.structure.Z_2^3} and  item (ii) in the statement implies $L_{110}=\mathcal O_W(2C_0)$ and $L_{101}=L_{011}=\mathcal O_W(2C_0+f)$. Thus we get Case 2 of Theorem~\ref{thm.Hirzebruch.irregular} (1) (iv).

\smallskip
\noindent   Recall that, in Case IIc, $L_{111}=\mathcal O_W((\alpha+2)C_0+(m+2)f)$ and
$L_{100}=L_{010}=\mathcal O_W(C_0+(m+1)f)$. By Formulae
 \eqref{eq.irregular.splitting}, \eqref{relations.ring.structure.Z_2^3} and  item (ii) in the statement, $L_{001}$ is isomorphic to $\mathcal O_W((m+1)f)$ or to 
$\mathcal O_W((\alpha+2)C_0+f)$. If $L_{001}$ is isomorphic to $\mathcal O_W((m+1)f)$, then \eqref{eq.Dabc} implies that  there are no effective divisors linearly equivalent to $D_{111}$, so this option for $L_{001}$ cannot happen. If $L_{001}$ is isomorphic  to 
$\mathcal O_W((\alpha+2)C_0+f)$, 
then  \eqref{eq.Dabc} implies 
$D_{001} \sim (2\alpha+2)C_0+(1-m)f$. Then $|(2\alpha+2)C_0+(1-m)f|$ is non empty if and only if $m \leq 1$, i.e., if and only if $m=1$. By Notation~\ref{notation.Hirzebruch}, $\alpha=1$. Then $D_{001} \sim 4C_0$, and \eqref{eq.Dabc} also implies $D_{100} \sim D_{010} \sim 2f$ and $D_{111} \sim 2C_0+2f$.
 Now we swap $C_0$ and $f$. Then, on the one hand, the line bundles $L_{100}, L_{010}, 
 {L_{001}}, L_{110}, L_{101}, L_{011}$ 
 become respectively the line bundles $L_{011}, L_{101}, L_{001}, L_{110}, L_{010}, L_{100}$ of Case IIe, when  $m=1$, while $L_{111}$ remains the same and, therefore, equal to the line bundle $L_{111}$  of Case IIe, $m=1$. On the other hand, the divisors $D_{100}, D_{010}, D_{001}, D_{111}$  become, respectively, the divisors
$D_{010}, D_{100}, D_{111}, D_{001}$ of Case IIe when $m=1$. By the same arguments at the end of the proof of Theorem~\ref{thm.scroll.regular} (1), this is harmelss relabeling, 
so Case IIc is a subcase of Case IIe. This completes the proof of (1) (iv).

\smallskip 
\noindent Item (iii) of (1) imply that, in Case 1 of (1) (iv), $D_{100}+D_{010}+D_{001}$ consists of $6$ distinct lines linearly equivalent to $C_0$ and $D_{111}$ consists of $2m+4$ distinct lines linearly equivalent to $f$. Then, necessarily, the intersections of any three among $D_{100}, D_{010}, D_{001}, D_{111}$ are empty and the pairwise intersections of $D_{100}, D_{010}, D_{001}, D_{111}$ are empty or transverse, so, arguing as in the proof of Theorem~\ref{thm.tridouble.alg.str} (2) (b), we conclude that $X$ is smooth in this case. This proves (1) (v). Finally, (1) (vi) follows follows from Theorem~\ref{thm.tridouble.alg.str} (1) (iii').

\smallskip
\noindent Now we prove (2). If (2) (a) is satisfied, then
Theorem~\ref{thm.tridouble.alg.str} (1) (iii') holds. 
Since $2C_0+2f$ is very ample, if (2) (b) is satisfied, then  Theorem~\ref{thm.tridouble.alg.str} (1) (iii') holds by Bertini theorem. Arguing as in the proof of Theorem~\ref{thm.scroll.regular} (2), we conclude that $X$ is smooth. Since  the $L_{a_1a_2a_3}$ satisfy (1) (iv), the claims on the irregulary of $X$ follow. 
 \end{proof}

 \subsection{Covers with group $\mathbb{Z}_2 \times \mathbb{Z}_4$}\label{section.Z2Z4.rational}

\begin{theorem}\label{thm.P2.Z2Z4}
Let $W=\mathbb P^2$ and $\mathcal{O}_W(1)=\mathcal{O}_{\mathbb P^2}(1)$.
Let $G=\mathbf{Z}_2 \times \mathbf{Z}_4$.
 
 \smallskip
\begin{enumerate}
    \item
    If $\pi: X \longrightarrow W$  is an abelian canonical cover with group $G$, after possibly relabeling $D_{b_1b_2}$ by  an automorphism of $G$ and $L_{a_1a_2}$ by the inverse of its dual automorphism on $G^*$, 
     then $X$ and the building data of $\pi$ are as follows:
 \begin{enumerate}
 \item[(i)] $L_{13}=\mathcal O_{\mathbb{P}^2}(4)$; 
 \item[] $L_{a_1a_2}=\mathcal O_{\mathbb{P}^2}(2)$, for all $(a_1,a_2) \neq (0,0), (1,3)$; 
     \item[(ii)] $D_{02}=D_{03}=D_{11}=D_{12}=0$;
     \item[(iii)] $D_{01}, D_{10}, D_{13}$ are conics;
     \item[(iv)] $D_{01}, D_{10}, D_{13}$ are reduced divisors without common components.  
     \item[(v)] $X$ is singular  and the mildest possible set of singular points of such a surface $X$ consists of $4$ singularities of type $A_1$.

      \smallskip
 \noindent
The $L_{a_1a_2}$ and $D_{b_1b_2}$ in (i) to (iv) above satisfy \eqref{relation.ring.structure}. 

 \end{enumerate}
 \smallskip

\item 
 \noindent Conversely,  let $L_{a_1a_2}$ and $D_{b_1b_2}$ be as in (i), (ii), (iii) above. If the element $(D_{01}, D_{10}, D_{13})$ is general in $|\mathcal O_{\mathbb{P}^2}(2)| \times |\mathcal O_{\mathbb{P}^2}(2)| \times |\mathcal O_{\mathbb{P}^2}(2)|$, then
 the 
 abelian  cover $\pi: X \longrightarrow W$ with group  $G$ and building data $L_{a_1a_2}, D_{b_1b_2}$ is canonical and $X$  has exactly $4$ singularities of type $A_1$.
\end{enumerate}
 \end{theorem}

\begin{proof} We first prove (1). 
    Statement  (i)  follows from  Proposition~\ref{prop.split.minimal.degree} (1), since we may assume, by Theorem~\ref{thm.Z2Z4.alg.str} (1) (i), that $L_{13}=\omega_W^*(1)= \mathcal O_{\mathbb P^2}(4)$. Theorem~\ref{thm.Z2Z4.alg.str} (1) (ii) then says that $D_{03}= D_{11} = D_{12}=0$. This together with \eqref{eq.alg.str.L_01L_02}, \eqref{eq.alg.str.L_01L_11}, \eqref{eq.alg.str.L_02L_02}, \eqref{eq.alg.str.L_10L_12} and (i),  implies $D_{02}=0$ and (iii). Statement (iv) follows from Theorem~\ref{thm.Z2Z4.alg.str} (1) (iii). Since $D_{01}$
 and $D_{13}$ are conics, they meet; since $H_{01}$ and $H_{13}$ do not fulfill \cite[Proposition 3.1 iii) a)]{Pardini}, $X$ is singular. To have a surface $X$ with the mildest possible singularities the best we can do is to choose $D_{01}, D_{10}, D_{13}$ smooth, meeting each other transversally and such that the intersection of the three of them is empty (this can be achieved by
 taking $(D_{01}, D_{10}, D_{13})$ general in $|\mathcal O_{\mathbb{P}^2}(2)| \times |\mathcal O_{\mathbb{P}^2}(2)| \times |\mathcal O_{\mathbb{P}^2}(2)|$). In this case, \cite[Proposition 3.1]{Pardini} and \cite[Proposition 3.3]{Pardini} complete the proof of (v).  
 
    \smallskip
    \noindent Now we prove (2). Note that, since $L_{a_1a_2}$ and $D_{b_1b_2}$ satisfy (i), (ii), (iii) of (1), it follows that $L_{a_1a_2}$ and $D_{b_1b_2}$ satisfy \eqref{relation.ring.structure} (see the statement of part 1 of the theorem).

    \noindent As explained above, if  $(D_{01}, D_{10}, D_{13})$ is general in $|\mathcal O_{\mathbb{P}^2}(2)| \times |\mathcal O_{\mathbb{P}^2}(2)| \times 
    |\mathcal O_{\mathbb{P}^2}(2)|$, then 
    $D_{01}, D_{10}, D_{13}$ are smooth, meet each other transversally and  the intersection of the three of them is empty.
    The cohomology of the line bundles of $\mathbb{P}^2$ implies that $L_{a_1a_2}$ and $D_{b_1b_2}$ satisfy Theorem~\ref{thm.Z2Z4.alg.str} (1) (v). Then, taking into account that $A_1$ singularities are locally Gorenstein,  Theorem~\ref{thm.Z2Z4.alg.str}, \cite[Propositions 3.1, 3.3]{Pardini} and \cite[Proposition 3.3]{Pardini} complete the proof of (2). 
\end{proof}

\begin{theorem}\label{thm.Hirzebruch.Z2Z4.regular}
  Let $W$ be a Hirzebruch surface $\mathbb{F}_e$ and 
  let $G= \mathbb{Z}_2 \times \mathbb{Z}_4$.
\begin{enumerate}
    \item
    If $\pi: X \longrightarrow W$  is an abelian canonical cover with group $G$ and $X$ is regular, then $e=0$ or $e=1$; $\alpha=1$ (i.e., $W$ is embedded by $|C_0+mf|$, hence, a rational normal scroll); and, after possibly relabeling $D_{b_1b_2}$ by  an automorphism of $G$ and $L_{a_1a_2}$ by the inverse of its dual automorphism on $G^*$ 
     and after possibly swapping $C_0$ and $f$ when $W=\mathbb{F}_0$, the building data of $\pi$ is the following:
 \begin{enumerate}
 \item[(i)]  $L_{13}=\mathcal O_W(3C_0+(m+e+2)f)$,
 \item[] $L_{10}=L_{11}=L_{12}=\mathcal O_W(C_0+(m+1)f)$, 
 \item[] $L_{01}=L_{02}=L_{03}=\mathcal O_W(2C_0+(e+1)f)$;   
  \smallskip
  
 \item[(ii)] $D_{02}=D_{03}=D_{11}=D_{12}=0$;  

 \item[(iii)] $D_{01} \sim D_{13}  \sim 2C_0+(e+1)f$, 
  $D_{10} \sim (2m-e+1)f$.

    \item[(iv)] $D_{01}, D_{10}, D_{13}$ are reduced divisors without common components.  

     \item[(v)] $X$ is singular  and the mildest possible set of singular points of such a surface $X$ consists of $4$ singularities of type $A_1$.

 \smallskip
 \noindent
The $L_{a_1a_2}$ and $D_{b_1b_2}$ in (i), (ii) and (iii) satisfy \eqref{relation.ring.structure}.

 \end{enumerate}

\smallskip
 \color{black}
\item 
 \noindent Conversely, let $e=0,1$, let $\alpha=1$ and let $L_{a_1a_2}$ and $D_{b_1b_2}$ be as in (i), (ii), (iii) above. If $D_{10}$  consists of $2m-e+1$ lines and $(D_{01},  D_{13})$ is general in $|2C_0+(e+1)f|^2$, 
 then
the
 abelian cover $\pi: X \longrightarrow W$ with group  $G$ and building data $L_{a_1a_2}, D_{b_1b_2}$ 
is canonical and $X$ is regular with exactly $4$ singularities of type $A_1$.

\end{enumerate}
 \end{theorem}

 \begin{proof}
First we prove (1). 
    By Theorem~\ref{thm.Z2Z4.alg.str} (1) we may assume $L_{13}=\mathcal O_W((\alpha+2)C_0+(m+e+2)f)$. By \eqref{eq.regular.splitting} and \eqref{eq.regular.splitting.a>1} the  line bundles $L_{01}, L_{02}, L_{03}$ are equal to $\mathcal O_W(C_0+(m+1)f)$ or $\mathcal O_W((\alpha+1)C_0+(e+1)f)$, so there are eight possibilities for the triple $L_{01}, L_{02}, L_{03}$. 
     By Theorem~\ref{thm.Z2Z4.alg.str} (1) (ii) and \eqref{eq.alg.str.L_01L_03}, \eqref{eq.alg.str.L_02L_02}, 
    \begin{equation}\label{eq.D_02}
        \mathcal O_W(D_{02})=L_{01} \otimes L_{03} \otimes L_{02}^{-2}.
    \end{equation}
Since $D_{02}$ is effective and $\alpha \geq 1$ and $m \geq e+1$,  \eqref{eq.D_02} implies that the only possible choices among the eight are $L_{01}=L_{02}=L_{03}=\mathcal O_W(C_0+(m+1)f)$ or  $L_{01}=L_{02}=L_{03}=\mathcal O_W((\alpha+1)C_0+(e+1)f)$. In both cases, $D_{02}=0$. Then, this and Theorem~\ref{thm.Z2Z4.alg.str} (1) (ii) imply (1) (ii).
We argue the two cases separately. 

\smallskip
\noindent \underline{\emph{Case 1}}: $L_{01}=L_{02}=L_{03}=\mathcal O_W((\alpha+1)C_0+(e+1)f)$. 

\noindent In this case, \eqref{eq.regular.splitting} and \eqref{eq.regular.splitting.a>1}
imply  $L_{10}=L_{11}=L_{12}=\mathcal O_W(C_0+(m+1)f)$. By (ii) and \eqref{eq.alg.str.L_10L_12}, 
$$D_{10} \sim (1-\alpha)C_0 + (2m+1-e)f.$$ Since $D_{10}$ is effective and $\alpha \ge 1$, we have $\alpha=1$. Then  (ii) and  \eqref{eq.alg.str.L_01L_02},  \eqref{eq.alg.str.L_01L_03} imply 
$$D_{01} \sim D_{13} \sim 2C_0+(e+1)f,$$
so (iii) is proved in this case.
Then $D_{01} \cdot C_0=  D_{13}\cdot C_0= 1-e$, so, if $e \ge 2$, then $C_0$ is contained in both $D_{01}$ and $D_{13}$. This is not possible by Theorem~\ref{thm.Z2Z4.alg.str} (1) (iii).
Then $e=0, 1$. 
Theorem~\ref{thm.Z2Z4.alg.str} (1) (iii) also yields (1) (iv). Theorem~\ref{thm.Z2Z4.alg.str} (1) (iv) implies that $L_{a_1a_2}$ and $D_{b_1b_2}$ in (i), (ii) and (iii) satisfy \eqref{relation.ring.structure}. 

\smallskip
\noindent Finally, the proof of (v) 
follows the same lines as the proof of (v) 
of Theorem~\ref{thm.P2.Z2Z4}, taking into account that $2C_0 + (e+1)f$ is base-point-free and big and that the projection from a smooth curve in $|2C_0 + (e+1)f|$ to $C_0$ has a finite number of ramification points. Thus, if $D_{10}$ consists of $2m-e+1$ distinct lines and $(D_{01}, 
D_{13})$ is general in $|D_{01}| 
\times |D_{13}|$, then
 $D_{01}, D_{13}$ are smooth and meet each other transversally; $D_{10}$  meets each $D_{01}, D_{13}$  transversally and  do not pass through $D_{01} \cap D_{13}$.

\smallskip
\noindent \underline{\emph{Case 2}}: $L_{01}=L_{02}=L_{03}=\mathcal O_W(C_0+(m+1)f)$. 

\noindent In this case, \eqref{eq.regular.splitting} and \eqref{eq.regular.splitting.a>1}
imply  $L_{10}=L_{11}=L_{12}=\mathcal O_W((\alpha+1)C_0+(e+1)f)$. Then \eqref{eq.alg.str.L_10L_12} implies $$D_{10} \sim (2\alpha+1)C_0 + (2e+1-m)f,$$
so $$(D_{10}-C_0) \cdot C_0 = 2(1-\alpha)e + (1-m),$$ which is negative if $e \ge 1$ (since $m \geq e+1$) or if $e=0$ and $m > 1$. Thus, unless $e=0$ and $m=1$, $2C_0$ would be in  the fixed locus of $|D_{10}|$ and this would contradict  Theorem~\ref{thm.Z2Z4.alg.str} (1) (iii).
Therefore $e=0$ and $m=1$ and, by Notation~\ref{notation.Hirzebruch} (2), $\alpha=1$, so $$D_{10} \sim 3C_0  ; \ D_{01} \sim D_{13} \sim C_0+2f.$$
After swapping $C_0$ and $f$ in $\mathbb F_0$, Case 2 becomes a particular subcase of Case 1 (the subcase in which $e=0$ and $m=1$). This completes the proof of (1). 

\smallskip
\noindent The proof of (2) goes along the same lines of the proof of Theorem~\ref{thm.P2.Z2Z4} (2). 
\end{proof} 

\begin{theorem}\label{thm.Hirzebruch.Z2Z4.irregular}
Let $W$ be the Hirzebruch surface $\mathbb{F}_0$ 
and let $G=\mathbb{Z}_2 \times   \mathbb{Z}_4$.
\begin{enumerate}
    \item
    If $\pi: X \longrightarrow W$  is an abelian canonical cover with group $G=\mathbb{Z}_2 \times   \mathbb{Z}_4$ and $X$ is irregular, then 
    $\alpha=1$ (i.e., $W$ is embedded by $|C_0+mf|$, hence, a rational normal scroll); and, after possibly relabeling $D_{b_1b_2}$ by  an automorphism of $G$ and $L_{a_1a_2}$ by the inverse of its dual automorphism on $G^*$ 
    and after possibly swapping $C_0$ and $f$, the surface $X$ and the building data of $\pi$ are as follows:
 \begin{enumerate}
 \item[(i)] $L_{13}=\mathcal O_W(3C_0+(m+2)f)$; 
 
 \item[(ii)] $D_{03}=D_{11}=D_{12}=0$; 
 
 \item[(iii)]  one of these three cases happens: 
 \item[] \underline{Case 1}: $L_{01}=L_{02}=L_{03}=\mathcal O_W(2C_0)$, $L_{10}=L_{11}=L_{12}=\mathcal O_W(C_0+(m+2)f)$; $D_{01}  \sim D_{13} \sim 2C_0$, $D_{02}=0$ and $D_{10} \sim (2m+4)f$. 
 \item[] \underline{Case 2}: 
 {$L_{01}=L_{03}=\mathcal O_W(2C_0+f)$, $L_{02}=\mathcal O_W(2C_0)$, $L_{10}=L_{12}=\mathcal O_W(C_0+(m+1)f)$, $L_{11}=\mathcal O_W(C_0+(m+2)f)$; $D_{01} \sim D_{13} \sim 2C_0$, $D_{02} \sim 2f$, $D_{10} \sim (2m+2)f$.}
 \item[] \underline{Case 3}: $m=1$; $L_{01}=L_{11} =\mathcal O_W(C_0+2f)$, $ L_{02}=L_{12} =\mathcal O_W(2C_0+f)$, $L_{03}=\mathcal O_W(3C_0+f)$, $L_{10} =\mathcal O_W(2f)$; $D_{01} \sim 4C_0$, $D_{02} \sim f$, $D_{10} \sim D_{13} \sim 2f$. 

 \item[(iv)] The divisor $D_{01}+ D_{02} + D_{10} + D_{13}$ consists of distinct lines, linearly equivalent to $C_0$ or $f$.

 \item[(v)] In Case 1 of (iii), $X$ is smooth and $q(X)=3$.

 \item[(vi)] In Cases 2 and 3 of (iii), $X$ has 16 singularities of type $A_1$ and $q(X)=1$. 
\end{enumerate}

\smallskip
 \noindent
The $L_{a_1a_2}$ and $D_{b_1b_2}$ in (i), (ii) and (iii) satisfy \eqref{relation.ring.structure}. 
\smallskip

 \color{black}
\item 
 \noindent Conversely, 
 let $L_{a_1a_2}$ and $D_{b_1b_2}$ be as in (i), (ii), (iii) above. If  (iv) above holds, 
 then
 the abelian  cover $\pi: X \longrightarrow W$ with group  $G$ and  building data $L_{a_1a_2}, D_{b_1b_2}$ is canonical and $X$ is irregular, with $q(X)$ and the singularities of $X$ as in (v) or (vi) above, accordingly.
\end{enumerate}
 \end{theorem}

\begin{proof}
If $a_1=i \in \{0,1\}$ and $a_2 = j \in \{0,1,2,3\}$, let $\alpha_{ij}, \beta_{ij}$ be non negative integers such that $L_{a_1a_2}=\mathcal O_W(\alpha_{ij}C_0+\beta_{ij}f)$. Recall that, by Proposition~\ref{prop.split.minimal.degree} (2) (b), $e=0$. By Theorem~\ref{thm.Z2Z4.alg.str} (1) (i), we may assume (1) (i), so $\alpha_{13}=\alpha+2$ and $\beta_{13}=m+2$. Theorem~\ref{thm.Z2Z4.alg.str} (1) (ii) implies (1) (ii). From \eqref{eq.alg.str.L_01L_02}, \eqref{eq.alg.str.L_01L_03}, \eqref{eq.alg.str.L_01L_12},  \eqref{eq.alg.str.L_02L_02}, \eqref{eq.alg.str.L_02L_03}, \eqref{eq.alg.str.L_03L_10}, \eqref{eq.alg.str.L_10L_12} and  (1) (ii), we get
 \begin{equation}\label{eq.divisors.Z2Z4}
     \begin{matrix}
 \mathcal O_W(D_{01}) & = &  L_{02} \otimes L_{03} \otimes L_{01}^*\\
\mathcal O_W(D_{02}) & = &  L_{01} \otimes L_{03} \otimes L_{02}^{-2} \\
\mathcal O_W(D_{10}) & = &  L_{13}^{\otimes 2} \otimes L_{01}^* \otimes L_{02}^* \otimes L_{03}^* \\
\mathcal O_W(D_{13}) & = &  L_{01} \otimes L_{02} \otimes L_{03}^* 
     \end{matrix}
     \end{equation}
Since $D_{01}, D_{02}, D_{10}, D_{13}$ are effective divisors on $W$,  \eqref{eq.divisors.Z2Z4} yields 
\begin{equation}\label{ineq.a}
    \begin{matrix}
        \alpha_{02} + \alpha_{03} - \alpha_{01} & \ge & 0 \\
       \alpha_{01} + \alpha_{03} - 2\alpha_{02} & \ge & 0 \\ 
            2\alpha+4 -\alpha_{01} - \alpha_{02} - \alpha_{03} & \ge & 0 \\
                 \alpha_{01} + \alpha_{02} - \alpha_{03} & \ge & 0 \\
    \end{matrix}
\end{equation}
and 
\begin{equation}\label{ineq.b}
    \begin{matrix}
        \beta_{02} + \beta_{03} - \beta_{01} & \ge & 0 \\
       \beta_{01} + \beta_{03} - 2\beta_{02} & \ge & 0 \\ 
            2m+4 -\beta_{01} - \beta_{02} - \beta_{03} & \ge & 0 \\
                 \beta_{01} + \beta_{02} - \beta_{03} & \ge & 0 \\
    \end{matrix}
\end{equation}
We divide the proof of (1) (iii) in several steps:

\smallskip
\noindent \underline{\emph{Step 1.}} $a_{02} \neq \alpha+2$ and $\beta_{02} \neq m+2$:

\noindent Assume $\alpha_{02} = \alpha+2$. Then \eqref{ineq.a} implies
$$2\alpha + 4 \ge \alpha_{01} + \alpha_{02} + \alpha_{03} \ge 3\alpha_{02} = 3\alpha + 6,$$ that leads to 
$\alpha \le -2$, a contradiction. Analogously, it follows from \eqref{ineq.b} that  $\beta_{02} \neq m+2$.

\smallskip

\noindent \underline{\emph{Step 2.}} The case $\alpha \ge 2$ cannot occur: 

\noindent Let us assume $\alpha \geq 2$ and arrive at a contradiction. We show first that $\alpha_{02}= 0$ or $1$. Indeed, suppose the contrary. By \eqref{eq.irregular.splitting} and Step 1, $\alpha_{02} = \alpha+1$, so \eqref{ineq.a} implies
$$2\alpha + 4 \ge \alpha_{01} + \alpha_{02} + \alpha_{03} \ge 3\alpha_{02} = 3\alpha + 3,$$
but this implies $\alpha \le 1$. Since $\alpha_{02}= 0$ or $1$, \eqref{eq.irregular.splitting} and Step 1 imply $\beta_{02} = m+1$. Then \eqref{ineq.b} implies
$$2m + 4 \ge \beta_{01} + \beta_{02} + \beta_{03} \ge 3\beta_{02} = 3m + 3,$$
which implies $m \le 1$. According to  Notation~\ref{notation.Hirzebruch} (2), $m \ge \alpha$, so $m \ge 2$ and we get a contradiction. 

\smallskip 
\noindent \underline{\emph{Step 3.}} Case $\alpha=1$ and $m \geq 2$: 

\noindent We just saw in Step 2 that $\beta_{02}=m+1$ implies $m \le 1$, so \eqref{eq.irregular.splitting} implies $\beta_{02} =0$ or $1$ and $L_{02} = \mathcal O_W(2C_0)$
or $\mathcal O_W(2C_0+f)$. 
\smallskip

\noindent   \underline{\emph{Step 3.1.}} Case $L_{02} = \mathcal O_W(2C_0)$: 

\noindent If $L_{02} = \mathcal O_W(2C_0)$, then \eqref{ineq.a} and \eqref{ineq.b} imply 
$\alpha_{01}+\alpha_{03}=4$ and $\beta_{01}=\beta_{03}$. This and \eqref{eq.irregular.splitting} imply that $L_{01}=L_{03}=\mathcal O_W(2C_0)$ or $L_{01}=L_{03}=\mathcal O_W(2C_0+f)$.

\smallskip

\noindent   \underline{\emph{Step 3.2.}} The case  $L_{02} = \mathcal O_W(2C_0+f)$ cannot occur: 

\noindent Assume $L_{02} = \mathcal O_W(2C_0+f)$. 
As in Step 3.1, $\alpha_{01}+\alpha_{03}=4$,
but now $\beta_{01}+\beta_{03} \ge 2$. Then, either $\alpha_{01}=\alpha_{03}=2$ and then, by \eqref{eq.irregular.splitting} and $\beta_{01}+\beta_{03} \ge 2$,
$L_{01}=L_{03}=\mathcal O_W(2C_0+f)$, 
but this contradicts Proposition~\ref{prop.split.minimal.degree} (2) (b), that says  $n_1 \neq 3$;  or $\alpha_{01}=1$ and $\alpha_{03}=3$; or vice versa.  In the last two cases, \eqref{eq.irregular.splitting} and \eqref{ineq.b} imply either $m \le \beta_{01} - \beta_{03} \le 1$ or
$m \le \beta_{03} - \beta_{01} \le 1$, respectively, which contradicts the hypothesis of Step 3. 

\smallskip 
\noindent \underline{\emph{Step 4.}} Case $\alpha=1$ and $m =1$: 

\noindent By Step 1 and \eqref{eq.irregular.splitting}, and after swapping $C_0$ and $f$ if necessary, either $L_{02}=\mathcal O_W(2C_0)$ or  $\mathcal O_W(2C_0+f)$. 

\smallskip 

\noindent   \underline{\emph{Step 4.1.}} Case $L_{02} = \mathcal O_W(2C_0)$: 

\noindent
This case was already dealt with in Step 3.1 (note that the arguments used there do not use the hypothesis $m \ge 2$ of Step 3), 
Therefore $L_{01}=L_{03}=\mathcal O_W(2C_0)$ or $L_{01}=L_{03}=\mathcal O_W(2C_0+f)$ in this case.

\smallskip 

\noindent   \underline{\emph{Step 4.2.}} Case $L_{02} = \mathcal O_W(2C_0+f)$: 

\noindent We argue like in Step 3.2.  Then  either $\alpha_{01}=\alpha_{03}=2$, 
{so we get again a contradiction};  or $\alpha_{01}=1$ and $\alpha_{03}=3$; or vice versa. 

\noindent
If
$\alpha_{01}=1$ and $\alpha_{03}=3$, then $\beta_{03}=\beta_{01}-1$, so \eqref{eq.irregular.splitting} implies 
that $L_{01}=\mathcal O_W(C_0+2f)$ and $L_{03}=\mathcal O_W(3C_0+f)$. 

\noindent If $\alpha_{01}=3$ and $\alpha_{03}=1$, then $\beta_{03}=\beta_{01}+1$. Since there is an automorphism of $G^*$ that fixes $\chi_{02}$ and $\chi_{13}$ and swaps $\chi_{01}$ and $\chi_{03}$, after  relabeling the  $L_{a_1a_2}$ by such automorphism and relabeling the $D_{b_1b_2}$ by the automorphism of $G$ that is the inverse of its dual, we may assume $\alpha_{01}=1$, $\alpha_{03}=3$ and $\beta_{03}=\beta_{01}-1$, so again $L_{01}=\mathcal O_W(C_0+2f)$ and $L_{03}=\mathcal O_W(3C_0+f)$.  

\smallskip
\noindent In  all the cases of Steps 3.1 and 4, once we know the values of $L_{01}, L_{02}$ and $L_{03}$, the values of $L_{12}, L_{11}$ and $L_{10}$ are determined by Theorem~\ref{thm.Z2Z4.alg.str} (1) (ii) and \eqref{eq.alg.str.L_01L_12}, \eqref{eq.alg.str.L_02L_11}, 
\eqref{eq.alg.str.L_03L_10}. Likewise the divisors $D_{b_1b_2}$ are determined by Theorem~\ref{thm.Z2Z4.alg.str} (1) (ii) and \eqref{eq.divisors.Z2Z4}. This way, Step 3.1 gives rise to Cases 1 and 2 of (1) (iii) when $m \ge 2$; Step 4.1 gives rise to Cases 1 and 2 of (1) (iii) when $m =1$; and Step 4.2 gives rise to Case 3 of (1) (iii). 

 \smallskip
\noindent Theorem~\ref{thm.Z2Z4.alg.str} (1) (iii)  yields (1) (iv). Theorem~\ref{thm.Z2Z4.alg.str} (1) (iv) implies that $L_{a_1a_2}$ and $D_{b_1b_2}$ in (i), (ii) and (iii) satisfy \eqref{relation.ring.structure}. To prove (v) note that (iii) and (iv) imply that $D_{01}, D_{10}, D_{13}$ are smooth, $D_{01} \cap D_{13} = \emptyset$ and $D_{10}$ meets each $D_{01}, D_{13}$ transversally.  In addition, the pairs $H_{01}, H_{10}$ and $H_{10}, H_{13}$ satisfy the condition 
in \cite[Proposition 3.1 iii) a)]{Pardini}, so \cite[Proposition 3.1]{Pardini} implies $X$ is smooth. The computation $q(X)=3$ follows from computing $h^1(\pi_*\mathcal{O}_X)$, taking into account the values of the $L_{a_1a_2}$ of Case 1 of (iii). This finishes the proof of (v). 

\smallskip
\noindent The same computations as above give $q(X)=1$ in Cases 2 and 3 of (iii). To finish the proof of (vi), we argue similarly to the proof of (v). This time, in Case 2 of (iii), it follows from (iv) that $D_{01}$ do not meet $D_{13}$, that $D_{02}$ meets   each of $D_{01}$, $D_{13}$ in exactly $4$ points and such intersections are transverse. The pairs $H_{01}, H_{02}$ and $H_{02}, H_{13}$ do not satisfy the condition 
in \cite[Proposition 3.1 iii) a)]{Pardini}. Then the number and type of the singularities of $X$ of Case 2 of (iii) follow from \cite[Proposition 3.3]{Pardini}. The argument for the singularities of Case 3 of (iii) is the same. This completes the proof of (vi).  

\smallskip
\noindent Now we prove (2). If 
$D_{01}, D_{02}, D_{10}, D_{13}$ satisfy (1) (iv), then, as in the proof of Theorem~\ref{thm.P2.Z2Z4} (2),  Theorem~\ref{thm.Z2Z4.alg.str} implies the existence of an 
abelian canonical cover $\pi: X \longrightarrow W$ associated to $L_{a_1a_2}, D_{b_1b_2}$. In addition,  (i), (ii), (iii) and (iv) of (1)  imply (v) and (vi) of (1) for such $X$. 
\end{proof}

 \subsection{Covers with group $\mathbb{Z}_8$}

\begin{theorem}\label{thm.noZ8.covers}
There are no abelian canonical covers  $\pi: X \longrightarrow W$ with group $\mathbb{Z}_8$.
\end{theorem}

\begin{proof} 
By Theorem~\ref{thm.Z8.alg.str} (1) 
  we may  assume $L_{7} = \omega_{W}^*(1)$ without any loss of generality. By \eqref{eq.alg.str.Z8.L_1L_2} and Theorem~\ref{thm.Z8.alg.str} (1) (iii), we have 
  \begin{equation}\label{eq.alg.str.Z8.for.contradiction}
   L_{1} \otimes L_{2} = L_{3}.    
  \end{equation}
  We split the proof into several cases.
\smallskip

\noindent \underline{\emph{Case 1}}: $W=\mathbb P^2$.

\smallskip
\noindent   If $W=\mathbb{P}^2$, by Proposition~\ref{prop.split.minimal.degree} (1), then $i(W)$ is linear $\mathbb{P}^2$. Then  $L_{7} = \omega_{W}^*(1)=\mathcal O_W(4)$, so Proposition~\ref{prop.split.minimal.degree} (1) implies 
  $L_1=L_2=L_3=\mathcal{O}_{W}(2)$ and this contradicts \eqref{eq.alg.str.Z8.for.contradiction}.  

  \smallskip

\noindent \underline{\emph{Case 2}}: $W=\mathbb F_e$ and $X$ is regular.
Since  $L_{7} = \omega_{W}^*(1)$, \eqref{eq.regular.splitting},  \eqref{eq.regular.splitting.a>1} and \eqref{eq.alg.str.Z8.for.contradiction} imply $\alpha=1$,  $L_1=L_2=\mathcal{O}_{W}(C_{0} +(m+1)f)$ and 
$L_3=\mathcal{O}_{W}(2C_{0} +(e+1)f)$. Then $2m+2 =e+1$. Since $m \geq e+1$, then $m \le -2$, which is a contradiction. 

\smallskip

\noindent \underline{\emph{Case 3}}: $X$ is irregular. Since  $L_{7} = \omega_{W}^*(1)$ and $\alpha$ is positive, \eqref{eq.irregular.splitting} and \eqref{eq.alg.str.Z8.for.contradiction} imply
that, if the coefficient of $C_0$ in the linear equivalence class of $L_1$ or $L_2$ is $\alpha+1$ or $\alpha+2$, then the coefficient of $C_0$ in the linear equivalence class of $L_3$ is at least $\alpha+1$ and  the coefficient of $f$ in the linear equivalence class of $L_3$ is at most $1$. Then \eqref{eq.irregular.splitting}, \eqref{eq.alg.str.Z8.for.contradiction}  and $\alpha$ being positive imply the coefficient of $C_0$ in the linear equivalence class of $L_1 \otimes L_2$ is at most $\alpha+2$, so the coefficient of $f$ in the linear equivalence class of $L_1 \otimes L_2$ is at least $m+1$. Since, by \eqref{eq.alg.str.Z8.for.contradiction}, the coefficients of $f$ in  the linear equivalence classes of $L_1 \otimes L_2$ and $L_3$ are equal, we have $m \leq 0$, which is a contradiction. 

\smallskip
\noindent If  the coefficient of $C_0$ in the linear equivalence class of $L_1$ or $L_2$ is $0$ or $1$, then the coefficient of $f$ in the linear equivalence class of $L_1$ or $L_2$ is $m+1$ or $m+2$, so an argument as the above, interchanging the roles of $C_0$ and $f$, yields again a contradiction. 
\end{proof}

\section*{Acknowledgments} We are most grateful to Purnaprajna Bangere and Roberto Pignatelli for sharing their time with us for very helpful discussions.


\end{document}